\documentclass[11pt]{article}
\usepackage[utf8x]{inputenc}
\usepackage[margin=1.5in]{geometry}
\usepackage{microtype}
\usepackage[T1]{fontenc}
\usepackage{ae,aecompl}
\usepackage{times}
\usepackage{tikz-cd}
\usepackage{mathrsfs} 
\usepackage{harmony}
\usepackage{verbatim,amsmath,amsthm,amsfonts,amssymb,latexsym,graphicx,mathtools,extpfeil,color,mathabx}
\usepackage{tikz}
\usepackage{epstopdf,pinlabel}
\usepackage[all]{xy}
\usepackage{graphicx}
\usepackage{booktabs}
\usepackage{caption}
\usepackage{tabularx}
\usepackage{subcaption}
\usepackage{slashed}
\DeclareMathAlphabet{\mathpzc}{OT1}{pzc}{m}{it}
\usepackage[colorlinks,pagebackref,hypertexnames=false]{hyperref} 
\usepackage[alphabetic,backrefs,msc-links]{amsrefs}
\DeclareUnicodeCharacter{2212}{\ensuremath{-}}
\usepackage{aliascnt}
\numberwithin{equation}{section}

\newcommand{\bC}{{\bf C}}

\newcommand{\bF}{{\bf F}}

\newcommand{\bP}{{\bf P}}
\newcommand{\bQ}{{\bf Q}}
\newcommand{\bR}{{\bf R}}

\newcommand{\bZ}{{\bf Z}}

\newcommand{\cC}{\mathcal{C}}

\newcommand{\cS}{\mathcal{S}}
\newcommand{\cT}{\mathcal{T}}

\newcommand{\sP}{\mathscr{P}}

\newcommand{\sS}{\mathscr{S}}

\newcommand{\fb}{{\mathfrak b}}

\newcommand{\ft}{{\mathfrak t}}

\newcommand{\fC}{{\mathfrak C}}
\newcommand{\fD}{{\mathfrak D}}

\newcommand{\fK}{{\mathfrak K}}

\newcommand{\Z}{\bZ}
\newcommand{\Q}{\bQ}
\newcommand{\R}{\bR}
\newcommand{\C}{\bC}

\newcommand{\su}{\mathfrak{su}}

\DeclareMathOperator{\ad}{ad}

\DeclareMathOperator{\rk}{rk}

\DeclareMathOperator{\ind}{ind}

\renewcommand{\det}{\operatorname{det}}

\renewcommand{\P}{\bP}
\renewcommand{\epsilon}{\varepsilon}

\def\({\mathopen{}\left(}
\def\){\right)\mathclose{}}
\def\<{\mathopen{}\left<}
\def\>{\right>\mathclose{}}

\usepackage{multicol, color}

\definecolor{gold}{rgb}{0.85,.66,0}
\definecolor{cherry}{rgb}{0.9,.1,.2}
\definecolor{burgundy}{rgb}{0.8,.2,.2}
\definecolor{orangered}{rgb}{0.85,.3,0}
\definecolor{orange}{rgb}{0.85,.4,0}
\definecolor{olive}{rgb}{.45,.4,0}
\definecolor{lime}{rgb}{.6,.9,0}
\definecolor{green}{rgb}{.2,.7,0}
\definecolor{grey}{rgb}{.4,.4,.2}
\definecolor{brown}{rgb}{.4,.3,.1}

\def\makeautorefname#1#2{\AtBeginDocument{\expandafter\def\csname#1autorefname\endcsname{#2}}}

\newcommand{\mynewtheorem}[2]{
  \newaliascnt{#1}{equation}          
  \newtheorem{#1}[#1]{#2}
  \aliascntresetthe{#1}
  \makeautorefname{#1}{#2}
}
\mynewtheorem{theorem}{Theorem}
\mynewtheorem{prop}{Proposition}
\mynewtheorem{cor}{Corollary}
\mynewtheorem{construction}{Construction}
\mynewtheorem{lemma}{Lemma}
\mynewtheorem{conjecture}{Conjecture}

\numberwithin{substep}{step}
\makeautorefname{step}{Step}
\makeautorefname{substep}{Step}

\numberwithin{subcase}{case}
\makeautorefname{case}{Case}
\makeautorefname{subcase}{case}

\theoremstyle{remark}
\mynewtheorem{remark}{Remark}

\theoremstyle{definition}
\mynewtheorem{definition}{Definition}
\mynewtheorem{example}{Example}
\mynewtheorem{exercise}{Exercise}
\mynewtheorem{convention}{Convention}
\newtheorem*{convention*}{Convention}
\newtheorem*{conventions*}{Conventions}
\mynewtheorem{question}{Question}
        
\makeautorefname{chapter}{Chapter}
\makeautorefname{section}{Section}
\makeautorefname{subsection}{Section}
\makeautorefname{subsubsection}{Section}

\theoremstyle{introthm}
\newtheorem{introthm}{Theorem}
\theoremstyle{introcor}
\newtheorem{introcor}{Corollary}
\theoremstyle{introprop}
\newtheorem{introprop}{Proposition}
\theoremstyle{introquestion}
\newtheorem{introquestion}{Question}

\newcommand\hrI{{\widehat { I}}}

\newcommand\crI{{\widecheck { I}}}

\newcommand\brI{{\overline { I}}}

\newcommand\diamd{\blacklozenge}
\newcommand\cSO{\mathcal {SO}}

\title{Chern-Simons functional, singular instantons, and the four-dimensional clasp number}
\author{Aliakbar Daemi\thanks{The work of AD was supported by NSF Grant DMS-1812033 and NSF FRG Grant DMS-1952762.} \hspace{1cm} Christopher Scaduto\thanks{The work of CS was supported by NSF Grant DMS-1952762.}}
\date{}

\newcommand{\Addresses}{{
  \bigskip
  \footnotesize
  Aliakbar Daemi, \textsc{Department of Mathematics, Washington University in St. Louis, One Brookings drive, Room 207A,
  St. Louis, MO 63130}\par\nopagebreak
  \textit{E-mail address}: \texttt{adaemi@wustl.edu}
  \vspace{.2cm}

Christopher Scaduto, \textsc{Department of Mathematics, University of Miami, 1365 Memorial Dr 515, Coral Gables, FL 33124}\par\nopagebreak
  \textit{E-mail address}: \texttt{cscaduto@miami.edu}
}}

\begin{document}
\maketitle

\begin{abstract}
Kronheimer and Mrowka asked whether the difference between the four-dimensional clasp number and the slice genus can be arbitrarily large. This question is answered affirmatively by studying a knot invariant derived from equivariant singular instanton theory, and which is closely related to the Chern--Simons functional. This also answers a conjecture of Livingston about slicing numbers. Also studied is the singular instanton Fr\o yshov invariant of a knot. If defined with integer coefficients, this gives a lower bound for the unoriented slice genus, and is computed for quasi-alternating and torus knots. In contrast, for certain other coefficient rings, the invariant is identified with a multiple of the knot signature. This result is used to address a conjecture by Poudel and Saveliev about traceless $SU(2)$ representations of torus knots. Further, for a concordance between knots with non-zero signature, it is shown that there is a traceless representation of the concordance complement which restricts to non-trivial representations of the knot groups. Finally, some evidence towards an extension of the slice-ribbon conjecture to torus knots is provided.
\end{abstract}


\hypersetup{linkcolor=black}
\tableofcontents



\section{Introduction}

In previous work \cite{DS}, the authors introduced a framework for studying $S^1$-equivariant instanton Floer theoretic invariants for a knot $K$ in an integer homology 3-sphere $Y$. Morally the constructions are derived from Morse theory for the Chern--Simons functional on the space of singular $SU(2)$ connections framed at a basepoint on $K$. This space has an $S^1$-action, and the critical set is acted on freely except for a unique isolated fixed point. This critical set is naturally in bijection with {\it traceless} representations of $\pi_1(Y\setminus K)$ into $SU(2)$, i.e., $SU(2)$-representations of the knot group which send a distinguished meridian near the basepoint of $K$ to $\mathbf{i}\in SU(2)$.

The details of the constructions rely on the extensive work of Kronheimer and Mrowka, who developed the foundations of singular instanton Floer theory \cite{KM:YAFT, KM:unknot}. We show in \cite{DS} that many of the invariants defined by Kronheimer and Mrowka can be recovered from the equivariant framework.

In this paper we continue our study of the theory initiated in \cite{DS}, with an emphasis on topological applications. We carry out computations for several families of 2-bridge knots and torus knots. Using the structure of the Chern--Simons filtration and computations of singular Fr\o yshov invariants, we obtain information about the four-dimensional clasp number of knots, fundamental groups of concordance complements, and more.

\subsection*{The four-dimensional clasp number of a knot}

A {\em normally immersed surface} is a smoothly immersed surface in a manifold which has transverse double points. The {\em $4$-dimensional clasp number} $c_s(K)$ of a knot $K$ is the minimal number of double points realized by a normally immersed disk in $B^4$ with boundary $K$ \cite{shibuya}. In the literature $c_s(K)$ is also called the {\em 4-ball crossing number} \cite{owens-strle}. Variations of $c_s$ can be defined by keeping track of the signs of double points. Let $c_s^+(K)$ (resp. $c_s^-(K)$) denote the minimal number of {\em positive} (resp. {\em negative}) double points realized by a normally immersed disk in $B^4$ with boundary $K$. Then $ c_s(K)\geqslant c_s^+(K)+c_s^-(K)$. Recall that two knots $K$ and $K'$ in the 3-sphere are {\em concordant} if they cobound a properly smoothly embedded annulus in $[0,1]\times S^3$. Clearly $c_s(K)$ and $c_s^\pm(K)$ are concordance invariants. 

A standard construction allows us to remove a double point of a normally immersed surface at the expense of increasing the genus. We thus have
\begin{equation*}
	c_s(K) \geqslant g_s(K) \label{eq:claspgenus}
\end{equation*}
where $g_s(K)$ is the slice genus of $K$, the minimal genus of a smoothly embedded oriented surface in the 4-ball which has boundary $K$. 
Kronheimer and Mrowka asked \cite{km-barnatan-v1,kronheimer-talk}:
\begin{introquestion}\label{KM-question}
	Can the difference $c_s(K)-g_s(K)$ be arbitrarily large?
\end{introquestion}

A more refined question is the following which is already implicit in \cite{km-concordance}.
\begin{introquestion}\label{KM-question-ref}
	Can the difference $c_s^+(K)-g_s(K)$ be arbitrarily large?
\end{introquestion}

Part of the subtlety of Question \ref{KM-question-ref} is that many of the known knot invariants which provide lower bounds for $g_s(K)$ do not give better bounds for $c_s^+(K)$. For example, the signature $\sigma$ of a knot $K$ satisfies
\begin{equation}\label{ineq-signature}
	\vert \frac{1}{2}\sigma(K)\vert \leq g_s(K),\hspace{1cm} - \frac{1}{2}\sigma(K)\leq c_s^+(K),\hspace{1cm} \frac{1}{2}\sigma(K)\leq c_s^-(K).
\end{equation}
Similar inequalities hold for Levine-Tristram signatures $\sigma_w(K)$, Heegaard Floer invariants $\tau(K)$ and $\Upsilon(K)$ \cite{Oz-Sz:tau,Ra:HFK,OSS:upsilon,JZ:clasp} and Rasmussen's invariant $s(K)$ \cite{Ras:s-inv,MMSW:s-exotic}. In \cite{DS}, the authors defined various concordance invariants. In this paper, we study the behavior of these invariants with respect to knot cobordisms and use them to prove the following theorem.
\begin{introthm}\label{thm:clasp}
	Let $K_1$ be the knot $7_4$, which is the $(15,4)$ 2-bridge knot, and let $K_n$ be the $n$-fold connected sum of $K_1$. Then we have:
	\[c_s^+(K_n)-g_s(K_n)\geqslant n/5.\]
\end{introthm}

As a corollary of this theorem, one obtains a positive answer to Question \ref{KM-question} by considering $n$-fold connected sums of $7_4$. Although $\sigma_w(K)$, $\tau(K)$, $\Upsilon(K)$ and $s(K)$ cannot be used to answer Question \ref{KM-question-ref} directly, they can be useful to answer Question \ref{KM-question}. The idea is to use the analogue of the second inequality in \eqref{ineq-signature} for one of these invariants to provide a lower bound for $c_s^+(K)$ and the analogue of the third inequality for another invariant to give a lower bound for $c_s^-(K)$. Then combining these lower bounds gives rise to a lower bound for $c_s(K)$ which potentially could be as large as $2g_s(K)$. In a recent article \cite{JZ:clasp}, Juh\'asz and Zemke answered Question \ref{KM-question} independently by applying this strategy using $\Upsilon(K)$ where $K$ is given by connected sums of certain torus knots. More recently, Feller and Park showed that Levine-Tristram signatures can be also used to answer Question \ref{KM-question} in the case that $K$ is again given by connected sums of torus knots \cite{PP:clasp}.

Theorem \ref{thm:clasp} also addresses a conjecture made in \cite{Liv:slice-num}. The {\emph{slicing number}} of a knot $K$ is the minimum number of crossing changes required to convert $K$ into a slice knot. Any sequence of $i$ crossing changes, which slices $K$, can be used to form an immersed disk in $B^4$ with boundary $K$ and $i$ double points. Thus the slicing number of $K$ is at least as large as $c_s(K)$. A refinement keeps track of positive and negative crossing changes. Let $I\subset \Z_{\geq 0}^2$ be the set of pairs $(m,n)$ such that a collection of $m$ positive crossing changes and $n$ negative crossing changes converts $K$ into a slice knot. Livingston defines
\[
  U_s(K):=\min_{(m,n)\in I}\max(m,n).
\]
Then $U_s$ is again bounded below by $g_s$ and Livingston asked if $U_s-g_s$ can be arbitrarily large \cite[Conjecture 5.4]{Liv:slice-num}. Since $c_s^+(K)\leq U_s(K)$, we have the following corollary.
\begin{introcor}
	The quantity $U_s(K_n)-g_s(K_n)$ is at least $n/5$, and hence the difference $U_s-g_s$ can be arbitrarily large.
\end{introcor}

It is easy to construct a genus 1 slicing surface for $K_1$ (see Figure \ref{fig:dmn}) and thus a genus $n$ slicing surface for $K_n$. Murasugi's inequality $g_s(K)\geqslant |\frac{1}{2}\sigma(K)|$ and the computation  $\sigma(K_n)/2=-n$ imply $g_s(K_n)=n$. So the non-trivial part of Theorem \ref{thm:clasp} is to provide the lower bound $6n/5$ for $c_s^+(K_n)$. The main result of \cite{Liv:slice-num} states that the slicing number of $K_1$ is $2$ and $c_s^+(K_1)=2$ is computed in \cite{owens-strle}. It is suggested in \cite{km-barnatan-v1} that the knots $K_n$ may give a family for which $c_s-g_s$ becomes arbitrarily large. The tools developed in \cite{km-concordance} by Kronheimer and Mrowka could potentially detect a sequence of knots with $c_s-g_s$ being arbitrarily large. However, the particular invariants of that article do not reproduce the above result for $n7_4$.

We outline the proof of Theorem \ref{thm:clasp}. The main tool is the concordance invariant $\Gamma_K$ for a knot $K$, constructed in \cite{DS}. It is derived from the structure of the Chern--Simons filtration on the equivariant singular instanton complex for $K$, and is an analogue of the invariant $\Gamma_Y$ for homology 3-spheres defined in \cite{AD:CS-Th}. The invariant $\Gamma_K$ is a function from the integers with values in $\R_{\geqslant 0}\cup \infty$. We will show:

\begin{introthm}\label{thm:gammaintro} For $K\subset S^3$, if $-\sigma(K)/2\geqslant 0$, then we have the inequality:
\[
	\Gamma_K\left(-\frac{1}{2}\sigma(K)\right) \leqslant \frac{1}{2} c_s^+(K).
\]
\end{introthm}

\begin{proof}[Proof of Theorem \ref{thm:clasp}]
Let $K_n=n7_4$. In Proposition \ref{double-twist-Gamma} we compute:
\[
		\Gamma_{K_n}(i) = \begin{cases} 0, & i \leqslant 0\\ \frac{3}{5} i, & 1\leq i \leq n\\
		 \infty, & i\geqslant n+1 \end{cases}
\]
Theorem \ref{thm:gammaintro} implies $\Gamma_{K_n}(n)=3n/5 \leqslant c_s^+(K_n)/2$, and so $c_s^+(K_n)\geqslant 6 n /5$.  
\end{proof}

 We can prove similar results for other families of knots, such as connected sums of double twist knots $D_{m,n}$. The double twist knot $D_{m,n}$ is given in Figure \ref{fig:dmn}. It is a 2-bridge knot with $g_s(D_{m,n})=1=-\sigma(D_{m,n})/2$ and in fact $D_{2,2}=7_4$.

\begin{introthm}\label{thm:doubletwist}
	For $m,n\in \Z_{>0}$ let $D_{m,n}$ be the double twist knot. For $k\in \Z_{>0}$ we have:
\[
	c_s^+(kD_{m,n})-g_s(kD_{m,n}) \geqslant \frac{k(4mn-4m-4n+3)}{4mn-1}.
\]
Thus if $m,n\geqslant 2$ then $c_s^+(kD_{m,n})-g_s(kD_{m,n})\geqslant Ck$ for some $C>0$.
\end{introthm}

Recall that the {\em unknotting number} $u(K)$ is the minimal number of crossing changes needed to turn $K$ into the unknot $U_1$. Performing these crossing changes in reverse order, to get from $U_1$ to $K$, induces a movie in $S^3\times [0,1]$, which is an immersed annulus with transverse double points. Capping off the unknot component of our annulus gives a normally immersed disk in $B^4$ with boundary $K$, with $u(K)$ double points. Thus $u(K)\geq c_s(K)\geq c_s^+(K)$, and Theorem \ref{thm:gammaintro} also gives a lower bound on $u(K)$. We can use this to determine the unknotting numbers of some 11-crossing 2-bridge knots.

\begin{introthm}\label{thm:unknottingnumbers}
	The knots $11a_{192}$, $11a_{341}$, $11a_{360}$ each have unknotting number equal to $3$. The unknotting number of $11a_{365}$ is equal to $4$. In fact, for each of these knots, $c_s^+$ and 
	the 4-dimensional clasp number are equal to the unknotting number.
\end{introthm}

\noindent The unknotting number of $11a_{365}$ was previously computed by Owens \cite{owens}, and its 4-dimensional clasp number follows from the work of Owens and Strle \cite{owens-strle}. The other three examples in the theorem are new, as far as the authors are aware.

\begin{remark} 
	Theorem \ref{thm:clasp} addresses a relative version of the following problem. Let $W$ be a smooth closed $4$-manifold and $\sigma\in H_2(W;\Z)$. Let $g(\sigma)$ be the minimal genus over all connected, orientable smoothly 
	embedded representatives of $\sigma$, and $c^+(\sigma)$ be the minimal number of positive double points over normally immersed sphere representatives. How large can $c^+(\sigma)-g(\sigma)$ be? This question has been 
	advertised by Kronheimer and Mrowka, partly motivated by its relationship to the Caporaso-Harris-Mazur conjecture about algebraic curves in a quintic surface \cite{kronheimer-talk}. $\diamd$
\end{remark}

\subsection*{Signatures, Fr\o yshov invariants, and representations}

For a knot $K$ in the 3-sphere, the knot group $\pi_1(S^3\setminus K)$ is infinite cyclic if and only if $K$ is an unknot, as follows from the work of Papakyriakopoulos \cite{papa}. Using instanton gauge theory, Kronheimer and Mrowka showed that in fact for a non-trivial knot there exists a non-abelian representation from $\pi_1(S^3\setminus K)$ into $SU(2)$ \cite{km-dehn}.

There are analogous problems in 4-dimensions. Given a 2-knot, i.e. a smoothly embedded 2-sphere $F$ in $S^4$, if $\pi_1(S^4\setminus F)$ is infinite cyclic, is $F$ an unknotted 2-sphere? If we instead work in the topologically locally flat category, then Freedman's work \cite{freedman} answers this in the affirmative. In either category, we may further ask: given a non-trivial 2-knot $F$, does $\pi_1(S^4\setminus F)$ admit a non-abelian representation to $SU(2)$? 

We may also pose analogous questions about embedded annuli in $[0,1]\times S^3$. The following gives a simple criterion for the existence of non-abelian $SU(2)$ representations in this scenario under an assumption involving the knot signature.

\begin{introthm}\label{thm:existrep1}
	If $ S \subset [0,1]\times S^3$ is a smooth concordance from a knot $K$ to itself, and $\sigma(K)\neq 0$, then $\pi_1([0,1]\times S^3\setminus S)$ admits a non-abelian representation to $SU(2)$.
\end{introthm}

The above generalizes as follows. A {\em homology concordance} is a cobordism of pairs $(W,S):(Y,K)\to (Y',K')$ between integer homology 3-spheres with knots such that $W$ is an integer homology cobordism and $S$ is a properly embedded annulus. The collection of integer homology 3-spheres with knots modulo homology concordance gives rise to an abelian group $\cC_\Z$. There is a homomorphism from the smooth concordance group to $\cC_\Z$. There is also a homomorphism $\Theta_\Z^3\to \cC_\Z$ where $\Theta_\Z^3$ is the homology cobordism group of integer homology 3-spheres, induced by $Y\mapsto (Y,\text{unknot})$. Write $h(Y)\in \Z$ for the instanton Fr\o yshov invariant for the integer homology 3-sphere $Y$, which is a homomorphism $h:\Theta_\Z^3\to\Z$ defined in  \cite{froyshov}.

\begin{introthm}\label{thm:existrep2}
	Suppose $(W,S):(Y,K)\to (Y',K')$ is a homology concordance where $K$ is null-homotopic in the integer homology 3-sphere $Y$. Suppose further that:
\begin{equation}
	-\frac{1}{2}\sigma(Y,K) + 4h(Y) \neq 0. \label{eq:existrepcondition}
\end{equation}
Then there exists a traceless representation $\pi_1(W\setminus S)\to SU(2)$ which extends non-abelian traceless representations of $\pi_1(Y\setminus K)$ and $\pi_1(Y'\setminus K')$.
\end{introthm}

Here ``traceless'' means that the representation sends classes of the circle fibers of the unit circle normal bundle of $S$ (resp. $K$, $K'$) to traceless elements of $SU(2)$. Theorem \ref{thm:existrep2} should be compared to analogous results in the context of homology cobordism of homology 3-spheres (without knots), see \cite{taubes-periodic,AD:CS-Th}.

The authors established in \cite[Theorem 1.17]{DS} a version of Theorem \ref{thm:existrep2} with the left hand side of \eqref{eq:existrepcondition} replaced by $h_\sS(Y,K)$, the singular instanton Fr\o yshov invariant. This invariant, also introduced in \cite{DS}, is a homomorphism $h_\sS(Y,K):\cC_\Z\to \Z$, and is defined for a coefficient ring $\sS$ which is an integral domain algebra over $\Z[U^{\pm 1}, T^{\pm 1} ]$. Thus Theorem \ref{thm:existrep2} follows from our previous work and the following result, which identifies $h_\sS(Y,K)$ with previously known invariants under some assumptions. We write $h_\sS(S^3,K)=h_\sS(K)$.

\begin{introthm}\label{thm:htwistedsign}
	Let $\sS=\Z[U^{\pm 1}, T^{\pm 1}]$. If $K$ is a knot in the 3-sphere, then we have:
\begin{equation}
	h_\sS(K)=-\frac{1}{2}\sigma(K). \label{eq:htwistedsign3sphere}
\end{equation}
If $K$ is a null-homotopic knot in an integer homology 3-sphere $Y$, then:
\begin{equation}
	h_\sS(Y,K) = -\frac{1}{2}\sigma(Y,K) + 4h(Y). \label{eq:htwistedsign}
\end{equation}
The result \eqref{eq:htwistedsign3sphere} holds more generally when $\sS$ is an integral domain algebra over $\Z[U^{\pm 1}, T^{\pm 1}]$ with $T^4\neq 1$. Further, \eqref{eq:htwistedsign} holds more generally as well; see Section \ref{sec:froyshovandsuspension}.

\end{introthm}

The identification of $h_\sS(K)$ in terms of the classical knot signature helps elucidate the general structure of the singular instanton Floer invariants of $K$, as we now explain. Let $I_\ast(Y,K)$ denote the irreducible singular instanton homology of $(Y,K)$, defined in \cite{DS}. It is a $\Z/4$-graded abelian group and its euler characteristic satisfies
\[
	\chi\left(I_\ast(Y,K)\right) =  \frac{1}{2}\sigma(Y,K)+ 4\lambda(Y)
\]
where $\lambda(Y)$ is the Casson invariant of $Y$, a result essentially due to Herald \cite{herald}. More generally, we can define an $\sS$-module $I_\ast(Y,K;\Delta_\sS)$, where as before $\sS$ is an algebra over $\Z[U^{\pm 1}, T^{\pm 1}]$, and $\Delta_\sS$ is a local coefficient system over $\sS$.

\begin{introcor}\label{eq:corirrineq}
Suppose $T^4\neq 1$ in $\sS$. If $K\subset S^3$ and $\sigma(K)\leqslant 0$ then we have inequalities:
\begin{equation}\label{eq:irrhomologyineq}
  \rk(I_1(K;\Delta_\sS))\geq \left\lceil -\frac{\sigma(K)}{4}\right\rceil,\hspace{.5cm}
  \rk(I_3(K;\Delta_\sS))\geq \left\lfloor  -\frac{\sigma(K)}{4} \right\rfloor.
\end{equation}
\end{introcor}

\noindent Inequalities in the case $\sigma(K) > 0$ are obtained using $\text{rk}(I_i(-K;\Delta_\sS))=\text{rk}(I_{3-i}(K;\Delta_\sS))$ and $\sigma(-K)=-\sigma(K)$. Corollary \ref{eq:corirrineq} follows from Theorem \ref{thm:htwistedsign} and elementary inequalities relating the irreducible instanton homology to the Fr\o yshov invariant, see \eqref{h-twisted-ineq-irr-I}. See Theorem \ref{thm:scomplexsign} for a stronger version of Corollary \ref{eq:corirrineq}.

The following was essentially conjectured in \cite[Section 7.4]{PS}.

\begin{introcor}\label{cor:torusknots}
Let $T_{p,q}$ be the $(p,q)$ torus knot. Then as a $\Z/4$-graded abelian group, 
\begin{equation}
	I_\ast(T_{p,q})\cong  \Z_{(1)}^{\lceil -\sigma(T_{p,q})/4 \rceil}\oplus \Z_{(3)}^{\lfloor -\sigma(T_{p,q})/4 \rfloor}.
\label{eq:irrhomtorusknots}
\end{equation}
\end{introcor}

\begin{proof}
	The result in fact holds for any coefficient ring $\sS$. From \cite{PS} and \cite[Section 9.5]{DS}, the chain complex for the irreducible homology of $T_{p,q}$ has total rank $-\sigma(T_{p,q})/2$ and is supported in gradings $1$ and $3$ (mod $4$). The differential is necessarily zero. The lower bounds from Corollary \ref{eq:corirrineq} imply the result when $T^4\neq 1$ in $\sS$, and the ranks of the $\Z/4$-graded chain complex are independent of the ring $\sS$.
\end{proof}

Another strategy to prove Corollary \ref{cor:torusknots} involves directly computing the mod $4$ gradings of the singular flat connections.  Anvari computed these mod $4$ gradings in terms of certain arithmetic functions and verified Corollary \ref{cor:torusknots} for the $(3,6k+1)$ torus knots \cite{anvari}. Corollary \ref{cor:torusknots} has implications about the arithmetic functions appearing in \cite{anvari}.

The special structure of the instanton homology for torus knots described above will be used to upgrade Theorem \ref{thm:existrep2} in the case that one of the knots is a torus knot. For simplicity we state the result for classical concordances in $[0,1]\times S^3$.

\begin{introthm}\label{torus-conc-reps}
	Suppose $S:T_{p,q} \to K$ is a smooth concordance. Then every traceless $SU(2)$ representation of $\pi_1(S^3\setminus T_{p,q})$ extends over the concordance complement.
\end{introthm}

\noindent A consequence of this result is that if a knot $K$ is concordant to a torus knot $T_{p,q}$ then the (singular) Chern--Simons invariants of $T_{p,q}$ are a subset of those of $K$. We also have the following relation at the level of Floer homology.

\begin{introthm}\label{torus-conc-floer-homology}
	Suppose $K$ is a knot which is concordant to a torus knot $T_{p,q}$. Then $I_\ast(T_{p,q})$ is a module summand of $I_\ast(K)$ for any choice of ordinary or local coefficient system. The same holds for other versions of singular instanton homology, such as $I^\natural_\ast(K)$ and $\widehat I_\ast(K)$.
\end{introthm}

Such results for instanton Floer homology and character varieties also hold when two knots are related by a ribbon concordance, see e.g. \cite{dlvw}, \cite[Theorem 7.4]{km-concordance}. Thus Theorems \ref{torus-conc-reps} and \ref{torus-conc-floer-homology} provide some evidence towards the following question due to Gordon, which is closely related to the slice-ribbon conjecture.
\begin{introquestion}\label{torus-slice-ribbon}
	If there is a concordance from a torus knot $T_{p,q}$ to another knot $K$, is there then a ribbon concordance from $T_{p,q}$ to $K$?
\end{introquestion}
\noindent A ribbon concordance from a knot $K_0$ to a knot $K_1$ is a properly embedded surface $S$ in $[0,1]\times S^3$ such that $S\cap \{i\}\times S^3=K_i$ and the projection of $S$ to $[0,1]$ is a Morse function without local maxima. (In particular, $S$ is a ribbon concordance from $K_1$ to $K_0$ in the convention of \cite{Gor:Rib}.) In fact, a more general version of Question \ref{torus-slice-ribbon} is proposed by Gordon. 
Now, let a knot $\fK$ be minimal with respect to the {\it ribbon concordance relation}, i.e., there is no ribbon concordance from a different knot to $\fK$. Then Question 6.1 in \cite{Gor:Rib} asks whether for any such $\fK$ and any other knot $K$ concordant to $\fK$, there is a ribbon concordance from $\fK$ to $K$. For $\fK=U_1$ this question specializes to the slice-ribbon conjecture. It is shown in \cite{Gor:Rib} that torus knots are minimal with respect to the ribbon concordance relation and Question \ref{torus-slice-ribbon} is the special case of \cite[Question 6.1]{Gor:Rib} for torus knots.

%

Turning back to the situation of more general knots in $S^3$, we now describe an algebraic refinement of Theorem \ref{thm:htwistedsign}.
The equivariant singular instanton theory developed in \cite{DS} associates to a knot $K$ a $\Z/4$-graded chain complex $\widetilde C_\ast(K;\Delta_\sS)$ over $\sS$. This complex has the additional structure of an {\em $\cS$-complex}, which formalizes the homological properties of the $S^1$-action on the framed configuration space. The homotopy type of the $\cS$-complex $\widetilde C_\ast(K;\Delta_\sS)$ is an invariant of $K$, and determines the invariant $h_\sS(K)$. Thus the following result is an upgrade of the computation $h_\sS(K)=-\sigma(K)/2$ from Theorem \ref{thm:htwistedsign}.

\begin{introthm}\label{thm:scomplexsign}
	If $T^4-1$ is invertible in $\sS$, then the $\cS$-complex $\widetilde C_\ast(K;\Delta_\sS)$ is homotopy equivalent to $\widetilde C_\ast(\varepsilon T_{2,2k+1};\Delta_\sS)$, where $k=|\sigma(K)|/2$ and $\varepsilon$ is the sign of $-\sigma(K)$. As a consequence, in the inequalities \eqref{eq:irrhomologyineq} of Corollary \ref{eq:corirrineq}, equality holds.
\end{introthm}

\noindent Furthermore, the $\cS$-complex for $T_{2,2k+1}$ is very simple. We prove a similar statement for any null-homotopic knot in an integer homology 3-sphere. See Theorem \ref{knot-inv-twisted}. 

Theorem \ref{thm:scomplexsign} shows that if our coefficient ring $\sS$ has $T^4\neq 1$, then the ranks of the irreducible instanton homology groups are determined by the signature of the knot in the same way as for torus knots in \eqref{eq:irrhomtorusknots}. In the special case of a 2-bridge knot, it was shown in \cite[Theorem 1.12]{DS} that the irreducible instanton homology for a certain local coefficient system with $T^4\neq 1$ determines a version of lens space instanton homology, denoted $I_\ast(L(p,q))$, defined by Furuta \cite{furuta-invariant} and Sasahira \cite{sasahira-lens}. The latter is a $\Z/4$-graded vector space over $\bF$, the field with 2 elements. In \cite{sasahira-lens}, Sasahira computed these groups for the family $L(8N+1,2)$. Combining \cite[Theorem 1.12]{DS} with Theorem \ref{thm:scomplexsign} we obtain:

\begin{introcor}\label{cor:sasahira}
	Suppose $\sigma(K_{p,q})\leqslant 0$ where $K_{p,q}$ is the 2-bridge knot with double branched cover $L(p,-q)$. Then the lens space instanton homology of Furuta and Sasahira is given by
	\[
		I_\ast(L(p,-q)) \cong \bF_{(1)}^{\lceil -\sigma(K_{p,q})/4 \rceil}\oplus \bF_{(3)}^{\lfloor -\sigma(K_{p,q})/4 \rfloor}
	\]
as a $\Z/4$-graded $\bF$-vector space. If $\sigma(K_{p,q})>0$, use $I_\ast(L(p,-q))\cong I_{3-\ast}(L(p,q))$. 
\end{introcor}

In \cite{sasahira-lens}, Sasahira proved a gluing result for 2-torsion instanton invariants of 4-manifolds which decompose along lens spaces. Using this result and knowledge of $I_\ast(L(p,q))$, he studied the question of whether $\bC\bP^2\#\bC\bP^2$ can be decomposed along $L(p,q)$.

\subsection*{The singular Fr\o yshov invariant and the unoriented $4$-ball genus}

The result of Theorem \ref{thm:htwistedsign} shows that the singular instanton Fr\o yshov invariant for a knot in the 3-sphere is determined by the signature, so long as we work with a coefficient ring $\sS$ for which $T^4\neq 1$. When $T^4=1$ the situation is rather different, as already evidenced by the computations in \cite{DS}. The following highlights this dichotomy further.

\begin{introprop}\label{prop:hzcomps} Suppose $T^4=1$ in $\sS$. Then we have the following computations for $h_\sS$:
	\begin{itemize}
		\item[{\em (i)}] If $K$ is alternating, or more generally quasi-alternating, then $h_\sS(K)=0$.
		\item[{\em (ii)}] For torus knots, $h_\sS$ is determined recursively by: for coprime $p,q\in \Z_{>0}$, 
		\[
			h_\sS(T_{p,p+q}) +\frac{1}{2}\sigma(T_{p,p+q}) = h_\sS(T_{p,q}) + \frac{1}{2}\sigma(T_{p,q}) - \left\lfloor p^2/4 \right\rfloor.
		\]
	\end{itemize}
\end{introprop}

\noindent These computations show that in the case $T^4=1$, for quasi-alternating knots and torus knots the invariant $h_\sS(K)+\frac{1}{2}\sigma(K)$ agrees with the Heegaard Floer invariant $\upsilon(K)$ (upsilon of $K$) defined by Ozsv\'{a}th--Stipsicz--Szab\'{o} \cite{oss} and also with the invariant $-\frac{1}{2}t(K)$ defined by Ballinger \cite{ballinger} in the setting of Khovanov homology.

An important ingredient in the computations of Proposition \ref{prop:hzcomps} is the relationship between $h_\sS(K)$ and the unoriented 4-ball genus $\gamma_4(K)$. Also called the 4-dimensional crosscap number, $\gamma_4(K)$ is defined to be the minimal $b_1(S)=b_1(S;\Z/2)$ among all smoothly embedded possibly non-orientable surfaces $S$ in the 4-ball with boundary $K$. Clearly $\gamma_4(K)\leq 2g_s(K)$. We have the following:

\begin{introthm}\label{thm:hzineq1}
	Suppose $T^4=1$ in $\sS$. Suppose $S$ is a properly smoothly embedded (possibly non-orientable) surface in the 4-ball with boundary the knot $K$. Then
	\begin{equation}
		\left|h_\sS(K)+\frac{1}{2}\sigma(K) - \frac{1}{4}S\cdot S \right| \leq b_1(S). \label{eq:ineqhznonor}
	\end{equation}
\end{introthm}

\noindent The term $S\cdot S$ should be interpreted as the euler number of the normal bundle of $S$. Just as in \cite{batson,oss,ballinger}, we may combine the inequality \eqref{eq:ineqhznonor} with an inequality $|\sigma(K)-\frac{1}{2}S\cdot S|\leq b_1(S)$ of Gordon--Litherland \cite{gl} to obtain the following.

\begin{introthm}\label{thm:hzunorientedgenus} Suppose $T^4=1$ in $\sS$. For a knot $K\subset S^3$ we have 
\[
	|h_\sS(K)| \leq \gamma_4(K).
\]
\end{introthm}

\noindent These inequalities are also satisfied upon replacing $h_\sS(K)+\frac{1}{2}\sigma(K)$ by the invariant $\upsilon(K)$ \cite[Theorems 1.1, 1.2]{oss} and also by $-\frac{1}{2}t(K)$ \cite[Theorem 1.1]{ballinger}. We remark that Theorem \ref{thm:hzunorientedgenus} applied to connected sums of $T_{3,4}$ shows singular instanton theory detects that $\gamma_4(K)$ can be arbitrarily large, a result first proved by Batson \cite[Theorem 2]{batson}, and which can also be proved using the invariants $\upsilon(K)$ and $t(K)$.

We will upgrade Proposition \ref{prop:hzcomps} (i) and show that in the case $T^4=1$, if $K$ is a quasi-alternating knot then the associated $\Gamma$-invariant is the same as that for the unknot; see Proposition \ref{prop:gammazquasialt}. We also remark that Theorem \ref{thm:existrep2} holds with $-\sigma(K)/2$ replaced by $h_\sS(K)$ when $T^4=1$ in $\sS$. Because in this situation $h_\sS$ and $\sigma$ are linearly independent concordance homomorphisms, this provides more examples of existence results for traceless $SU(2)$ representations; see Theorem \ref{thm:existrep3}.

In \cite{km-concordance}, Kronheimer and Mrowka define homomorphisms $f_\sigma$ from the smooth concordance group to $\Z$ for suitable base changes $\sigma:\bF[T_0^{\pm 1}, T_1^{\pm 1}, ^{\pm 1}, T_2^{\pm 1},T_3^{\pm 1}]\to \sS$ where $\sS$ is a ring and $\bF$ is the field with 2 elements. If $\sigma(T_0)=\sigma(T_1)=1$, then the corresponding invariant $f_\sigma$ behaves similarly to $h_\sS$ (here $T^4=1$), in that it satisfies the same kinds of inequalities; see Proposition 5.5 of \cite{km-concordance} and the discussion thereafter. It would be interesting to determine the precise relationship between these invariants.

\subsection*{Outline}

Section \ref{sec:prelims} provides background on the equivariant singular instanton constructions from \cite{DS}. The discussion therein of cobordism-induced morphisms is slightly more general than that of \cite{DS}, and is generalized even further in Section \ref{sec:inequalities}. In Section \ref{sec:prelims} we also consider cobordisms with immersed surfaces, and the relations for some basic cobordism moves in this setting are discussed following \cite{Kr:obs}. Section \ref{sec:comps} contains computations for $2$-bridge knots, and we prove Theorems \ref{thm:doubletwist} and \ref{thm:unknottingnumbers}. In Section \ref{sec:inequalities} we prove inequalities for singular Fr\o yshov invariants and for the $\Gamma_K$ invariant, of which Theorem \ref{thm:gammaintro} is a special case, and this in turn completes the proof of Theorem \ref{thm:clasp}. Theorems \ref{torus-conc-reps} and \ref{torus-conc-floer-homology} are also proved in Section \ref{sec:inequalities}. In Section \ref{sec:froyshovandsuspension} we prove Theorems \ref{thm:htwistedsign} and \ref{thm:scomplexsign}. In Section \ref{sec:unoriented}, we discuss the invariant $h_\sS$ in the case that $T^4=1$ and prove Proposition \ref{prop:hzcomps} and Theorem \ref{thm:hzineq1}.

\subsection*{Acknowledgments}

The authors thank Peter Kronheimer and Tom Mrowka for steering attention to the problem answered by Theorem \ref{thm:clasp}, and for suggesting the potential utility of using connected sums of $7_4$, and more generally double twist knots. The authors also thank Charles Livingston and Brendan Owens for helpful correspondences and the anonymous referee for helpful comments and suggestions on earlier versions of the paper.

\subsection*{Conventions}

We write $\sS$ for a commutative integral domain algebra over $\Z[U^{\pm 1}, T^{\pm 1}]$ with unit. The pair $(Y,K)$ always denotes a knot $K$ embedded in an integer homology 3-sphere $Y$. In the case that $Y=S^3$, we drop $S^3$ from our notation whenever it does not cause any confusion. The unknot in the 3-sphere is denoted $U_1$. If $K$ is a knot in the 3-sphere, $-K$ denotes its mirror. Given a knot $K$ and a positive integer $n$, we write $nK$ for the $n$-fold connected sum of $K$. If $n$ is negative, $nK=n(-K)$, and $0K=U_1$.



\section{Preliminaries}\label{sec:prelims}

In this section we provide some background material. Most but not all of the material discussed is from \cite{DS}, and this section serves in part as an introduction to that more detailed reference. In the first two subsections, we review the algebra of $\cS$-complexes and their Fr\o yshov invariants, and briefly recall some of the invariants defined for knots from \cite{DS}. Subsections \ref{subsec:reducibles} and \ref{sec:morphisms} slightly generalize the class of cobordism-induced morphisms from \cite{DS}. In Subsection \ref{subsec:cobrels}
 we review some relations on morphisms induced by cobordism moves following \cite{Kr:obs}. In Subsection \ref{sec:enriched} we review the structure of the Chern--Simons filtration on the $\cS$-complex for a knot, and the resulting concordance invariants $\Gamma_{(Y,K)}$ are discussed in Subsection \ref{sec:gammadef}.

\subsection{$\cS$-complexes and Fr\o yshov invariants}\label{subsec:scomplexes} Let $R$ be an integral domain. An {\emph{$\cS$-complex over $R$} is the data $(\widetilde C_\ast,\widetilde d,\chi)$ where 
\begin{itemize}
\item $(\widetilde C_\ast, \widetilde d)$ is a chain complex with $\widetilde C_\ast$ a free $R$-module;
\item $\chi$ is a degree 1 endomorphism of $\widetilde C_\ast$ such that $\chi^2=0$ and $\widetilde d \chi + \chi \widetilde d = 0$;
\item $\text{ker}(\chi)/\text{im}(\chi)$ is isomorphic to $R_{(0)}$, a copy of $R$ in grading $0$.
\end{itemize}
All $\cS$-complexes in the sequel will be $\Z/4$-graded. A typical $\cS$-complex is given by $\widetilde C_\ast = C_\ast \oplus C_{\ast-1}\oplus R$ where $(C_\ast,d)$ is a chain complex and:
\[
	\widetilde d = \left[ \begin{array}{ccc} d & 0 & 0 \\ v & -d & \delta_2 \\ \delta_1 & 0 & 0  \end{array} \right], \qquad \chi =  \left[ \begin{array}{ccc}0 & 0 & 0 \\ 1 & 0 & 0 \\ 0 & 0 & 0  \end{array} \right].
\]
Here $v:C_\ast\to C_{\ast-2}$, $\delta_1:C_1\to R$ and $\delta_2:R\to C_{-2}$ are maps that necessarily satisfy $\delta_1d=0$, $d\delta_2=0$, and $dv-vd-\delta_2\delta_1=0$. In fact every $\cS$-complex is isomorphic to such an $\cS$-complex, and we therefore freely use this latter description.

A {\emph{morphism}} $\widetilde \lambda:\widetilde C_\ast\to \widetilde C_\ast'$ between two $\cS$-complexes $\widetilde C_\ast = C_\ast \oplus C_{\ast-1}\oplus R$ and $\widetilde C'_\ast = C'_\ast \oplus C'_{\ast-1}\oplus R$ is a chain map which has the form
\begin{equation}\label{eq:morphism}
	\widetilde \lambda = \left[ \begin{array}{ccc} \lambda & 0 & 0 \\ \mu & \lambda & \Delta_2 \\ \Delta_1 & 0 & \eta  \end{array} \right]
\end{equation}
where $\eta\in R$ is a {\em non-zero} element. Here we slightly diverge from \cite{DS}, where a morphism was required to have $\eta=1$. However, all of the constructions easily extend to the case in which $\eta$ is non-zero.

\begin{remark}
In Subsection \ref{sec:higherlevelmorphisms} we will consider more general types of morphisms. $\diamd$
\end{remark}

An {\emph{$\cS$-chain homotopy}} $\widetilde C_\ast\to \widetilde C_\ast'$ between two morphisms is a chain homotopy in the usual sense which with respect to the decompositions has the form
\[
	 \left[ \begin{array}{ccc} K & 0 & 0 \\ L & -K & M_2 \\ M_1 & 0 & 0  \end{array} \right].
\]
A {\emph{chain homotopy equivalence}} between $\cS$-complexes $\widetilde C_\ast$ and $\widetilde C_\ast'$ is a pair of morphisms $\widetilde \lambda:\widetilde C_\ast\to \widetilde C_\ast'$ and $\widetilde\lambda':\widetilde C'_\ast\to \widetilde C_\ast$ such that $\widetilde \lambda'\widetilde\lambda$ and $\widetilde \lambda\widetilde\lambda'$ are $\cS$-chain homotopy equivalent to the identity morphisms.

Given $\cS$-complexes $(\widetilde C_\ast, \widetilde d,\chi)$ and $(\widetilde C'_\ast, \widetilde d',\chi')$ we have the tensor product $\cS$-complex $(\widetilde C^\otimes_\ast , \widetilde d^\otimes, \chi^\otimes)$ defined by setting $\widetilde C_\ast^\otimes =\widetilde C_\ast \otimes \widetilde C'_\ast$ with the tensor product grading, and
\begin{equation}
	\widetilde d^\otimes = \widetilde d\otimes 1 + \epsilon \otimes \widetilde d', \qquad \chi^\otimes = \chi\otimes 1 + \epsilon \otimes \chi'.\label{eq:tensorprod}
\end{equation}
Here $\epsilon$ is the ``sign map'' on $\widetilde C_\ast'$, multiplying elements of homogeneous even (resp. odd) degree by $+1$ (resp. $-1$). The dual of an $\cS$-complex $(\widetilde C_\ast, \widetilde d,\chi)$ has underlying complex the dual chain complex, and with the dual of $\chi$ acting on it.

We say two $\cS$-complexes $\widetilde C_\ast$ and $\widetilde C_\ast'$ are {\emph{locally equivalent}} if there exist morphisms in both directions: $\widetilde C_\ast\to \widetilde C_\ast'$ and $\widetilde C'_\ast\to \widetilde C_\ast$. The equivalence classes form an abelian group $\Theta_R^\cS$ with addition and inversion induced by tensor product and dual, respectively. We call $\Theta_R^\cS$ the {\emph{local equivalence group}} of $\cS$-complexes (over $R$ and $\Z/4$-graded). 

The Fr\o yshov invariant is a surjective homomorphism $h:\Theta_R^\cS\to \Z$. It is characterized as follows, which the reader may take as a definition:

\begin{prop}[Prop. 4.15 of \cite{DS}]\label{h-reinterpret}
	The invariant $h(\widetilde C_*)$ is positive if and only if there is an $\alpha\in C_1$ such that $d(\alpha)=0$ and $\delta_1(\alpha)\neq 0$.
	If $h(\widetilde C_*)=k$ for a positive integer $k$, then $k$ is the largest integer such that there exists $\alpha\in C_*$
	satisfying the following properties:
	\begin{equation*}
	  d\alpha=0,\hspace{1cm}\delta_1v^{k-1}(\alpha)\neq 0,\hspace{1cm}\delta_1v^{i}(\alpha)=0\hspace{.5cm} \text{ for }i\leq k-2.
	\end{equation*}
\end{prop}

\noindent If $R$ is a field, then $h:\Theta_R^\cS\to \Z$ is an isomorphism. In general, $h:\Theta_R^\cS\to \Z$ factors through the isomorphism given by the Fr\o yshov invariant for the field of fractions of $R$.

\subsection{Invariants for knots in homology spheres}\label{subsec:invariants}

Let $(Y,K)$ be a pair of an oriented integer homology 3-sphere $Y$ with a knot $K\subset Y$. In \cite{DS}, the authors associated to $(Y,K)$ a $\Z/4$-graded $\cS$-complex over $\Z$:
\[
	\widetilde C_\ast(Y,K) = C_\ast(Y,K) \oplus C_{\ast-1}(Y,K)\oplus \Z.
\]
In general, this $\cS$-complex depends on an orbifold metric on $Y$ with $\Z/2$-orbifold singularity along $K$ and some perturbation data. But the $\cS$-chain homotopy type of this is an invariant of $(Y,K)$. Its homology was identified in \cite[Theorem 8.9]{DS} with Kronheimer and Mrowka's $I^\natural_\ast(Y,K)$ from \cite{KM:unknot}, up to a grading shift. The homology of $C_\ast(Y,K)$ is denoted $I_\ast(Y,K)$ and is called the {\emph{irreducible}} instanton homology of the knot, as its chain complex is generated by (perturbed) irreducible $SU(2)$ flat connections on $Y\setminus K$ with holonomy of order 4 around small meridians. We have
\[
	\chi\left( I_\ast(Y,K)\right) = \frac{1}{2}\sigma(Y,K) + 4\lambda(Y)
\]
where $\lambda(Y)$ is the Casson invariant of $Y$, a result essentially due to Herald \cite{herald}.

The $\cS$-complex structure of $\widetilde C_\ast(Y,K)$ is closely tied to the following important fact. For knots in integer homology 3-spheres $(Y,K)$, among singular flat $SU(2)$ connections mod gauge on $(Y,K)$ as above, all are irreducible except for a unique {\emph{reducible}}, which is typically denoted $\theta$. This reducible has gauge stabilizer $S^1$ and corresponds to the abelian holonomy representation $\pi_1(Y\setminus K)\to SU(2)$ with image  $\{\pm 1,\pm i\}$, and sends a meridian to $i$. The meridian is around a basepoint $p\in K$, and its orientation is fixed by choosing an orientation of $K$. However the choices of the basepoint $p$ and the orientation of $K$ do not affect the isomorphism classes of the invariants we consider, and so they are usually suppressed from the notation.

The Fr\o yshov invariant associated to the $\cS$-complex $\widetilde C_\ast(Y,K)$ is denoted $h_\Z(Y,K)\in \Z$. This induces a homomorphism from the homology concordance group to the integers. The same constructions carry through for any coefficient ring $R$ which is an integral domain. Some computations of $h_\Z(Y,K)$ were given in \cite[Section 9]{DS}. 

Motivated by constructions from \cite{KM:YAFT}, the authors also defined versions of the above invariants with local coefficients. Generally, for an integral domain algebra $\sS$ over $\Z[U^{\pm 1}, T^{\pm 1}]$, we construct a $\Z/4$-graded $\cS$-complex over $\sS$, denoted
\begin{equation}\label{eq:twistedscomplex}
	\widetilde C_\ast(Y,K;\Delta_\sS),
\end{equation}
whose $\cS$-chain homotopy type is an invariant of $(Y,K)$. The variable $U$ encodes information about the Chern-Simons functional, while $T$ has to do with the monopole numbers of instantons. In fact there is more structure on \eqref{eq:twistedscomplex} than just the $\cS$-chain homotopy type; see Subsection \ref{sec:enriched}. From the $\cS$-complex $\widetilde C_\ast(Y,K;\Delta_\sS)$ we obtain a Fr\o yshov invariant
\[
	h_\sS(Y,K)\in \Z,
\]
also a homomorphism from the homology concordance group to $\Z$. 

The $\Z/4$-graded irreducible homology $I_\ast(Y,K;\Delta_\sS)$ is related to the Fr\o yshov invariant $h_\sS(Y,K)$. If $h_\sS(Y,K)\geq 0$, then Proposition \ref{h-reinterpret} directly implies:
\begin{equation}\label{h-twisted-ineq-irr-I}
  \rk(I_1(Y,K;\Delta_\sS))\geq \left\lceil \frac{h_\sS(Y,K)}{2}\right\rceil,\hspace{.5cm}
  \rk(I_3(Y,K;\Delta_\sS))\geq \left\lfloor \frac{h_\sS(Y,K)}{2} \right\rfloor.
\end{equation}
Inequalities for the case $h_\sS(Y,K)\leq 0$ are then obtained from the general properties $\text{rk}(I_*(-Y,-K;\Delta_\sS))=\text{rk}(I_{3-\ast}(Y,K;\Delta_\sS))$ and $h_\sS(-Y,-K)=-h_\sS(Y,K)$.

Given two pairs $(Y,K)$ and $(Y',K')$ of knots in integer homology 3-spheres, we may form their connected sum $(Y\# Y',K\# K')$. The invariants are related by:

\begin{theorem}[Theorem 6.1 of \cite{DS}]\label{thm:connectedsum}
	There is a chain homotopy equivalence
\[
	\widetilde C_\ast(Y\#Y',K\#K';\Delta_\sS) \simeq \widetilde C_\ast(Y,K;\Delta_\sS) \otimes \widetilde C_\ast(Y',K';\Delta_\sS) 
\]
of $\Z/4$-graded $\cS$-complexes, and it is natural with respect to split cobordisms, when defined.
\end{theorem}
\noindent This result was a key ingredient in \cite{DS} for relating the equivariant theory to the invariants defined previously by Kronheimer and Mrowka.

To $(Y,K)$ we may also associate three equivariant Floer groups, similar in flavor to the Heegaard Floer package for knots \cite{os-knot, rasmussen-thesis}, which comes with three $\Z/4$-graded $\sS[x]$-modules fitting into an exact triangle:
\begin{equation}\label{eq:exacttriangletwisted}
	\cdots \longrightarrow \hrI(Y,K;\Delta_\sS) \xrightarrow{i}  \brI(Y,K;\Delta_\sS) \xrightarrow{p} \crI(Y,K;\Delta_\sS) \xrightarrow{j}  \hrI(Y,K;\Delta_\sS) \longrightarrow \cdots
\end{equation}
In fact, this is an entirely algebraic construction, depending only on the associated $\cS$-complex for $(Y,K)$; for any $\cS$-complex one has a similar set of equivariant homology groups. See \cite[Section 4]{DS} for details.

\subsection{Reducibles on negative definite cobordisms}\label{subsec:reducibles}

Let $(W,S):(Y,K)\to (Y',K')$ be a cobordism of pairs between oriented integer homology 3-spheres with knots. In particular, $W$ is an oriented cobordism $Y\to Y'$, and we assume that $S$ is a connected, oriented surface cobordism $K\to K'$ embedded in $W$. Fix a cohomology class $c\in H^2(W;\Z)$. Let $E$ be a $U(2)$-bundle over $W$ with $c_1(E)=c$.

Let $W^+$ (resp. $S^+$) be obtained from $W$ (resp. $S$) by attaching cylindrical ends to the boundary. A singular connection associated to $(W,S,c)$ is a connection on $E|_{W^+\backslash S^+}$ which is asymptotic to connections on $Y\setminus K$ and $Y'\setminus K'$ with a controlled singularity along $S^+$. The latter condition is that the holonomy of $A$ around a meridian of $S^+$ is asymptotic to 
\[
  \left[
  \begin{array}{cc}
  	i&0\\
	0&-i
  \end{array}
  \right]
\]
as the size of the meridian goes to zero. The adjoint action of $U(2)$ on $\su(2)$, the Lie algebra of $SU(2)$, induces the adjoint map $\ad:U(2) \to SO(3)$. We focus on singular connections $A$ such that $\det(A)$ is a fixed (non-singular) $U(1)$ connection $\lambda$ on  $\det(E)$. Determinant one automorphisms of the bundle $E$ which are compatible with the asymptotic and singularity conditions of singular connections define a gauge group acting on the space of singular connections. A connection is irreducible if its stabilizer with respect to this gauge group is $\pm 1$, and is reducible otherwise. (See \cite{KM:YAFT} for more details on singular connections.)

The topological energy, or action, of a singular connection $A$ is given by
\begin{equation}
	\kappa(A) = \frac{1}{8\pi^2}\int_{W^+\setminus S^+} \text{Tr}(F_{\ad (A)}\wedge F_{\ad(A)}).\label{eq:kappadef}
\end{equation}
The curvature of $\ad(A)$ extends to the singular locus and its restriction to $S^+$ has the form 
\[
  \left[
  \begin{array}{cc}
  	\Omega&0\\
	0&-\Omega\\
  \end{array}
  \right]
\]
where $\Omega$ is a 2-form with values in the orientation bundle of $S^+$. The monopole number of a singular connection $A$ is given by
\begin{equation}
	\nu(A) = \frac{i}{\pi}\int_{S^+} \Omega - \frac{1}{2} S\cdot S.\label{eq:nudef}
\end{equation}
Note that our convention here is slightly different from \cite{DS} where the self-intersection term in \eqref{eq:nudef} does not appear. 
Equip $W^+$ with a Riemannian metric which is singular of cone angle $\pi$ along $S^+$ and agrees with product metrics at the ends. A {\emph{singular instanton}} associated to $(W,S,c)$ is a singular connection $A$ such that $\ad(A)$ is a finite-energy anti-self-dual connection. 

We write $M(W,S,c;\alpha,\alpha')_d$ for the moduli space, of expected dimension $d$, of singular instantons on $(W,S,c)$ which are asymptotic to flat connections $\alpha$ on $Y\setminus K$ and $\alpha'$ on $Y'\setminus K'$. In general, perturbations are required to achieve transversality for these moduli spaces. We write $M(\alpha,\alpha')_d$ in the case that $(W,S,c)$ is the product $[0,1]\times (Y,K)$, and $c=0$, with product metric structure. This moduli space has a natural $\R$-action and we write $\breve{M}(\alpha,\alpha')_{d-1}$ for the quotient of the part where $\R$ acts freely.

We now discuss {\it reducible singular instantons}, or simply {\it reducibles} for short. Assume $b^1(W)=b^+(W)=0$. In particular, we may assume for simplicity that $\lambda$ has harmonic anti-self-dual $L^2$ curvature. Consider singular instantons on $E\to W^+\setminus S^+$ that are asymptotic to reducible flat singular connections associated to $(Y,K)$ and $(Y',K')$. The reducibles among these singular instantons are in bijection with elements of $H^2(W;\Z)$; the correspondence sends a singular instanton compatible with a splitting
\begin{equation}\label{eq:bundlesplitting}
	E = L\oplus L^\ast \otimes \det(E)
\end{equation}
to $c_1(L)\in H^2(W;\Z)$. In fact, there is a unique singular instanton $A_L$ of the form $\lambda_L\oplus \lambda_L^\ast\otimes \lambda$, compatible with \eqref{eq:bundlesplitting}, where $\lambda_L$ is a $U(1)$ connection on $L\to W^+\setminus S^+$ with harmonic anti-self-dual $L^2$ curvature representing the cohomology class $c_1(L)+\frac{1}{4}S$, and for which $\lambda_L$ has holonomy of a meridian asymptotic to $i\in U(1)$. We will slightly abuse notation and write $S$ for both the homology class induced by $S$ and its Poincar\'{e} dual. Unlike in the non-singular setup, the order of the line bundle factors in the reduction \eqref{eq:bundlesplitting} matters.

The topological energy of the reducible instanton $A_L$ is computed from \eqref{eq:kappadef} to be
\begin{equation}\label{eq:kappa}
	\kappa(A_L)  = -\left(c_1(L) + \frac{1}{4}S - \frac{1}{2}c\right)^2.
\end{equation}
The monopole number of $A_L$ is computed from \eqref{eq:nudef} to be
\begin{equation}\label{eq:nu}
	\nu(A_L) = (2c_1(L)-c)\cdot S.
\end{equation}
The index of the reducible $A_L$ is computed in \cite[Lemma 2.6]{DS} as follows:
\begin{equation}\label{eq:index}
	\text{ind}(A_L) = 8\kappa(A_L) - \frac{3}{2}(\chi(W) + \sigma(W)) + \chi(S) + \frac{1}{2}S\cdot S + \sigma(Y,K) - \sigma(Y',K') -1.
\end{equation}
Let us call $A_L$ (and $c_1(L)$) {\emph{minimal}} (with respect to $c$) if it minimizes $\kappa(A_L)$ among all reducibles on $E$. Define $\kappa_\text{min}(W,S,c)$ to be $\kappa(A_L)$ for any minimal reducible $A_L$:
\begin{equation}\label{eq:kappamin}
	\kappa_\text{min}(W,S,c) := \min_{z\in H^2(W;\Z)} -\left( z+\frac{1}{4}S - \frac{1}{2}c\right)^2.
\end{equation} 
Note that this quantity depends only on the cohomology ring of $W$, the homology class of $S$, and the index 2 coset of $c$. If $c=0$, we simply write $\kappa_\text{min}(W,S)$ for the above quantity.

The moduli spaces $M(W,S,c;\alpha,\alpha')_d$ are orientable. We use similar conventions as in \cite{km-embedded-ii,DS} to orient these moduli spaces and we refer the reader to these references for more details. The determinant line bundle $l(W,S,c;\alpha,\alpha')$ of the deformation complex of the elements of $M(W,S,c;\alpha,\alpha')_d$ can be identified with the orientation line bundle of this moduli space. In fact, these determinant line bundles can be extended to the configuration space of all singular connections and they are related to each other with respect to gluing of connections. This allows us to orient all moduli spaces $M(W,S,c;\alpha,\alpha')_d$ after fixing a small amount of data as it is recalled in \cite[Subsection 2.9 and Section 3]{DS}. To be more precise, we only considered the case that $c=0$ in \cite{DS}. 

In the more general case, we need to fix a (not necessarily minimal) reducible, and we let this reducible be $A_{L_0}$ with $c_1(L_0)=0$. Since $b_1(W)=b^+(W)=0$, there is a canonical way to trivialize the orientation bundle at any reducible $A_L$. To fix orientations of $M(W,S,c;\alpha,\alpha')_d$, we firstly trivialize the determinant bundle of the moduli space containing $A_{L_0}$ in the canonical way and then extend this trivialization to the other moduli spaces using compatibility of determinant bundles with respect to gluing of connections. Using this convention, the determinant bundle at a reducible $A_L$ is oriented in the standard way if and only if $c_1(L)^2$ is even  \cite[Appendix 1(ii)]{km-embedded-ii}. (The result in \cite[Appendix 1(ii)]{km-embedded-ii}, which is stated in the case that $(W,S)$ is closed, can be extended to the case of cobordisms of pairs by an excision argument.) Motivated by this the signed {\it count} of minimal reducibles with powers of $T$ keeping track of monopole numbers is:
\begin{equation}\label{eq:etadef}
	\eta(W,S,c) := \sum (-1)^{c_1(L)^2}T^{\nu(A_L)}\in \Z[T^{\pm 1}].
\end{equation}
Here the sum is over all $c_1(L)$ corresponding to a minimal reducible $A_L$ with respect to $c$. In the case that $c=0$, we drop again $c$ from our notation in \eqref{eq:etadef}.

\subsection{Morphisms from cobordisms}\label{sec:morphisms}

We now slightly generalize the setup from \cite{DS}, and formulate a class of cobordisms amenable for inducing morphisms between $\cS$-complexes for knots.

\begin{definition}\label{defn:negdefpair}
	Let $(W,S):(Y,K)\to (Y',K')$ be a cobordism of pairs between oriented knots in integer homology 3-spheres, where $S$ is connected and oriented, and $c\in H^2(W;\Z)$. Let $\sS$ be an algebra over $\Z[U^{\pm 1}, T^{\pm 1}]$. We say $(W,S,c)$ is {\emph{negative definite over $\sS$}} if:
	\begin{itemize}
		\item[(i)] $b^1(W)=b^+(W)=0$;
		\item[(ii)] the index of one (and thus every) minimal reducible $A_L$ is $-1$;
		\item[(iii)] $\eta(W,S,c)$ is non-zero as an element in $\sS$. 
	\end{itemize}
If $c=0$, we also say $(W,S)$ is a {\emph{negative definite pair over $\sS$}}.  $\diamd$
\end{definition}

\begin{remark}
	More general types of cobordisms will be considered in Subsection \ref{sec:higherlevelmorphisms}. $\diamd$
\end{remark}

\begin{remark} Using \eqref{eq:index}, condition (ii) simplifies to the following condition:
\[
	 8\kappa_\text{min}(W,S,c) + \chi(S) + \frac{1}{2}S\cdot S +  \sigma(Y,K) - \sigma(Y',K')  = 0. \,\,\,\,\,\, \diamd
\]
\end{remark}

\begin{remark} Let $\cT\subset H^2(W;\Z)$ be the torsion subgroup of cardinality $|\cT|$. Call $z\in H^2(W;\Z)/\cT$ minimal (with respect to $c$) if it lifts to a minimal class in $H^2(W;\Z)$. Then
\[ 
	\eta(W,S,c) =|\cT|\cdot\sum_{\substack{z\in H^2(W;\Z)/\cT\\ \text{minimal}}} (-1)^{z^2}T^{(2z-c)\cdot S}.
\]
Consequently, if the order of torsion $|\cT|$ is zero in $\sS$, then $(W,S,c)$ is not negative definite over $\sS$.  $\diamd$
\end{remark}

Suppose $(W,S,c)$ is negative definite over $\sS$. Then we can associate to it a morphism of $\cS$-complexes, in the sense discussed in Subsection \ref{subsec:scomplexes}:
\begin{equation}\label{eq:cobmap}
	\widetilde \lambda_{(W,S,c)} : \widetilde C_\ast(Y,K;\Delta_\sS)\to \widetilde C_\ast(Y',K';\Delta_\sS)
\end{equation}
To explain this we recount some more background on the construction of $\widetilde C_\ast(Y,K;\Delta_\sS)$.

First, we mention the two key identities on which the local coefficients structure relies. First, let $(W,S):(Y,K)\to (Y',K')$ be a cobordism of pairs with $S$ oriented, and suppose $\alpha$ and $\alpha'$ are singular $SU(2)$ connections on $Y\setminus K$ and $Y'\setminus K'$ respectively. Let $E$ be a $U(2)$-bundle over $W$ with $c=c_1(E)$. Then if $A$ is a singular connection on $W\setminus S$ which restricts to $\alpha$ and $\alpha'$ at the boundary components, we have:
\begin{equation}
	2\kappa(A) - 2\kappa_\text{min}(W,S,c) \equiv \text{CS}(\alpha)-\text{CS}(\alpha') \pmod \Z.\label{eq:kappacs}
\end{equation}
Furthermore, we note that $2\kappa_\text{min}(W,S,c)\equiv -\frac{1}{8}(S-2c)^2 \pmod \Z$. The above may be taken as a definition of $\text{CS}(\alpha)\in \R/\Z$ by letting $\alpha'$ be the reducible. The second identity is:
\begin{equation}
	\nu(A)  \equiv \text{hol}_{K'}(\alpha') - \text{hol}_K(\alpha) \pmod \Z.\label{eq:nuhol}
\end{equation}
Here $\text{hol}_K(\alpha)$ is the limiting $U(1)$ holonomy around a Seifert longitude for $K$ of the restriction of $\alpha$ to the complex line bundle $L$ defining the reduction near $K$. 

We next recall from \cite[Subsection 7.1]{DS} that $C_\ast(Y,K;\Delta_\sS)= \bigoplus_\alpha \Delta_\alpha$ where the sum is over (perturbed) irreducible flat $SU(2)$ connections $\alpha$ on $Y\setminus K$ with traceless holonomy around meridians, and $\Delta_\alpha$ is defined to be the $\sS$-module
\[
	\Delta_\alpha := \sS\cdot U^{\text{CS}(\alpha)}T^{\text{hol}_K(\alpha)}. 
\]
Here $\text{CS}(\alpha)$ is any lift to $\R$ of the Chern--Simons invariant of $\alpha$, and similarly for $\text{hol}_K(\alpha)$. The maps $d$, $v$, $\delta_1$, $\delta_2$ for $C_\ast(Y,K;\Delta_\sS)$ are defined by certain counts of singular instantons $[A]$ on $\R\times Y$, where each instanton has weight $U^{-2\kappa(A)}T^{\nu(A)}$. For example, we have
\[
	d(\alpha) = \sum \varepsilon(A)\Delta(A)(\alpha)
\]
where the sum is over singular instantons $[A]$ on $\R\times Y$ in $0$-dimensional moduli spaces $\breve{M}(\alpha,\alpha')_0$ and $\Delta(A):\Delta_\alpha\to \Delta_{\alpha'}$ is the module homomorphism which is multiplication by $U^{-2\kappa(A)}T^{\nu(A)}$. The sign $\varepsilon(A)=\pm 1$ is determined by the orientation of the moduli space. When carrying out these constructions generic metrics and perturbations are chosen, but we typically omit these from our discussions. For more details see \cite{DS}.

We now assume $(W,S,c)$ is negative definite over $\sS$. To construct the morphism $\widetilde \lambda_{(W,S,c)}$ we must prescribe the maps appearing in the decomposition \eqref{eq:morphism}. The components $\lambda$, $\mu$, $\Delta_1$, $\Delta_2$ are defined just as in \cite[Section 3]{DS}, adapted to the local coefficients setting of \cite[Section 7]{DS}. For example, $\lambda (\alpha) = \sum\epsilon(A)\Delta(A)(\alpha)$ where the sum is over $[A]\in M(W,S,c;\alpha,\alpha')_0$ and $\Delta(A):\Delta_\alpha\to \Delta_{\alpha'}$ is multiplication by: 
\[
	\Delta(A)=U^{2\kappa_\text{min}(W,S,c)-2\kappa(A)} T^{\nu(A)}.
\]
The other components of $\widetilde \lambda_{(W,S,c)}$ are defined similarly. The term $\eta\in \sS$ in the decomposition \eqref{eq:morphism} is defined to be the count of reducibles $\eta(W,S,c)$ from \eqref{eq:etadef}, which is non-zero by assumption. The compatibility of these maps with the local coefficient systems is owed to \eqref{eq:kappacs} and \eqref{eq:nuhol}.

Having prescribed the components of $\widetilde \lambda_{(W,S,c)}$, the relations
\[
	\lambda  d - d'  \lambda = 0, \qquad  \mu d +  \lambda v + \Delta_2 \delta_1 - v' \lambda + d'  \mu - \delta_2' \Delta_1 = 0
\]
follow just as for negative definite pairs; see e.g. \cite[Proposition 3.19]{DS}. The relations 
\[
	\Delta_1  d  + \eta \delta_1 - \delta'_1 \lambda = 0,  \qquad d' \Delta_2 - \eta\delta'_2 + \lambda \delta_2  = 0
\]
are only slight modifications of those in \cite[Proposition 3.10]{DS}; in the argument, one must keep track of trajectory breakings at all possible minimal reducibles with index $-1$, instead of a unique flat reducible. There is one technical point to make here: unlike for a negative definite pair, a general $(W,S,c)$ which is negative definite over $\sS$ may a priori have obstructed (degenerate) minimal reducibles. However, this can be rectified by using a small perturbation, for example as is done in \cite[Section 7.3]{dcx}. We obtain:

\begin{prop}\label{prop:morphism}
	Let $(W,S,c):(Y,K)\to (Y',K')$ be negative definite over $\sS$. Then there is an induced morphism of $\cS$-complexes $\widetilde \lambda_{(W,S,c)}:\widetilde C_\ast(Y,K;\Delta_\sS)\to \widetilde C_\ast(Y',K';\Delta_\sS)$.
\end{prop}

\begin{remark}
	Suppose $c=0$, $H_1(W;\Z)=0$, and the homology class of $S$ is divisible by $4$. The minimum \eqref{eq:kappamin} is achieved uniquely by $z=- \frac{1}{4}S$. Call the corresponding reducible $A_L$. Then $\kappa_\text{min}(W,L)=\kappa(A_L)=0$ and $\nu(A_L)=-\frac{1}{2}S\cdot S$. In particular, $A_L$ is flat, and $\eta(W,S)=T^{-\frac{1}{2}S\cdot S}\neq 0$. Thus $(W,S)$ is negative definite over $\sS$, for any $\sS$. This is the class of ``negative definite pairs'' in \cite[Defintion 2.33]{DS}. $\diamd$
\end{remark}

\begin{remark}
	If we drop the condition (iii) in Definition \ref{defn:negdefpair}, so that $\eta(W,S,c)$ may be zero, then we still have a map $\widetilde \lambda_{(W,S,c)}$ which respects the structure of the $\cS$-complexes. However, it is not a morphism in the terminology of Subsection \ref{subsec:reducibles}.  $\diamd$
\end{remark}

\begin{remark}
	In defining the cobordism map for $(W,S,c)$, a path on $S$ that interpolates between chosen basepoints on $K$ and $K'$ must be specified to define the component $\mu$. However, this choice is never important for us and so we suppress it from the notation. $\diamd$
\end{remark}

\begin{remark}
	A more functorial discussion for cobordism maps replaces $c\in H^2(W;\Z)$ with an oriented 2-manifold properly embedded in $W$ transverse to $S$ which represents the Poincar\'{e} dual to $c$. This is because such a submanifold naturally determines a $U(2)$ bundle over $(W,S)$, compare \cite{KM:unknot}. However, this is not important for our purposes and in this paper we settle for using cohomology classes $c\in H^2(W;\Z)$. $\diamd$
\end{remark}

\subsection{Some cobordism relations}\label{subsec:cobrels}

The importance of using local coefficients and in particular the invertibility of $T^4-1$ comes from the effect, on the singular instanton invariants, of some basic moves on surfaces in 4-manifolds. These relations were first established in the closed case by Kronheimer \cite{Kr:obs}, and have played an important role in Kronheimer and Mrowka's work on singular instanton Floer homology; see \cite{KM:YAFT, km-rasmussen, km-barnatan, km-concordance}. Here we state the relations as they apply in our current setup, at the chain level. The proof is a direct adaptation of the one for \cite[Proposition 3.1]{Kr:obs}; see also \cite[Proposition 5.2]{KM:YAFT}. 

Thus far, our cobordisms of pairs $(W,S):(Y,K)\to (Y',K')$ have assumed the surface $S$ to be embedded in $W$. We slightly relax our above hypotheses so that $S$ may be a smoothly {\emph{immersed}} cobordism with normal crossings, i.e. transverse double points. Blow up $W$ at each double point of $S$ to obtain $\overline W$, diffeomorphic to a connected sum of $W$ with copies of $\overline{\C\P}^2$, and let $\overline S\subset \overline W$ be the proper transform of $S$. We say $(W,S,c)$ is negative definite over $\sS$ if $(\overline W,\overline S,c)$ satisfies the conditions of Definition \ref{defn:negdefpair}. In this case we define
\begin{equation}
	\widetilde \lambda_{(W,S,c)} := (-1)^{s_+}T^{-2s_+}\widetilde \lambda_{(\overline W, \overline S,c)}.\label{eq:immersedconvention}
\end{equation}
Here $s_+$ is the number of positive double points of $S$. The factor $(-1)^{s_+}T^{-2s_+}$ is not so important and is only included to symmetrize some expressions. Here $c\in H^2(W;\Z)$, and we write $c$ also for its image under $H^2(W;\Z)\to H^2(\overline{W};\Z)$.

Our main motivation for the extension of cobordism maps to the normally immersed case is to study the behavior of $\widetilde \lambda_{(W,S,c)} $ with respect to relative homotopies of the surface $S$ in $W$. Any such homotopy in general position is a composition of a sequence of standard moves: relative ambient isotopy of $S$, positive twist, negative twist, finger move and the reverses of these moves \cite{FQ:top-4-manifolds}. A positive twist increases the number of positive double points by one, the negative twist increases the number of negative double points by one and a finger move introduces a cancelling pair of double points.

\begin{prop}\label{homotopy-moves}
	Suppose $(W,S,c)$ is negative definite over $\sS$, where $S$ is possibly immersed in $W$ as above. Let $S^\ast$ be obtained from $S$ by either a positive twist move or a finger move. Then
	\[
		\widetilde \lambda_{(W,S^\ast,c)} \sim (T^2-T^{-2})\widetilde \lambda_{(W,S,c)}
	\]
	where $\sim$ stands for $\cS$-chain homotopic, as morphisms of $\Z/4$-graded $\cS$-complexes over $\sS$.
	If $S^\ast$ is obtained by a negative twist move, then $\widetilde \lambda_{(W,S^\ast,c)}$ is $\cS$-chain homotopic to $\widetilde \lambda_{(W,S,c)}$.
\end{prop}

\begin{proof}	
	As already mentioned, the proof is adapted from \cite{Kr:obs} to our current setup. For this reason we omit some of the details and only highlight the key ingredients. 

	Without loss of generality, we may assume $S$ is smoothly embedded in $W$, and $c=0$. Suppose $S^\ast$ is obtained from $S$ by a positive twist move. Then the map $\widetilde \lambda_{(W,S^\ast)}=-T^{-2}\widetilde \lambda_{(\overline{W},\overline{S}^\ast)}$ is defined by choosing a metric and generic perturbation for the pair
\begin{equation}\label{eq:connsumtwist}
	(\overline{W},\overline{S}^\ast)=(W,S)\#(\overline{\C\P}^2, S_2)
\end{equation}
where $S_2$ is a conic in $\overline{\C\P}^2$, and in particular $[S_2]=2e$ where $e\in H^2(\overline{\C\P}^2;\Z)$ is a generator. Thus $S_2\cdot S_2 =-4$. Consider reducibles on $(\overline{\C\P}^2, S_2)$. There are exactly two that have minimal topological energy; the two reducibles $A_0$ and $A_{-e}$ correspond, respectively, to the minimizers $z=0$ and $z=-e$ in \eqref{eq:kappa}, in which we set $c=0$. Thus $\kappa_\text{min}=1/4$. 

A 1-parameter family of auxiliary data that stretches along the connected sum sphere induces an $\cS$-chain homotopy between $\widetilde \lambda_{(\overline{W},\overline{S}^\ast)}$ and the morphism $\widetilde \lambda^\infty$ for the pair \eqref{eq:connsumtwist} defined using a broken metric along the connected sum region. A gluing argument shows
\[
	\widetilde \lambda^\infty = \eta(\overline{\C\P}^2, S_2)\cdot \widetilde \lambda_{(W',S')}
\]
where $\widetilde \lambda_{(W',S')}$ is the morphism for $(W',S')$, a puncturing of $(W,S)$, equipped with auxiliary data involving a metric with cylindrical ends at the punctures. In other words, the instantons that contribute to $\widetilde \lambda^\infty$ can be described as instantons on $(W',S')$ grafted to the minimal reducible instantons on $(\overline{\C\P}^2, S_2)$, which are unobstructed. For the two minimal reducibles $A_0$ and $A_{-e}$ we compute the monopole numbers:
\[
	\nu(A_z) = 2z\cdot 2e= \begin{cases}0, & z=0\\ 4, & z=-e \end{cases}
\]
Thus $\eta(\overline{\C\P}^2, S_2)=1-T^4$.

Finally, a 1-parameter family of auxiliary data from any given metric and perturbation that defines $\widetilde \lambda_{(W,S)}$ and which stretches along the 3-sphere which encloses the connected sum location on $S$ provides a chain homotopy from $\widetilde \lambda_{(W,S)}$ to the morphism $\widetilde \lambda_{(W',S')}$.

The other moves are dealt with by similarly adapting \cite{Kr:obs}.
\end{proof}

\begin{cor}
	Suppose $(W,S,c)$ is negative definite over $\sS$, where the surface $S$ is possibly immersed in $W$ with transverse double points.
	\begin{itemize}	
		\item[(i)] If $S^\ast$ is obtained from $S$ by a negative twist, then $(W,S^\ast,c)$ negative definite over $\sS$.
		\item[(ii)] If $T^4\neq 1$ and $S^\ast$ is obtained from $S$ by a positive twist or a finger move, then $(W,S^\ast,c)$ is negative definite over $\sS$. 
	\end{itemize}
\end{cor}

\subsection{Chern--Simons filtration and enriched complexes}\label{sec:enriched}

In this subsection we review the structure of the Chern--Simons filtration on our $\cS$-complexes for knots. We will be brief, and refer to \cite[Section 7]{DS} for more details. 

\begin{definition}\label{defn:igraded}
	Let $R$ be an integral domain. An {\em I-graded $\cS$-complex} over $R[U^{\pm 1}]$ is an $\cS$-complex $(\widetilde C, \widetilde d, \chi)$ over $R[U^{\pm 1}]$ with a $\Z\times \R$-bigrading as an $R$-module which satisfies the following. Write $\widetilde C_{i,j}$ for the $(i,j)\in \Z\times \R$ graded summand. Then
\begin{itemize}
	\item[(i)] $U \widetilde C_{i,j} \subset  \widetilde C_{i+4,j+1}$
	\item[(ii)] $\widetilde d \widetilde C_{i,j} \subset  \bigcup_{k<j}\widetilde C_{i-1,k}$
	\item[(iii)] $\chi \widetilde C_{i,j} \subset  \widetilde C_{i+1,j}$
\end{itemize}
Also, $\widetilde C$ is generated as an $R[U^{\pm 1}]$-module by homogeneously bigraded elements. The distinguished summand in $\widetilde C$, isomorphic to $R[U^{\pm 1}]$, has generator $1$ in bigrading $(0,0)$. $\diamd$
\end{definition}

For an I-graded $\cS$-complex we write $\widetilde{\text{gr}}$ for the $\Z$-grading and $\text{deg}_I$ for the $\R$-grading, which we also call the ``instanton'' grading. When we write $\widetilde C_\ast$ for an I-graded $\cS$-complex, the subscript only records the $\Z$-grading, or the induced $\Z/4$-grading.

Note that if we forget the $\R$-grading, we still have a $\Z$-graded $\cS$-complex in the sense that $\widetilde C_\ast$ is a graded module over $R[U^{\pm 1}]$, where the latter is a graded ring with $U$ in degree $4$. However, in this paper, for simplicity the algebraic objects we consider are either I-graded $\cS$-complexes as above, or $\Z/4$-graded $\cS$-complexes as discussed earlier.

Suppose no perturbations are needed when constructing an $\cS$-complex for $(Y,K)$. Then $\widetilde C(Y,K;\Delta_\sS)$ has the structure of an I-graded $\cS$-complex with $\sS=R[U^{\pm 1}]$ and $R=\Z[T^{\pm 1}]$, as we now explain. For a singular connection $\alpha$ on $Y\setminus K$ choose a path to the reducible, i.e. a singular $SU(2)$ connection $A$ on $\R\times Y$ whose restrictions to $(-\infty,-1]\times Y$ and $[1,\infty)\times Y$ are pullbacks of $\alpha$ and a representative for $\theta$, respectively. Another such path $A'$ is equivalent to  $A$ if there is a determinant one automorphism $u$ of the underlying bundle of $A$ and $A'$ such that the restrictions of $u$ to $(-\infty,-1]\times Y$ and $[1,\infty)\times Y$ are pullbacks of gauge transformations of $E$ and the difference between $A$ and the pullback of $A'$ by $u$ is compactly supported on $\R\times Y$. We write $\widetilde \alpha$ for the equivalence class of the connection $A$. We call such a choice of $\widetilde \alpha$ a {\em lift} of $\alpha$. The Chern--Simons invariant of $\widetilde \alpha$ is
\[
	\text{CS}(\widetilde \alpha) := 2\kappa(\widetilde \alpha) \in \R,
\]
and is a lift of $\text{CS}(\alpha)\in \R/\Z$. The invariant $\text{CS}(\widetilde \alpha)$ depends only on the homotopy class of $\widetilde \alpha$, viewed as a path in the configuration space of singular connections on $\R\times Y$ relative to its limits $\alpha$ and $\theta$. Similarly, the path $\widetilde \alpha$ determines a lift $\text{hol}_K(\widetilde \alpha)\in \R$ of the holonomy invariant $\text{hol}_K(\alpha)\in \R/\Z$.

Given a singular connection $A$ on $\R\times Y$, we can ``add an instanton'' by grafting on to $A$ a standard instanton on $S^4$. This changes $(\kappa(A),\nu(A))$ into $(\kappa(A)+1,\nu(A))$. 
Similarly, we can ``add a monopole'' by grafting on to $A$ a standard singular instanton on $(S^4,S^2)$. This changes $(\kappa(A),\nu(A))$ to $(\kappa(A)+1/2,\nu(A)+2)$. In the case that $A$ represents a lift $\widetilde \alpha$ 
of $\alpha$, we can faithfully generate the set of all lifts of $\alpha$ by a sequence of these operations and their inverses. The pair $(\text{CS},\text{hol}_K)$ defines a bijection from the set of lifts of $\alpha$ into a subspace of $\R^2$
which is given as 
\[
  (\text{CS}(\widetilde \alpha_0),\text{hol}_K(\widetilde \alpha_0))+\{(i,2j)\in \Z^2\mid i \equiv j \mod 2\},
\] 
with $\widetilde \alpha_0$ being a fixed lift of $\alpha$. The above shows that the correspondence defined by
\begin{equation}\label{lifts-translates}
  \widetilde \alpha \mapsto U^{\text{CS}(\widetilde \alpha)-\text{CS}(\widetilde \alpha_0)}
  T^{\text{hol}_K(\widetilde \alpha)-\text{hol}_K(\widetilde \alpha_0)}\widetilde \alpha_0
\end{equation}
allows us to identify an index $4$ subgroup of the monomials in $\Z[U^{\pm 1},T^{\pm 1}]\widetilde \alpha_0$ with 
the lifts of $\alpha$. Motivated by this, we make the following.

\begin{definition}\label{defn:lift}
	An {\emph{honest lift}} of a (perturbed) flat connection $\alpha$ is a choice of path $\widetilde{\alpha}$ of connections mod gauge from $\alpha$ to the reducible $\theta$. We consider honest lifts up to homotopy relative to the endpoints. A {\em{lift}} of $\alpha$ is a monomial $U^{i}T^{j} \widetilde \alpha$ where $\widetilde \alpha$ is an honest lift. The set of all such lifts are denoted by $\sP_\alpha$. $\diamd$
\end{definition}

\noindent We extend $(\text{CS},\text{hol}_K)$ to any lift of $\alpha$ by setting
	  \begin{equation}\label{extend-CS-hol}
	    \text{CS}(U^{i}T^j\widetilde \alpha)=i+\text{CS}(\widetilde \alpha),\hspace{1cm}
	    \text{hol}_K(U^{i}T^j\widetilde \alpha)=j+\text{hol}_K(\widetilde \alpha). 
	  \end{equation}

Turning to the I-graded $\cS$-complex structure, the instanton gradings are defined by
\[
	\text{deg}_I(\widetilde\alpha) := \text{CS}(\widetilde\alpha) \in \R.
\]
and we extend this to $\widetilde C_\ast(Y,K;\Delta_\sS)$ using \eqref{extend-CS-hol}. The index of the ASD operator for $\widetilde\alpha$ determines $\widetilde{\text{gr}}(\widetilde\alpha)\in \Z$, which in this case is the expected dimension of the moduli space of instantons in the same homotopy class as $\widetilde\alpha$. We extend the $\Z$-grading $\widetilde{\text{gr}}$ to all of $\widetilde C_\ast(Y,K;\Delta_\sS)$ by setting $\widetilde{\text{gr}}(U^iT^j\widetilde\alpha)=4i + \widetilde{\text{gr}}(\widetilde \alpha)$. The reducible $\theta$ is identified with the constant path $\widetilde\theta$, with bigrading $(0,0)$. The component of $\alpha'$ in $d (\widetilde \alpha)$ is given by the following sum over $[A]\in \breve{M}(\alpha,\alpha')_0$:
\[
  \sum \epsilon(A) \Delta(A)(\widetilde \alpha)
\]
where $\Delta(A):\sP_\alpha\to \sP_{\alpha'}$ is defined by requiring that the concatenation of $A$ with $\widetilde \alpha'$ is homotopic to the path $\widetilde \alpha$. Fixing a lift for each generator of $\widetilde C_\ast(Y,K;\Delta_\sS)$ and using \eqref{lifts-translates} we can regard the above expression as an element of $\Z[U^{\pm 1},T^{\pm1 }]$. The other maps are described similarly. That $\widetilde C_\ast(Y,K;\Delta_\sS)$ is an I-graded $\cS$-complex follows from the additivity of topological energy, which implies
\[
	2\kappa(A)  =  \text{CS}(\widetilde\alpha) - \text{CS}(\widetilde \alpha')
\]
in this situation, and the fact that the energy $\kappa(A)$ is always non-negative for instantons, and is positive for instantons on the cylinder with $\alpha\neq \alpha'$.

A {\em morphism } $\widetilde \lambda:\widetilde C_\ast\to C_\ast'$ {\em of level $\delta\in\R$} of I-graded complexes is an $R[U^{\pm 1}]$-module homomorphism which is a morphism of $\cS$-complexes and satisfies
\begin{equation}\label{eq:morigraded}
	\widetilde \lambda \widetilde C_{i,j} \subset  \bigcup_{k\leqslant j+\delta}\widetilde C'_{i,k}
\end{equation}
When $\delta=0$ we simply call $\widetilde \lambda$ a morphism. Similarly, a chain homotopy of level $\delta$ of I-graded $\cS$-complexes is a chain homotopy of the underlying $\cS$-complexes which increases the instanton grading by at most $\delta$. We may also form tensor products, duals, and local equivalence classes of I-graded $\cS$-complexes.

\begin{prop}\label{prop:morphismigraded}
	Let $(W,S,c):(Y,K)\to (Y',K')$ be negative definite over $\sS$. Suppose no perturbations are required in the construction of $\widetilde \lambda_{(W,S,c)}$ from Proposition \ref{prop:morphism}. Then this map lifts to a morphism of level $2\kappa_\text{{\em min}}(W,S,c)$ of I-graded $\cS$-complexes.
\end{prop}

The lift is straightforward to define, following the construction of Proposition \ref{prop:morphism}. For instance, the contribution of $\alpha'$ in $\lambda(\widetilde \alpha)$ may be defined as a sum over $[A]\in M(W,S,c;\alpha,\alpha')_0$:
\[
	 \sum \epsilon(A) \Delta(A)(\widetilde \alpha)
\]
where $\Delta(A):\sP_\alpha\to \sP_{\alpha'}$ maps a lift $\widetilde \alpha$ of $\alpha$ to the lift $\widetilde \alpha'$ of $\alpha'$ satisfying
\begin{equation}\label{lift-relation}
	2\kappa(A)-2\kappa_\text{min}(W,S,c) = \text{CS}(\widetilde \alpha)-\text{CS}(\widetilde \alpha'), \hspace{0.5cm}\nu(A) = \text{hol}_{K'}(\alpha') - \text{hol}_K(\alpha).
\end{equation}
The other components are described similarly.
That the resulting morphism is of level $2\kappa_\text{min}(W,S,c)$ follows from the first identity in \eqref{lift-relation}. The above definitions can be adapted to the case that $R$ is an algebra over the ring $\Z[T^{\pm 1}]$ in a straightforward way.

The following lemma is a consequence of the discussion above.
\begin{lemma}\label{lift-morphisms}
	Suppose $\alpha$ and $\alpha'$ are generators of $ C_\ast(Y,K;\Delta_\sS)$ with honest lifts
	$\widetilde \alpha$ and $\widetilde \alpha'$. Then the matrix entry 
	$ \langle d ( \widetilde \alpha) , \widetilde\alpha'\rangle$ has the form
	\begin{equation}
	  U^{m}T^{2m}
	  \sum_{i\in \Z}a_iT^{4i}\in \Z[U^{\pm 1}, T^{\pm1}]
	\end{equation}
	where $m=(\widetilde{\rm {gr}}(\widetilde\alpha')-\widetilde{\rm {gr}}(\widetilde\alpha)-1)/4$ and $a_i$ is an integer for each $i$. 
	A similar result holds for other matrix entries of $\widetilde d$ and for a morphism $\widetilde \lambda_{(W,S,c)}$ associated to a negative definite pair; that is to say,
	any matrix entry is a Laurent polynomial of the above form for an appropriate choice of $m$.
\end{lemma}

In general we cannot get by without using perturbations. In the terminology of \cite{DS}, the I-graded $\cS$-complexes of Definition \ref{defn:igraded} are of ``level $0$''. More generally, an I-graded complex of level $\delta\geqslant 0$ is defined as in Definition \ref{defn:igraded}, except that in (ii) we allow $\widetilde d$ to increase the instanton grading by at most $\delta$. Then $\widetilde C(Y,K;\Delta_\sS)$ always has the structure of an I-graded $\sS$-complex of level some small $\delta\geqslant 0$ determined by the perturbation data. (Alternatively, one can always give it the structure of an I-graded complex, with the caveat that instanton gradings are tethered to a {\em perturbed} Chern--Simons functional.) 

An {\em enriched complex} is a sequence of I-graded $\cS$-complexes of levels $\delta_i\geqslant 0$ with $\lim_{i\to\infty}\delta_i=0$, where the complexes are related to one another by suitable morphisms. Associated to a sequence of admissible perturbation data going to zero is an enriched complex for $(Y,K)$, unique up to chain homotopy of enriched complexes. All invariants of $(Y,K)$ considered here may be derived from it. For more details see \cite[Section 7]{DS}.

\begin{remark}
	In this paper we typically work with I-graded $\cS$-complexes instead of enriched complexes. 
	In fact, most of the examples in our computations require no perturbations, and so do not 
	require the formalism of enriched complexes.
	The remaining computations involve the connected sum of knots such that each of the summands does not require any perturbation.  	
	Theorem \ref{thm:connectedsum} for connected sums holds at the level of enriched complexes, 
	see \cite[Theorem 7.20]{DS}. 
	So for the knots that appear in our computations this theorem provides provides I-graded $\cS$-complexes.
	$\diamd$
\end{remark}

\subsection{The concordance invariant $\Gamma_K$} \label{sec:gammadef}

We next review the definition of concordance invariants $\Gamma_{(Y,K)}^R$ from \cite{DS}. These invariants, similar to the case of the Fr\o yshov invariants, are obtained by applying an algebraic construction to the enriched complex for $(Y,K)$.

We first define the invariant $\Gamma_{\widetilde C}$ for an I-graded $\cS$-complex $\widetilde C=C_\ast \oplus C_{\ast-1}\oplus R[U^{\pm 1}]$ over $R[U^{\pm 1}]$. For $\alpha = \sum s_k \zeta_k \in C$ where $\zeta_k$ are homogeneous of distinct bigradings and $s_k\in R$, define $\text{deg}_I(\alpha)$ to be the maximum of the $\text{deg}_I(\zeta_k)$ such that $s_k\neq 0$. The invariant $\Gamma_{\widetilde C}$ is a function from the integers to $\overline{\R}_{\geqslant 0}:=\R_{\geqslant 0}\cup \infty$. For $k$ a positive integer, set
\[
	\Gamma_{\widetilde C}(k) := \text{inf}\left( \text{deg}_I(\alpha)\right) \in \overline{\R}_{\geqslant 0}
\]
where the infimum is over $\alpha \in C$ with $\widetilde{\text{gr}}(\alpha)=2k-1$ such that
\[
	d \alpha =0, \qquad k-1 = \min\left\{i\geqslant 0: \delta_1 v^i \alpha \neq 0  \right\}.
\]
For $k$ a non-positive integer we define
\[
	\Gamma_{\widetilde C}(k) := \max\left(\text{inf}\left( \text{deg}_I(\alpha)\right),0\right) \in\overline{\R}_{\geqslant 0}
\]
where the infimum is over $\alpha\in C$ with $\widetilde{\text{gr}}(\alpha)=2k-1$ such that there exist elements $a_0,\ldots,a_{-k}\in R[U^{\pm 1}]$ with $a_{-k}\neq 0$ and 
\[
	d \alpha  = \sum_{i=0}^{-k} v^i\delta_2(a_i).
\]

Let $\sS=R[U^{\pm 1}]$ where $R$ is an algebra over $\Z[T^{\pm 1}]$. If $\widetilde C_\ast(Y,K;\Delta_\sS)$ has the structure of an I-graded $\cS$-complex, then we define
\[
	\Gamma_{(Y,K)}^R := \Gamma_{\widetilde  C_\ast(Y,K;\Delta_\sS)}
\]
In the more general case, for an enriched complex, one should take an appropriate limit of the above construction. We refer to \cite[Section 7]{DS} for details.

When $R=\Z[T^{\pm 1}]$ we omit it from the notation, and write $\Gamma_{(Y,K)}$. Similarly, when $K\subset S^3$ we write $\Gamma_{K}^R$. For basic properties of these invariants see \cite[Theorem 7.24]{DS}.



\section{Computations for 2-bridge knots}\label{sec:comps}

In this section, we firstly recall some of the computations of \cite[Section 9]{DS} where we used the ADHM description of the moduli of instantons on $\R^4$ to study the $\cS$-complexes of 2-bridge knots. There we characterized the $\cS$-complex of the trefoil. In Subsections \ref{subsec:2bridgetorus} and \ref{subsec:doubletwist}, respectively, we consider two families of 2-bridge knots both of which generalize the trefoil example: 2-bridge torus knots and double twist knots. In Subsection \ref{subsec:connectedsums} we compute $\Gamma$-invariants for connected sums of double twist knots. Finally, in Subsection \ref{subsec:othercomps} we compute the $\Gamma$-invariants for some 2-bridge knots with 11 crossings.

\subsection{Background}\label{2-bridge}

Let $K_{p,q}$ denote the 2-bridge knot whose branched double cover is the lens space $L(p,q)$ with $p$ being an odd integer. Fix the orbifold metric on $S^3$ which is induced by the spherical metric on $L(p,q)$. The $\cS$-complex of $K_{p,q}$, defined using this orbifold metric and the trivial perturbation data, was studied in \cite[Section 9]{DS}, which we recall here. There are $(p+1)/2$ singular flat $SU(2)$ connections on $S^3\setminus K_{p,q}$ up to gauge, denoted $\xi^0,\xi^1,\ldots,\xi^{(p-1)/2}$ where $\xi^0=\theta$ is the unique reducible. Let $\sS=\Z[T^{\pm 1}]$. Then the underlying module of the $\cS$-complex $\widetilde C_\ast=\widetilde C_\ast(K_{p,q};\Delta_\sS)$ of $K_{p,q}$ is given by 
\[
  \widetilde C_\ast = C_\ast \oplus C_{\ast -1 } \oplus \sS, \qquad C_\ast = \bigoplus_{i=1}^{\frac{p-1}{2}} \sS\cdot \zeta^i
\]
where $\zeta^i$ is a lift of $\xi^i$. The remaining part of the $\cS$-complex structure on $\widetilde C_\ast$ is given by the maps $d,v:C_\ast\to  C_\ast$, $\delta_1:C_*\to \sS$, $\delta_2:\sS\to C_*$, which are defined using the unparametrized moduli spaces $\breve{M}(\xi^i,\xi^j)_d$ with dimension $d\leq 1$. 

All of the critical points and moduli spaces here are non-degenerate. Given an instanton $[A]\in \breve{M}(\xi^i,\xi^j)_{d}$, there are positive integers $k_1,k_2$ and $\epsilon_1,\epsilon_2\in \{\pm 1\}$ satisfying
\begin{equation}
		k_1 \equiv \varepsilon_1 i +  \varepsilon_2 j\;\; (\text{mod }p), \quad qk_2 \equiv -\varepsilon_1 i  + \varepsilon_2 j \;\;(\text{mod }p)\label{eq:k1k2}
\end{equation}
and which furthermore satisfy the following relations:
\begin{equation}
	d = N_1(k_1,k_2;p,q) + \frac{1}{2}N_2(k_1,k_2;p,q)-1, \qquad 2\kappa(A) = \frac{k_1k_2}{p}. \label{eq:filtdata}
\end{equation}
Here $N_1(k_1,k_2;p,q)$ is the number of $(a,b)\in \Z^2$ with $|a|<k_1$, $|b|<k_2$ solving
\begin{equation}
	a+qb \equiv 0 \pmod p \label{eq:congruence}
\end{equation}
and similarly $N_2(k_1,k_2;p,q)$ is the number of solutions $(a,b)\in \Z^2$ with either $|a|=k_1$, $|b|<k_2$ or $|a|<k_1$, $|b|=k_2$. Note that $(a,b)=(0,0)$ always satisfies \eqref{eq:congruence} and given any solution $(a,b)$ of \eqref{eq:congruence}, $(-a,-b)$ is also a solution. Thus $N_1(k_1,k_2;p,q)$ is an odd positive integer and $N_2(k_1,k_2;p,q)$ is an even non-negative integer.

The $0$-dimensional moduli space $\breve{M}(\xi^i,\xi^j)_{0}$ is non-empty only if there are $k_1$ and $k_2$ satisfying \eqref{eq:k1k2} such that the following hold:
\begin{equation}\label{N1N2-d0}
  N_1(k_1,k_2;p,q)=1,\hspace{1cm}N_2(k_1,k_2;p,q)=0.
\end{equation}
In fact, if there are such $k_1$ and $k_2$, then $\breve{M}(\xi^i,\xi^j)_{0}$ consists of two oppositely oriented points, related to each other by the so-called {\it flip symmetry}, which is an involution acting on the moduli spaces. The set of monopole numbers for this pair of singular instantons is $\{2,-2\}$ if $k_1$ and $k_2$ are both odd and $\{0\}$ otherwise. 

For any $\xi^i$, we assume the lift $\zeta^i$ satisfies $\text{hol}_K(\zeta^i)=0$. This is possible because $\text{hol}_K(\xi^i)$ is zero in $\R/\Z$. (However, note that this lift might not always be honest in the sense of Definition \ref{defn:lift}.) Then the matrix entry $\langle d\zeta^i,\zeta^j\rangle$, for $1\leq i,j\leq \frac{p-1}{2}$, is non-zero if and only if there are odd positive integers $k_1,k_2$ satisfying \eqref{eq:k1k2} and \eqref{N1N2-d0}. We obtain
\begin{equation}
	\langle d\zeta^i,\zeta^j\rangle  = \begin{cases} \pm (T^2-T^{-2}) & \exists {\text{ odd }} k_1,k_2\in \Z_{>0} {\text{ solving \eqref{eq:k1k2} and \eqref{N1N2-d0}}}  \\ 0 & \text{ otherwise}\end{cases}  \label{eq:d2bridge}
\end{equation}
By letting $j=0$ (resp. $i=0$) above, we obtain a similar characterization of the element $\delta_1(\zeta^i)$ (resp. $\langle \delta_2(1),\zeta^j\rangle$). 

There is a similar characterization for the $1$-dimensional moduli spaces $\breve{M}(\xi^i,\xi^j)_{1}$, which are the geometrical input in the definition of the map $v$. The space $\breve{M}(\xi^i,\xi^j)_{1}$ is non-empty if and only if there are positive integers $k_1$ and $k_2$ satisfying \eqref{eq:k1k2} such that
\begin{equation}\label{N1N2-d1}
  N_1(k_1,k_2;p,q)=1,\hspace{1cm}N_2(k_1,k_2;p,q)=2.
\end{equation}
Moreover, as before, the set of monopole numbers for instantons in $\breve{M}(\xi^i,\xi^j)_{1}$ is $\{2,-2\}$ if $k_1$ and $k_2$ are odd and is $\{0\}$ otherwise. Thus we have
\[
  \langle v\zeta^i,\zeta^j\rangle= T^2A^+ + T^{-2}A^- + B
\]
where either $B=0$ or $A^+=A^-=0$ depending on whether the relevant integers $k_1$ and $k_2$ for the moduli space $\breve{M}(\xi^i,\xi^j)_{1}$ are odd or not. It is shown in \cite{DS} that the homology of $\widetilde C_\ast$ after evaluating $T$ at $1$ and working with coefficient ring $\Z$ is isomorphic to $I^\natural(K_{p,q};\Z)$. This is free abelian of rank $p$, by \cite[Corollary 1.6]{KM:unknot}. As $\widetilde C_\ast$ has rank $p$, we must have $A^-=-A^+$, and $B=0$. Consequently, we have:
\begin{eqnarray}\label{eq:v2bridge}
	\lefteqn{\langle v\zeta^i,\zeta^j\rangle  = n(T^2-T^{-2}), \; n\in \Z, \text{and}}\\[1.5ex] &&  \hspace{1.5cm}  n\neq 0 \implies  \exists {\text{ odd }} k_1,k_2\in \Z_{>0} {\text{ solving \eqref{eq:k1k2} and \eqref{N1N2-d1}}}  \nonumber
\end{eqnarray}

The congruences \eqref{eq:k1k2} can be also used to determine Floer and instanton gradings. Let $k_1$ and $k_2$ be any pair of positive integers that solve \eqref{eq:k1k2} where $0\leq i, j\leq \frac{p-1}{2}$. Then there are lifts $\zeta^i$ and $\zeta^j$ of $\xi^i$ and $\xi^j$ respectively such that their $\Z\times \R$ bigradings satisfy
\begin{eqnarray} \lefteqn{\left({\widetilde{\text{gr}}}(\zeta^i),\deg_I(\zeta^i)\right)-\left({\widetilde{\text{gr}}}(\zeta^j),\deg_I(\zeta^j)\right)= } \nonumber \\[1.5ex] & & \hspace{1.5cm}(N_1(k_1,k_2;p,q) + \frac{1}{2}N_2(k_1,k_2;p,q),\frac{k_1k_2}{p})\label{eq:2brbigr}
 \end{eqnarray}
We always assume $\zeta^0$ is the constant path for $\theta$, so that its bigrading is $(0,0)$. Having fixed $\zeta^0$, if $j=0$ (resp. $i=0$), then imposing \eqref{eq:2brbigr} for some $k_1,k_2$ determines the bigrading of $\zeta^i$ (resp. $\zeta^j$). If instead $i$ and $j$ are both non-zero, it determines $\zeta^i$ and $\zeta^j$ up to simultaneously multiplying them by $U^k$ for some $k\in \Z$, which alters bigradings by $(4k,k)$. It is important to note that the rule \eqref{eq:2brbigr} depends on $k_1$, $k_2$.

Now consider $\widetilde C_\ast(K_{p,q};\Delta_\sS)$ as an I-graded $\cS$-complex over $\sS=\Z[U^{\pm 1},T^{\pm 1}]$. Suppose $k_1,k_2$ solve \eqref{eq:k1k2} and \eqref{N1N2-d0}, and $i,j\neq 0$. If we choose lifts $\zeta^i$ and $\zeta^j$ satisfying \eqref{eq:2brbigr}, then \eqref{eq:d2bridge} continues to hold. Letting one of $i$ or $j$ be zero gives an analogous statement for $\delta_1$ and $\delta_2$. Next, suppose $i,j\neq 0$ and $k_1,k_2$ solve \eqref{eq:k1k2} and \eqref{N1N2-d1}. If in this case we choose lifts $\zeta^i$ and $\zeta^j$ satisfying \eqref{eq:2brbigr}, then \eqref{eq:v2bridge} continues to hold. (Different choices of lifts in any of these situations might introduce powers of $U$ into the expressions.)

\begin{remark}
The authors expect that a more thorough investigation of the equivariant ADHM description of instantons on $\R^4$ in this context should allow for the direct computation of the $v$-maps for 2-bridge knots. This would yield a complete combinatorial description of the I-graded $\cS$-complexes for 2-bridge knots, and their $\Gamma$-invariants. $\diamd$
\end{remark}

\subsection{The torus knots $T_{2,2k+1}$}\label{subsec:2bridgetorus}

In this section we determine the full structure of the invariants defined in \cite{DS} for the $(2,2k+1)$ torus knots $T_{2,2k+1}$ where $k$ is a positive integer.

\begin{prop}\label{prop:22kp1} The I-graded $\cS$-complex 
	$\widetilde C_\ast=\widetilde C_\ast(T_{2,2k+1};\Delta_\sS)$ is given by
	\[
	  \widetilde C_\ast = C_\ast \oplus C_{\ast -1 } \oplus \sS, \qquad 
	  C_\ast = \bigoplus_{i=1}^{k} \sS\cdot \zeta^i.
	\]
	The differential $\widetilde d$ has the components $d=\delta_2=0$ and
	\[
	  \delta_1\left(\zeta^{i}\right) = \begin{cases} (T^2-T^{-2}) & i=1 \\ 0 & 2\leqslant i \leqslant k\end{cases}
	\]
	\[
	  v\left( \zeta^{i} \right)=\begin{cases} (T^2-T^{-2})\zeta^{i-1} & 2\leqslant i \leqslant k\\ 0 & i=1 \end{cases}
	\]
	The $\Z\times \R$ bigrading of $\zeta^i$ is given by $(2i-1, i^2/(2k+1))$.
\end{prop}

\begin{proof}
	Most of the statement follows from the content of Subsection \ref{2-bridge}.
	The torus knot $T_{2,2k+1}$ is the 2-bridge knot $K_{p,-1}$ where $p=2k+1$. In this case, the congruence \eqref{eq:congruence} is $a\equiv b$ (mod $p$). 
	The constraints $N_1(k_1,k_2;p,q)=1$, $N_2(k_1,k_2;p,q)=0$ in \eqref{N1N2-d0} imply that $(0,0)$ is the only 
	solution to $a\equiv b$ (mod $p$) with $|a|\leqslant k_1$, $|b|\leqslant k_2$, with one of these two inequalities strict. This is true 
	if and only if $k_1=k_2=1$, and from \eqref{eq:k1k2}, this in turn happens if and only if $j=0$, $i=1$, and 
	$\epsilon_1=1$. Thus $d=\delta_2=0$, and 
	$\delta_1\left(\zeta^{i}\right)$ is equal to $\pm (T^2-T^{-2})$ if $i=1$ and is zero otherwise.
	Here we have chosen the lift $\zeta^1$ of $\xi^1$ satisfying \eqref{eq:2brbigr}, so that it has $\Z\times \R$ bigrading $(1,\frac{1}{p})$.
	
	We now turn to the map $v$. The equations \eqref{eq:congruence} and \eqref{N1N2-d1} hold simultaneously
	if and only if one of $k_1$, $k_2$ is equal to $1$ and the other is an odd integer $>1$ and $<p$. This implies 
	that $j=i-1$. Choose the lifts $\zeta^i$ of $\xi^i$ for $2\leq i \leq k$ satisfying \eqref{eq:2brbigr}:
	\begin{equation}\label{rel-bigrading}
	  \left({\widetilde{\text{gr}}}(\zeta^i),\deg_I(\zeta^i)\right)-\left({\widetilde{\text{gr}}}(\zeta^{i-1}),\deg_I(\zeta^{i-1})\right)=
	  (2,\frac{(2i-1)}{p}).
	\end{equation}
	This implies that the $\Z\times \R$ bigrading of $\zeta^i$ is given by $(2i-1, i^2/(2k+1))$. In the basis $\zeta^1,\ldots,\zeta^k$ the map $v$ has the form:
\[	  v= (T^2-T^{-2})\left[\begin{array}{cccccc} 0 & a_1 & 0 & \cdots  & 0 \\ 
										0 &  0  & a_2 &   & \vdots \\
										 \vdots &    &  & \ddots  & \\
										 & &   & & a_{k-1} \\ 
										0 &  &  \cdots &  & 0 \end{array} \right]. \qquad
	\]
	To compute the entries $a_i$ we use Theorem \ref{thm:htwistedsign}, which asserts when $T^2-T^{-2}\neq 0$ that 
	\[
	  h_\sS(K_{p,-1})=-\sigma(K_{p,-1})/2=k.
	\]
	Proposition \ref{h-reinterpret} implies there is an $\alpha\in C_\ast$ with $\delta_1v^{k-1}(\alpha)\neq 0$, 
	which happens only if $a_1,\ldots,a_{k-1}$ are nonzero. This holds over any characteristic, and so $a_i=\pm 1$. 
	After possibly replacing $\zeta^i$ with $-\zeta^i$ we may assume $a_i=1$, and $v$ has the claimed form.
\end{proof}

\begin{cor}
	For any $k$, the $\Gamma$-invariant of $T_{2,2k+1}$ is given as
	\[
		\Gamma_{T_{2,2k+1}}(i) = \begin{cases} 0, & i \leqslant 0\\ \frac{i^2}{2k+1}, & 1\leqslant i \leqslant k\\ \infty, 
		& i\geqslant k+1 \end{cases}
	\]
\end{cor}

The case $k=1$ of Proposition \ref{prop:22kp1} recovers the trefoil complex determined in \cite[Section 9]{DS}. This simple type of complex appears in the sequel, and so we make the following.

\begin{definition}\label{defn:atomcomplex}
	For any positive real number $t$ we let $\widetilde \fC(t)=\fC_\ast \oplus \fC_{\ast-1}\oplus \sS$ be the I-graded $\cS$-complex where $\fC_\ast$ is freely generated by a single generator $\zeta$ with bigrading
\[
	(\widetilde{\text{gr}}(\zeta), \text{deg}_I(\zeta))=(1, t).
\]
The differential has $\delta_1(\zeta)=T^{2}-T^{-2}$ while $d$, $v$, $\delta_2$ are zero. $\diamd$
\end{definition}

\noindent The trefoil $T_{2,3}$ has an I-graded $\cS$-complex isomorphic to $\widetilde \fC(\frac{1}{3})$. We note that the Fr\o yshov invariant of $\widetilde\fC(t)$ is equal to $1$ if $T^4\neq 1$ in $\sS$ and is otherwise zero. Recall that the undecorated $\Gamma$-invariant for an I-graded $\cS$-complex is to be understood as defined for the coefficient ring $\sS=\Z[T^{\pm 1}]$. The $\Gamma$-invariant of the complex $\widetilde\fC(t)$ is given by:
\[
		\Gamma_{\widetilde{\fC}(t)}(i) = \begin{cases} 0, & i \leqslant 0\\  t, & i=1 \\ \infty, 
		& i\geqslant 2 \end{cases}
\]

\subsection{Double twist knots}\label{subsec:doubletwist}

For a pair of positive integers $m,n$, let $D_{m,n}$ denote the 2-bridge knot with parameters 
\[
  (4mn-1,2n).
\]
In particular, the branched double cover of $D_{m,n}$ is the lens space $L(4mn-1,2n)$. The signature of this knot is equal to $-2$, and its slice genus is equal to $1$. The main result of this section is the following.

\begin{prop}\label{tilde-complex-dbletwist}
	The I-graded $\cS$-complex of the double twist knot $D_{m,n}$ is locally equivalent to the I-graded $\cS$-complex $\widetilde \fC(t)$ of Definition \ref{defn:atomcomplex} where $t=\frac{(2m-1)(2n-1)}{4mn-1}$.
\end{prop}

Recall from Subsection \ref{sec:enriched} that I-graded $\cS$-complexes $(\widetilde C, \widetilde d, \chi)$ and $(\widetilde C', \widetilde d', \chi')$ are locally equivalent if there are level $0$ morphisms $\widetilde \lambda:\widetilde C\to \widetilde C'$ and $\widetilde \lambda':\widetilde C'\to \widetilde C$. We begin the proof of Proposition \ref{tilde-complex-dbletwist} with several combinatorial lemmas.

\begin{figure}[t]
\centering
\includegraphics[scale=2.25]{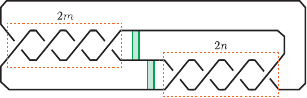}
\caption{The double twist knot $D_{m,n}$ has $2m$ and $2n$ half twists in the indicated boxes. The example shown has $m=n=2$, and is the two bridge knot $(15,4)$. The two indicated band moves induce a genus 1 cobordism from $D_{m,n}$ to the unknot.}
\label{fig:dmn}
\end{figure}

\begin{lemma}\label{trivial-A}
	For integers $p$, $q$ and positive odd integers $k_1$, $k_2$ define
	\[
  	  A(k_1,k_2;p,q)=\left\{(a,b)\in \Z^2 \mid a+qb \equiv 0 \;(\text{{\em mod }} p),\, 
	  \begin{array}{c}|a|<k_1, |b|\leq k_2 \text{{ \em{or}}}  \\ |a|\leq k_1, |b|<k_2 \phantom{\text{{ \em{or}}}}\end{array}\right\}.
	\]
	The set $A(k_1,k_2;4mn-1,2n)$ has only $(0,0)$ if and only if 
	$k_1\leq 2n-1$ and $k_2\leq 2m-1$.
\end{lemma}
\begin{proof}
	Note that the pairs $(-1,2m)$ and $(2n,-1)$ satisfy
	\[
	  a+2n b \equiv 0\pmod{4mn-1}.
	\]
	Thus $A(k_1,k_2;4mn-1,2n)$ has a non-trivial element unless 
	$k_1\leq 2n-1$ and $k_2\leq 2m-1$.  It is straightforward to check that for such choices of $(k_1,k_2)$, the 
	set $A(k_1,k_2;4mn-1,2n)$ does not have any non-trivial element.
\end{proof}

\begin{lemma}\label{alm-trivial-A}
	The set $A(k_1,k_2;4mn-1,2n)$ consists of three elements $(0,0)$, $(a_0,b_0)$ and $(-a_0,-b_0)$ with $|a_0|=k_1$ or $|b_0|=k_2$ if and only if 
	$k_1=1$ and $2m+1 \leq k_2 \leq 4mn-2m-1$ or $2n+1 \leq k_1 \leq 4mn-2n-1$ and $k_2=1$.
\end{lemma}
\begin{proof}
	Lemma \ref{trivial-A} implies that either $k_1\geq 2m+1$ or $k_2\geq 2n+1$. Therefore, either $(a_0,b_0)=(-1,2m)$ or $(a_0,b_0)=(2n,-1)$ belongs to $A(k_1,k_2;4mn-1,2n)$. In the first case, $k_1=1$ and $k_2\leq4mn-2m-1$.
	In the latter case, $k_2=1$ and $k_1\leq4mn-2n-1$.
\end{proof}

\begin{lemma}\label{matching-d}
	For any non-zero $l\in \Z/{(4mn-1)}$, there exists a unique pair of positive odd integers $k_1$ and $k_2$ and $\epsilon\in \{\pm1\}$ such that $k_1\leq 2n-1$, $k_2\leq 2m-1$ and exactly one of the following congruences modulo $4mn-1$ hold:
	\[
	  k_1+2n k_2 \equiv \epsilon l, \hspace{1cm} k_1-2n k_2 \equiv \epsilon l.
	\]
	 In the case that $l=0$, and under the same constraints on $k_1$ and $k_2$, the only solutions for the first equation are $(k_1,k_2,\epsilon)=(2n-1,2m-1,+1)$ and $(2n-1,2m-1,-1)$, and the second congruence does not have any solution.
\end{lemma}
\begin{proof}
	We obtain $4mn$ elements of $ \Z/{(4mn-1)}$ by considering numbers of the form 
	\[
	   \pm k_1\pm 2n k_2
	\]
	for a pair of positive odd integers $k_1$ and $k_2$ with $k_1\leq 2n-1$, $k_2\leq 2m-1$. These numbers represent distinct elements of $ \Z/{(4mn-1)}$ with one exception:
	\[
	  (2n-1)+ 2n(2m-1)  \equiv -(2n-1)- 2n(2m-1).
	\]	
	This verifies the desired claim.
\end{proof}

\begin{remark}
	All of the congruences in this subsection will be taken modulo $4mn-1$, and so we will often omit this from the notation. $\diamd$
\end{remark}

Following the construction of Subsection \ref{2-bridge}, an I-graded $\cS$-complex for the knot $D_{m,n}$ with coefficient ring $\sS$ has underlying $\sS$-module given by 
\[
	  \widetilde C_\ast = C_\ast \oplus C_{\ast -1 } \oplus \sS, \qquad 
	  C_\ast = \bigoplus_{i=1}^{2mn-1} \sS\cdot \zeta^i
\]
where the $\zeta^i$ are lifts of the connections $\xi^i$. We will only need to specify the precise lifts of a few generators below. As before the lift $\zeta^0$ is in bigrading $(0,0)$.

Lemma \ref{trivial-A} can be used to characterize the maps $d$, $\delta_1$ and $\delta_2$ for this I-graded $\cS$-complex. If $\delta_1(\zeta^i)$ is non-zero, it is equal to $\pm (T^2-T^{-2})$ and this happens if and only if there are positive odd integers $k_1$ and $k_2$ and $\varepsilon_1\in\{\pm 1\}$ such that 
\begin{equation}
	k_1 \equiv \varepsilon_1 i, \quad 
	2nk_2 \equiv -\varepsilon_1 i \label{eq:k1k2-dbletwist},
\end{equation}
$N_1(k_1,k_2;4mn-1,2n)=1$ and $N_2(k_1,k_2;4mn-1,2n)=0$. The last two identities mean $A(k_1,k_2;4mn-1,2n)$ has only the element $(0,0)$. Thus by Lemma \ref{trivial-A}, non-vanishing of $\delta_1(\zeta^i)$ is equivalent to existence of a solution for \eqref{eq:k1k2-dbletwist} with $k_1\leq 2n-1$ and $k_2\leq 2m-1$. The only solution is $(k_1,k_2)=(2n-1,2m-1)$ which implies that $i=2n-1$. We fix the lift $\zeta^{2n-1}$ using \eqref{eq:2brbigr}, so that its bigrading is equal to
\[
  (1, \frac{(2m-1)(2n-1)}{4mn-1}).
\]
Another application of Lemma \ref{trivial-A} shows that $\delta_2(1)\in C_*$ has a contribution from $\zeta^j$ if and only if there are positive odd integers $k_1$ and $k_2$ and $\varepsilon_2\in \{\pm 1\}$ such that 
\begin{equation*}
	k_1 \equiv \varepsilon_2 j, \quad 
	2nk_2 \equiv \varepsilon_2 j 
\end{equation*}
with $k_1\leq 2n-1$ and $k_2\leq 2m-1$. Since there is no solution to the congruence $k_1-2nk_2 \equiv 0$ satisfying these inequlaities, the map $\delta_2$ vanishes.

Finally, we may apply similar arguments making use of Lemmas \ref{trivial-A} and \ref{matching-d} to choose lifts such that the generators of $C_*$ can be partitioned as
\[
  \{\zeta^{i_1},\zeta^{i_2},\dots,\zeta^{i_{mn-1}}\}\cup \{\zeta^{j_1},\zeta^{j_2},\dots,\zeta^{j_{mn-1}}\}\cup \{\zeta^{2n-1}\}
\]
such that $d\zeta^{i_k}=(T^2-T^{-2})\zeta^{j_k}$, $d\zeta^{j_k}=0$ and $d\zeta^{2n-1}=0$. Of course, we can be more specific about the values $i_k$ and $j_k$. However, the exact values of most of these integers are not relevant below. We just single out two of the relations which shall be useful later:
\begin{equation}\label{eq:twochainidentities}
  d\zeta^{2n-2}=(T^2-T^{-2})\zeta^{1},\hspace{1cm}d\zeta^{4n-1}=(T^2-T^{-2})\zeta^{2n}.
\end{equation}
These identities hold as a result of the following respective pairs of congruences:
\begin{equation}\label{bi-gr-2n-2}
  2n-3 \equiv (2n-2)-1, \hspace{.2cm}
	2n(2m-1) \equiv  -(2n-2)-1, \text{ and} 
\end{equation}
\begin{equation}\label{bi-gr-2n-1}
  2n-1 \equiv (4n-1)-2n , \hspace{.2cm}
	2n(2m-3) \equiv  -(4n-1)-2n 
\end{equation}
Note that the first relation in \eqref{eq:twochainidentities} only holds under the assumption that $n\geqslant 2$, while the second relation holds only when $m\geqslant 2$.

Lemma \ref{alm-trivial-A} may also be used to obtain some information about the map $v$ associated to the I-graded $\cS$-complex $\widetilde C_*$. Based on the discussion of Subsection \ref{2-bridge}, the coefficient $\langle v\zeta^i,\zeta^j\rangle$ is non-zero only if there are positive odd integers such that 
\[
  	k_1 \equiv \varepsilon_1 i+\varepsilon_2j , \quad \quad
	2nk_2 \equiv -\varepsilon_1i+\varepsilon_2j  ,
\]
$N_1(k_1,k_2;4mn-1,2n)=1$ and $N_2(k_1,k_2;4mn-1,2n)=2$. Lemma \ref{alm-trivial-A} classifies all $k_1$ and $k_2$ which satisfy the latter two identities. Therefore, this lemma allows us to constrain the non-zero terms among $\langle v\zeta^i,\zeta^j\rangle$. For example, for $m,n\geq 2$, $v\zeta^{2n-1}=0$ and the only possibilities for integers $i$ that have $\langle v\zeta^i,\zeta^{2n-1}\rangle\neq 0$ are $\zeta^{2n-2}$ and $\zeta^{4n-1}$ and the relevant integers $k_1$ and $k_2$ for each of these two cases are given, respectively, by the pairs of congruences
\begin{equation*}
  1 \equiv -(2n-2)+(2n-1) , \hspace{.2cm}
	2n(4mn-6m+1) \equiv  (2n-2)+(2n-1), \text{ and}
\end{equation*}
\begin{equation*}
  4mn-6n+1 \equiv -(4n-1)-(2n-1), \hspace{.2cm}
	2n \equiv  (4n-1)-(2n-1).
\end{equation*}
The above identities, with \eqref{bi-gr-2n-2} and \eqref{bi-gr-2n-1}, and the computation of the bigrading of $\zeta^{2n-1}$, imply that the bigradings of $\zeta^{2n-2}$, $\zeta^{4n-1}$, $\zeta^{1}$ and $\zeta^{2n}$ are respectively equal to
\[
  (3,\tfrac{2(4m-1)(n-1)}{4mn-1}),\hspace{.4cm}(3,\tfrac{2(m-1)(4n-1)}{4mn-1}),\hspace{.4cm}(2,\tfrac{4mn-2m-1}{4mn-1}),
  \hspace{.4cm}(2,\tfrac{4mn-2n-1}{4mn-1}).
\]
Note that we are choosing the lift $\zeta^1$ (resp. $\zeta^{2n}$) such that \eqref{eq:2brbigr} holds with $i=1$ (resp. $i=2n$) and \eqref{N1N2-d0} holds. Now we are ready to prove Proposition \ref{tilde-complex-dbletwist}.

\begin{proof}[Proof of Proposition \ref{tilde-complex-dbletwist}]
Let $\widetilde \fC=\fC(t)$ be as in Definition \ref{defn:atomcomplex} with $t=\frac{(2m-1)(2n-1)}{4mn-1}$. Recall that $\fC_\ast$ has a single generator $\zeta$ in bigrading $(1,t)$.

	In the case that $m,n\geq 2$, define the morphism $\widetilde \lambda:\widetilde C_*\to \widetilde \fC_*$ using the components $\lambda,\mu:C_*\to \fC_*$, $\Delta_1:C_*\to \sS$ and $\Delta_2:\sS\to \fC_*$ which 
	have the non-zero matrix entries
	\[
	  \langle \lambda \zeta^{2n-1},\zeta\rangle=1,\hspace{0.5cm}\langle \mu \zeta^{1},\zeta\rangle=\frac{\langle v \zeta^{2n-2},\zeta^{2n-1}\rangle}{T^{-2}-T^2},
	  \hspace{0.5cm}\langle \mu \zeta^{2n},\zeta\rangle=\frac{\langle v \zeta^{4n-1},\zeta^{2n-1}\rangle}{T^{-2}-T^2}.
	\]
	Furthermore, the component $\eta\in \sS$ of $\widetilde \lambda$ is equal to $1$. Using the above discussion of the bigradings for the generators of $C_*$, one can see easily that $\widetilde \lambda$ is a morphism of I-graded $\cS$-complexes.
	Similarly, we have a morphism $\widetilde \lambda':\widetilde \fC_*\to \widetilde C_*$ defined using the maps 
	$\lambda',\mu':\fC_*\to C_*$, $\Delta_1':\fC_*\to \sS$ and $\Delta_2':\sS\to C_*$ with the only non-zero entry $\langle \lambda'\zeta,\zeta^{2n-1} \rangle=1$. Again, the component $\eta'\in \sS$ of $\widetilde \lambda'$ is equal to $1$.
	Here to show that $\widetilde \lambda'$ is a morphism we need the fact that $v\zeta^{2n-1}=0$.
	This shows that $\widetilde C_*$ and $\widetilde \fC_*$ are locally equivalent as I-graded $\cS$-complexes.

	If $m=1$ and $n\geqslant 2$, then the above proof goes through with the simplification that the generator $\zeta^{4n-1}$ is omitted. If instead $n=1$ and $m\geqslant 2$, we omit $\zeta^{2n-1}$. Alternatively, we note that $D_{1,n}=D_{n,1}$. Finally, the case $D_{1,1}$ is the trefoil, whose complex we have already determined; this is the case where only $\zeta^1$ and $\zeta^0$ appear.
\end{proof}

\begin{remark}
	For general $m,n\in \Z_{>0}$, the $\cS$-complexes $\widetilde C_*$ and $\widetilde\fC_*$ are not $\cS$-chain homotopy equivalent to each other as enriched complexes because the homology of the graded complexes associated to the instanton
	grading for these two $\cS$-complexes have ranks $4mn-1$ and $3$. After forgetting the I-gradings, however, one can see that the morphisms $\widetilde \lambda$ and $\widetilde \lambda'$ in the proof of the above 
	lemma are $\cS$-chain homotopy equivalences once $T^2-T^{-2}$ is a unit in $\sS$. Later we shall see this statement is not special to the knot $D_{m,n}$, and that for any pair $(Y,K)$ with $\sigma=-2$, 
	the $\cS$-chain homotopy equivalence of the $\cS$-complex associated to $(Y,K)$ after forgetting the Chern-Simons filtration, over a ring $\sS$ with $T^2-T^{-2}$ a unit, is $\cS$-chain homotopy equivalent to $\widetilde\fC_*$. $\diamd$
\end{remark}

\begin{cor}
	For positive integers $m,n$, the $\Gamma$-invariant of $D_{m,n}$ is given as
	\[
		\Gamma_{D_{m,n}}(i) = \begin{cases} 0, & i \leqslant 0\\ \frac{(2m-1)(2n-1)}{4mn-1}, & i=1\\ \infty, 
		& i\geqslant 2 \end{cases}
	\]
\end{cor}

\begin{figure}[t]
\centering
\begin{tikzpicture}
\draw (0,0) node {$\zeta^0$};
\node [below, purple] at (0,-0.25) {$(0,0)$};
\draw (1.5,0.0) node {$\zeta^3$};
\node [below, purple] at (1.5,-0.25) {$(1, \frac{9}{15})$};
\draw [->, blue] (1.15,0.0) -- (0.35,0.0); 
\draw (3.0,0.0) node {$\zeta^1$};
\node [below, purple] at (3.0,-0.25) {$(2, \frac{11}{15})$};
\draw (3.0,-1.5) node {$\zeta^4$};
\node [below, purple] at (3.0,-1.75) {$(2, \frac{11}{15})$};
\draw (4.5,0.0) node {$\zeta^2$};
\node [below, purple] at (4.5,-0.25) {$(3, \frac{14}{15})$};
\draw [->, blue] (4.15,0) -- (3.35,0); 
\draw (4.5,-1.5) node {$\zeta^7$};
\node [below, purple] at (4.5,-1.75) {$(3, \frac{14}{15})$};
\draw [->, blue] (4.15,-1.5) -- (3.35,-1.5); 
\draw (6.0,0.0) node {$\zeta^5$};
\node [below, purple] at (6.0,-0.25) {$(4, \frac{20}{15})$};
\draw (7.5,0.0) node {$\zeta^6$};
\node [below, purple] at (7.5,-0.25) {$(5, \frac{21}{15})$};
\draw [->, blue] (7.15,0.0) -- (6.35,0.0); 
\end{tikzpicture}
\caption{Above are generators of $C_\ast(D_{2,2};\Delta_\sS)$ along with the reducible $\zeta^0$. Bigradings are listed under each corresponding generator. Arrows are multiplication by $T^2-T^{-2}$ and represent $d$ except for the left-most arrow, which is $\delta_1$. The component $\delta_2$ is zero. While $v$ might be nonzero, there is no cycle that $v$ maps to $\zeta^3$. The local chain equivalence to $\widetilde \fC(t)$ with $t=\frac{9}{15}$ leaves only $\zeta^0\leftarrow \zeta^3$ in this diagram.}
\label{fig:74complex}
\end{figure}

A simple example that illustrates the above computations is the case of $D_{2,2}$, which is the same as $7_4$ and the $2$-bridge knot $(15,4)$, featured in Theorem \ref{thm:clasp}. See Figure \ref{fig:74complex} for an I-graded $\cS$-complex for this example.

\subsection{Connected sums}\label{subsec:connectedsums}

Given I-graded $\cS$-complexes $\widetilde C_\ast$ and $\widetilde C_\ast'$ for knots $K$ and $K'$, the connected sum formula of Theorem \ref{thm:connectedsum} applies in this setting and gives $\widetilde C_\ast\otimes \widetilde C_\ast'$ as an I-graded $\cS$-complex for $K\# K'$. The same is true if we are working with local equivalence classes of I-graded $\cS$-complexes. In particular, the $\Gamma$-invariant for $K\# K'$ can be computed from the local equivalence classes associated to $K$ and $K'$. We first carry this out for a $k$-fold connected sum of a knot that has an I-graded $\cS$-complex locally chain equivalent to $\widetilde{\fC}(t)$ from Definition \ref{defn:atomcomplex}, for some real number $t>0$.

\begin{lemma}\label{lemma:ctconnectedsum}
The $\Gamma$-invariant for the $k$-fold tensor product of $\widetilde \fC(t)$ is given by
	\[
		\Gamma_{\widetilde \fC(t)^{\otimes k}}(i) = \begin{cases} 0, & i \leqslant 0\\ it, & 1\leq i \leq k\\
		 \infty, & i\geqslant k+1 \end{cases}
	\]
\end{lemma}

\begin{proof}
	In the following we use the decomposition for tensor products of $\cS$-complexes as described in \cite[Section 4.5]{DS}, and refer the reader there for more details. 

Let $\widetilde \fC=\widetilde \fC(t) = \fC_\ast\oplus \fC_{\ast-1}\oplus \sS$. Write $\zeta$ for the generator in $\fC_\ast$ of bigrading $(1,t)$ and $\underline\zeta$ its copy in $\fC_{\ast-1}$ of bigrading $(2,t)$. The tensor product $\widetilde \fC^{\otimes k}$ splits as $C_\ast(k)\oplus C_{\ast-1}(k)\oplus \sS$ where $C_\ast(k)$ is the module of rank $(3^k-1)/2$ generated by elements of the form
\begin{equation}
	\sigma = \sigma_1 \otimes \sigma_2 \otimes \cdots \otimes \sigma_k\label{eq:sigmatensor}
\end{equation}
where $\sigma_i\in\{1,\zeta,\underline{\zeta}\}$, not all $\sigma_i$ are equal to $1$, and the largest $i$ such that $\sigma_i\neq 1$ has $\sigma_i=\zeta$. The number of appearances of $\zeta$ and $\underline{\zeta}$ are respectively denoted by $\fb(\sigma)$ and $\ft(\sigma)$. The bigrading of $\sigma$ is given by
\[
	(\fb( \sigma)+2\ft( \sigma),(\fb( \sigma)+\ft( \sigma))t).
\]
Note the $\sS$-module $C_*(k)$ can be alternatively described inductively as 
	\begin{equation}\label{decomp-Ck}
		C_*(k)=\widetilde\fC_*\otimes C_*(k-1)\oplus \sS\langle\zeta\otimes 1\otimes 1\otimes \dots\otimes 1\rangle.
	\end{equation}
The maps $d_k,v_k:C_*(k)\to C_*(k)$, $\delta_{1,k}:C_*(k)\to \sS$ and 
	$\delta_{2,k}:\sS\to C_*(k)$ associated to the $\cS$-complex structure on $\widetilde \fC^{\otimes k}$ can be described explicitly as follows. First, $v_k( \sigma)$, for $\sigma$ as in \eqref{eq:sigmatensor}, is non-zero if the smallest $i$ that $\sigma_i\neq 1$ is $\underline \zeta$, and in this case we have
\[
	v_k(\sigma) =  (T^{2}-T^{-2})\underbrace{1\otimes \cdots \otimes 1}_{i}\otimes \sigma_{i+1}\otimes \dots  \otimes \sigma_{n}.
\]
The map $\delta_{2,k}$ vanishes and $\delta_{1,k}(\sigma)$ is equal to $T^2-T^{-2}$ if  $\fb( \sigma)=1$ and $\ft( \sigma)=0$ and is zero otherwise. The map $d_k$ sends the rank one summand in \eqref{decomp-Ck} to zero. A generator in the first summand has the form $\sigma_1\otimes \sigma'$  for $\sigma'\in C_*(k-1)$. For $\sigma_1\in\{1,\zeta,\underline{\zeta}\}$, we have
	\[
	  d_k(\zeta \otimes  \sigma')=-\zeta \otimes d_{k-1}(\sigma')-\zeta \otimes v_{k-1}(\sigma')-\zeta \otimes \delta_{1,k-1}( \sigma')+
	  (T^2-T^{-2})1\otimes \sigma',
	\]
	\[
	   d_k(\underline \zeta \otimes  \sigma')=\underline \zeta \otimes d_{k-1} \sigma',\hspace{1cm} d_k(1\otimes \sigma')=1\otimes d_{k-1} \sigma'.
	\]
The above description of $\widetilde \fC^{\otimes k}$ implies that for any $1\leq i\leq k$, $\sigma$ satisfies the constraint
	\[
	  i-1=\min\{j\mid \delta_{1,k}v_k^j(\sigma)\neq 0\}
	\]
	if and only if $\fb( \sigma)=1$ and $\ft( \sigma)=i-1$. Moreover, any such element is in the kernel of $d_{k-1}$. This shows that $\Gamma_{\widetilde\fC^{\otimes k}}(i)=i t$. For any $i\geq k+1$, the above description of $v_k$ and $\delta_{1,k}$ also implies that $\delta_{1,k}v_k^{i-1}$ is trivial and hence $\Gamma_{\widetilde\fC^{\otimes k}}(i)=\infty$. Vanishing of $\Gamma$ for non-positive integers follows from $\delta_{2,k}=0$.
\end{proof}

Together with Proposition \ref{tilde-complex-dbletwist} we obtain the following.

\begin{prop}\label{double-twist-Gamma}
	For any positive integers $m$, $n$ and $k$, we have
	\[
		\Gamma_{kD_{m,n}}(i) = \begin{cases} 0, & i \leqslant 0\\ i \frac{(2m-1)(2n-1)}{4mn-1}, & 1\leq i \leq k\\
		 \infty, & i\geqslant k+1 \end{cases}
	\]
\end{prop}

It is straightforward to see that the method of the proof of Lemma \ref{lemma:ctconnectedsum} can be used to compute the $\Gamma$-invariant for arbitrary tensor products of the complexes $\widetilde \fC(t)$ for varying $t$.

\begin{lemma}
	For a set $A=\{t_1,\dots,t_k\}$ of positive real numbers, let $\widetilde \fC(A)$ denote the tensor product of the I-graded $\cS$-complexes $\widetilde \fC(t_i)$.
	Then we have
	\[
		\Gamma_{\widetilde \fC(A)}(i) = \begin{cases} 0, & i \leqslant 0\\ \min_{\{t_{j_1},\dots,t_{j_i}\}\subseteq A}\{t_{j_1}+\dots+t_{j_i}\}, & 1\leq i \leq k\\
		 \infty, & i\geqslant k+1 \end{cases}
	\]
\end{lemma}

We may then apply Proposition \ref{tilde-complex-dbletwist} to obtain the following.

\begin{prop}
	For the connected sum $K=D_{m_1,n_1}\#\dots\#D_{m_k,n_k}$ we have
	\[
		\Gamma_{K}(i) = \begin{cases} 0, & i \leqslant 0\\ \min_{\{{j_1},\dots,{j_i}\}\subseteq \{1,\dots,k\}}\left\{\sum_{a=1}^i\frac{(2m_{j_a}-1)(2n_{j_a}-1)}{4m_{j_a}n_{j_a}-1}\right\}, & 1\leq i \leq k\\
		 \infty, & i\geqslant k+1 \end{cases}
	\]	
\end{prop}

\subsection{Other examples}\label{subsec:othercomps}

For any 2-bridge knot $K$, we can use the partial description of I-graded $\cS$-complexes given in Subsection \ref{2-bridge} to give lower bounds on $\Gamma_K( \ell)$ for any $ \ell\in \Z$, and sometimes determine $\Gamma_K$ completely. We in turn obtain lower bounds on the 4-dimensional clasp numbers and unknotting numbers for 2-bridge knots via Theorem \ref{thm:gammaintro}.

\begin{figure}[t]
\centering
\begin{tikzpicture}[scale=0.9]
\draw (0,0) node {$\zeta^0$};
\node [below, purple] at (0,-0.25) {$(0,0)$};
\draw (1.5,0.0) node {$\zeta^{3}$};
\node [below, purple] at (1.5,-0.25) {$(1, \frac{9}{51})$};
\draw [->, blue] (1.0,0.0) -- (0.5,0.0); 
\draw (3.0,0.0) node {$\zeta^{1}$};
\node [below, purple] at (3.0,-0.25) {$(2, \frac{35}{51})$};
\draw (3.0,-1.5) node {$\zeta^{16}$};
\node [below, purple] at (3.0,-1.75) {$(2, \frac{35}{51})$};
\draw (4.5,0.0) node {$\zeta^{2}$};
\node [below, purple] at (4.5,-0.25) {$(3, \frac{38}{51})$};
\draw [->, blue] (4.0,0.0) -- (3.5,0.0); 
\draw (4.5,-1.5) node {$\zeta^{19}$};
\node [below, purple] at (4.5,-1.75) {$(3, \frac{38}{51})$};
\draw [->, blue] (4,-1.5) -- (3.5,-1.5);
\draw (4.5,-3) node {$\zeta^{6}$};
\node [below, purple] at (4.5,-3.25) {$(3, \frac{36}{51})$};
\draw [->,green] (4.15,-3) [rounded corners=10pt] -- (1.5,-3) -- (1.5,-1.15);
\draw (6.0,0.0) node {$\zeta^{4}$};
\node [below, purple] at (6.0,-0.25) {$(4, \frac{50}{51})$};
\draw (6.0,-1.5) node {$\zeta^{13}$};
\node [below, purple] at (6.0,-1.75) {$(4, \frac{50}{51})$};
\draw (6.0,-3.0) node {$\zeta^{17}$};
\node [below, purple] at (6.0,-3.25) {$(4, \frac{68}{51})$};
\draw (7.5,0.0) node {$\zeta^{5}$};
\node [below, purple] at (7.5,-0.25) {$(5, \frac{59}{51})$};
\draw [->, blue] (7.0,0.0) -- (6.5,0.0); 
\draw (7.5,-4.5) node {$\zeta^{15}$};
\node [below, purple] at (7.5,-4.75) {$(5, \frac{72}{51})$};
\draw (7.5,-6) node {$\zeta^{9}$};
\node [below, purple] at (7.5,-6.25) {$(5, \frac{81}{51})$};
\draw [->,green] (7.15,-6) [rounded corners=10pt] -- (4.5,-6) -- (4.5,-4.15);
\draw (7.5,-3.0) node {$\zeta^{18}$};
\node [below, purple] at (7.5,-3.25) {$(5, \frac{69}{51})$};
\draw [->, blue] (7,-3) -- (6.5,-3); 
\draw (7.5,-1.5) node {$\zeta^{22}$};
\node [below, purple] at (7.5,-1.75) {$(5, \frac{59}{51})$};
\draw [->, blue] (7,-1.5) -- (6.5,-1.5); 
\draw (9.0,0.0) node {$\zeta^{7}$};
\node [below, purple] at (9.0,-0.25) {$(6, \frac{83}{51})$};
\draw (9.0,-1.5) node {$\zeta^{10}$};
\node [below, purple] at (9.0,-1.75) {$(6, \frac{83}{51})$};
\draw (9.0,-3.0) node {$\zeta^{14}$};
\node [below, purple] at (9.0,-3.25) {$(6, \frac{77}{51})$};
\draw [->, blue] (8.5,-3.15) -- (8,-4.15); 
\draw (9.0,-4.5) node {$\zeta^{20}$};
\draw [->, blue] (8.5,-4.5) -- (8,-4.5); 
\node [below, purple] at (9.0,-4.75) {$(6, \frac{77}{51})$};
\draw (10.5,0) node {$\zeta^{8}$};
\node [below, purple] at (10.5,-0.25) {$(7, \frac{98}{51})$};
\draw [->, blue] (10.0,0.0) -- (9.5,0.0); 
\draw (10.5,-1.5) node {$\zeta^{12}$};
\node [below, purple] at (10.5,-1.75) {$(7, \frac{93}{51})$};
\draw (10.5,-4.5) node {$\zeta^{21}$};
\node [below, purple] at (10.5,-4.75) {$(7, \frac{84}{51})$};
\draw [->, blue] (10.2,-4.35) -- (9.5,-3.1); 
\draw [->, blue] (10.0,-4.5) -- (9.5,-4.5); 
\draw (10.5,-3) node {$\zeta^{25}$};
\node [below, purple] at (10.5,-3.25) {$(7, \frac{98}{51})$};
\draw [->, blue] (10.2,-2.85) -- (9.5,-1.6); 
\draw (12.0,0.0) node {$\zeta^{11}$};
\node [below, purple] at (12.0,-0.25) {$(8, \frac{104}{51})$};
\draw [->, blue] (11.5,-0.15) -- (11,-1.15); 
\draw (12.0,-1.5) node {$\zeta^{23}$};
\node [below, purple] at (12.0,-1.75) {$(8, \frac{104}{51})$};
\draw [->, blue] (11.5,-1.5) -- (11,-1.5); 
\draw (13.5,0) node {$\zeta^{24}$};
\node [below, purple] at (13.5,-0.25) {$\;(9, \frac{117}{51})$};
\draw [->, blue] (13,0) -- (12.5,0); 
\draw [->, blue] (13.1,-0.1) -- (12.5,-1.35); 
\end{tikzpicture}
\caption{This diagram shows generators for $C_\ast(K;\Delta_\sS)$ and the reducible $\zeta^0$ for $K=K_{51,16}$. As in Figure \ref{fig:74complex}, all arrows are multiplication by $\pm (T^2-T^{-2})$. The short (blue) arrows represent $d$ except for the left-most one, which is $\delta_1$. Again, $\delta_2=0$. The two L-shaped (green) arrows are multiplication by $\pm(T^2-T^{-2})$, representing components of $v$. The $v$-map might have other non-zero components, but the only ones that contribute to $\Gamma_{K}$ are the ones shown.}
\label{fig:5116complex}
\end{figure}
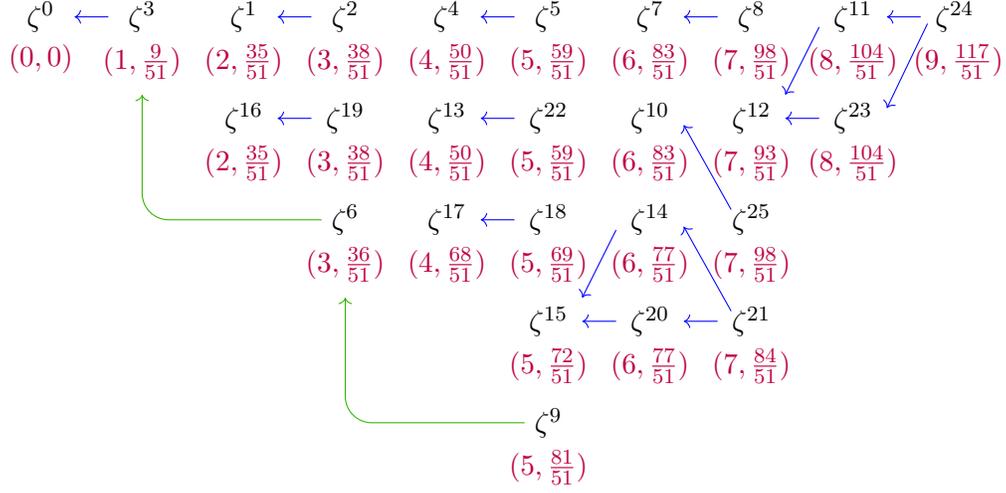

First we note that $\Gamma_K(0)$ and $\Gamma_K(1)$ only depend on knowledge of the maps $d,\delta_1,\delta_2$, and these are completely determined in the discussion of Subsection \ref{2-bridge}. Now we consider $\Gamma( \ell)$ for $ \ell\geqslant 2$. We assume for a moment that $-\sigma(K)/2\geqslant 1$. Although the $v$-map is not entirely determined there, condition \eqref{eq:v2bridge} narrows down the possibile nonzero matrix entries $\langle v \zeta^i , \zeta^j\rangle$. Further information about $v$ comes from Theorem \ref{thm:htwistedsign}, which implies
\begin{equation}
	 -\frac{1}{2}\sigma(K) = \max \{ k: \exists \alpha \in C_{2k-1}, \; \delta_1v^{k-1}\alpha \neq 0, \; \delta_1v^i\alpha = 0\;  (i\leq k-2)\}. \label{eq:signatureinfoforv}
\end{equation}
Using this information we obtain lower bounds on $\Gamma_K( \ell)$ for $ \ell\geq 2$. Similarly, our determination of $d,\delta_2$ and the possible non-zero entries $\langle v \zeta^i , \zeta^j\rangle$ give us lower bounds on $\Gamma_K(\ell)$ for $\ell  \leq -1$. Note that \eqref{eq:signatureinfoforv} implies $\Gamma_K(\ell)=\infty$ for $\ell > -\sigma(K)/2$. We remark that for the double twist family $D_{m,n}$ of Subsection \ref{subsec:doubletwist}, which have $-\sigma/2=1$, we verified \eqref{eq:signatureinfoforv} directly, without having to appeal to Theorem \ref{thm:htwistedsign}.

\begin{figure}[t]
\captionsetup[subfigure]{labelformat=empty}
    \centering
    \begin{subfigure}{0.45\textwidth}
        \centering
        \includegraphics[scale=0.55,angle=90,origin=c]{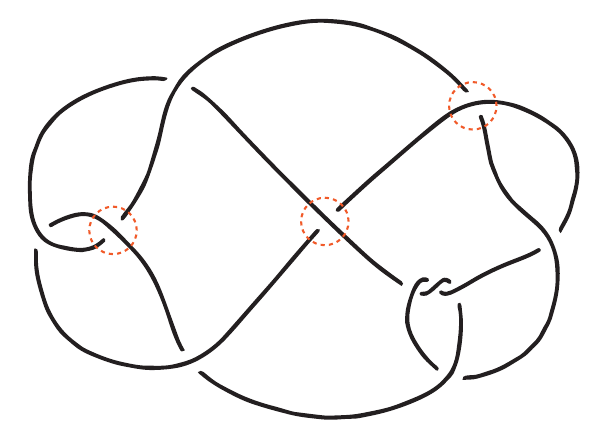} 
	\vspace{-.65cm}
        \caption{$11a_{192}$} 
    \end{subfigure}
    \hspace{.25cm}
    \begin{subfigure}{0.45\textwidth}
        \centering
        \includegraphics[scale=0.55,angle=90,origin=c]{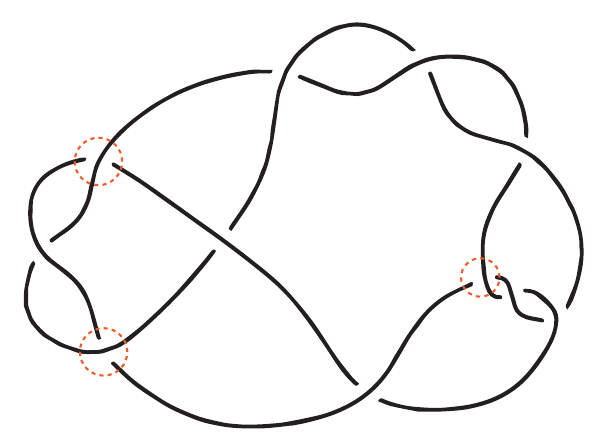} 
\vspace{-.65cm}
        \caption{$11a_{341}$} 
    \end{subfigure}
    \vspace{0.31cm}

    \begin{subfigure}{0.45\textwidth}
        \centering
        \includegraphics[scale=0.55,angle=90,origin=c]{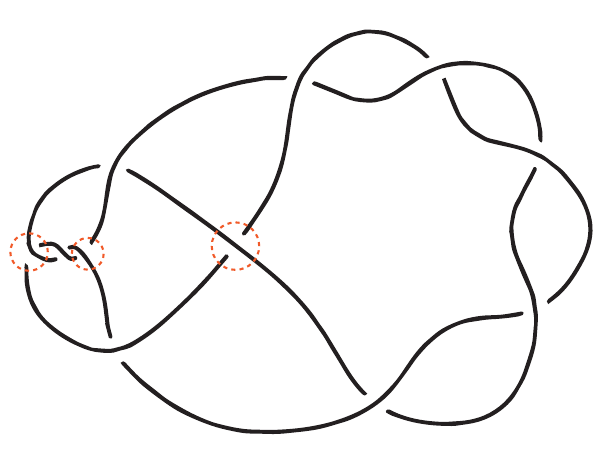} 
\vspace{-.5cm}
        \caption{$11a_{360}$} 
    \end{subfigure}
    \hspace{.25cm}
    \begin{subfigure}{0.45\textwidth}
        \centering
        \includegraphics[scale=0.55]{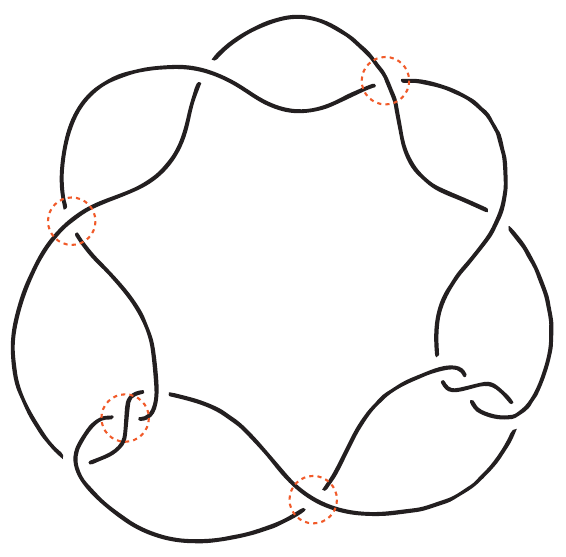} 
        \caption{$11a_{365}$} 
    \end{subfigure}
    \caption{Crossing changes realizing the unknotting numbers of the indicated knots.}\label{fig:unknotting}
\end{figure}

Consider the 2-bridge knot $K=K_{51,16}$. Up to taking mirrors this is the $11$-crossing knot $11a_{365}$, and it has $-\sigma/2=3$. Generators for an I-graded $\cS$-complex for $K$ are shown in Figure \ref{fig:5116complex}. The choice of lifts $\zeta^i$ are such that any possible non-zero component of $v$ moves a generator two columns to the left. The constraints discussed above imply that the instanton gradings of $\zeta^3$, $\zeta^6$, $\zeta^9$ realize the values of $\Gamma_K(1)$, $\Gamma_K(2)$, $\Gamma_K(3)$, respectively. We obtain
\[
	\Gamma_{K}(i) = \begin{cases} 0 & \qquad i\leqslant 0 \\ 3/17 &\qquad i=1 \\  12/17 &\qquad i=2 \\  27/17 &\qquad i=3 \\ \infty &\qquad i \geq 4 \end{cases}
\]
Theorem \ref{thm:gammaintro} then implies $27/17=\Gamma_{K}(3)\leq \frac{1}{2}c_s^+(K)$ so that $c_s^+(K)>3$. On the other hand, the unknotting number of $K$ is at most $4$, as illustrated in Figure \ref{fig:unknotting}. This in turn implies $c_s^+(K)=u(K)=4$, so that the unknotting number is in fact $4$.

We remark that the unknotting number of $11a_{365}$ was previously computed by Owens \cite{owens} using different techniques, and that the computation of $c_s^+(K)$ can be obtained from the work of Owens and Strle \cite{owens-strle}.

We next consider some 2-bridge knots whose unknotting numbers and four-dimensional clasp numbers have not been computed previously, as far as the authors are aware. The examples are the 2-bridge knots $(97,26)$, $(61,42)$, $(57,10)$. Up to mirrors these are the 11-crossing knots $11a_{192}$, $11a_{341}$, $11a_{360}$, respectively. Each has $-\sigma/2 = 2$. We use the above constraints to compute $\Gamma_K(2)>1$ in each case, implying $c_s^+(K)>2$. On the other hand, each of these knots has unknotting number at most $3$, as illustrated in Figure \ref{fig:unknotting}, so that in fact $c_s^+(K)=u(K)=3$ in each case. These observations together give Theorem \ref{thm:unknottingnumbers}.

To explain a bit more, consider first $(57,10)$. We check that lifts $\zeta^8$, $\zeta^{11}$, $\zeta^{17}$ with instanton gradings $62/57$, $62/57$, $68/57$ are the only possibilities with $\delta_1 v(\zeta^i)\neq 0$. Because all the instanton gradings are bigger than $1$, we have $\Gamma_K(2) >1$.

Next consider $(61,42)$. Here $\zeta^2$, $\zeta^{18}$, $\zeta^{22}$ are the only possibilities with $\delta_1v(\zeta^i)\neq 0$, with respective instanton gradings $58/61$, $62/61$, $64/61$. The first of these is not greater than $1$. However, $d\zeta^{2}=\pm (T^2-T^{-2})\zeta^{1}$, and the only other generator with $\zeta^{1}$ in its differential is $\zeta^{18}$. Thus any $\alpha$ with $\delta_1 v(\alpha)=0$ and $d\alpha=0$ which involves $\zeta^{2}$ must also involve $\zeta^{18}$. In this case $\alpha$ has instanton grading at least $62/61$. Thus $\Gamma_K(2)\geq 62/61 >1$. 

The case of $(97,26)$ is similar. We can check that the only lifts $\zeta^i$ which can possibly have $\delta_1 v (\zeta^i)\neq 0$ are $\zeta^{20}$, $\zeta^{27}$, $\zeta^{31}$, $\zeta^{45}$. These have respective instanton gradings $104/97$, $110/97$, $116/97$, $90/97$. However the instanton grading of $\zeta^{45}$ can be ruled out as in the previous example, since $d\zeta^{45}=\pm (T^2-T^{-2})\zeta^{26}$, and the only other $\zeta^i$ with $\zeta^{26}$ in its differential is $\zeta^{27}$. We conclude $\Gamma_K(2)\geqslant 104/97 > 1$.

%
%
%
%




\section{Inequalities for singular instanton concordance invariants}\label{sec:inequalities}

In this section we establish an inequality for the singular instanton Fr\o yshov invariants similar to ones for the non-singular Fr\o yshov invariant obtained in \cite{froyshov, froyshov-inequality}. We also obtain some inequalities for $\Gamma_{(Y,K)}$, and discuss the relationship to the 4D clasp number. Finally, we discuss constraints on cobordisms involving torus knots.

\subsection{An inequality for the $h$-invariant}

In this subsection we prove the following inequality for the invariant $h_\sS(Y,K)$.

\begin{theorem}\label{thm:hineq}
	Let $W$ be a compact, oriented, smooth 4-manifold with $b^1(W)=b^+(W)=0$, $\partial X=Y$ an integer homology 3-sphere, and $S\subset W$ an oriented surface with $\partial S=K$ a knot in $Y$. Let $c\in H^2(W;\Z)$. Suppose that $\eta(W,S,c)\neq 0 \in \sS$. Then:
	\begin{equation}
		h_\sS(Y,K) \geqslant 4\kappa_\text{\emph{min}}(W,S,c) - g(S) + \frac{1}{4}S\cdot S - \frac{1}{2}\sigma(Y,K). \label{eq:hinvinequality}
	\end{equation}
\end{theorem}

The quantities $\eta(W,S,c)$ and $\kappa_\text{\rm{min}}(W,S,c)$ are defined  as in Subsection \ref{sec:morphisms} for $(W,S,c)$. Alternatively one can remove a small regular neighborhood of a point on the interior of $S$ to obtain a cobordism from $(S^3,U_1)$ to $(Y,K)$, and then use the definitions of Subsection \ref{sec:morphisms} for cobordisms of pairs. The reader may easily check that the right side of \eqref{eq:hinvinequality} is an integer. We begin with a result about maps on the irreducible singular instanton Floer homology induced by the types of cobordisms appearing above.

\begin{prop}\label{prop:relinv}
	Let $(W,S,c)$ and $(Y,K)$ be as in Theorem \ref{thm:hineq}. Define
\begin{equation}
	n:=4\kappa_\text{\emph{min}}(W,S,c) - g(S) + \frac{1}{4}S\cdot S - \frac{1}{2}\sigma(Y,K)-1.\label{eq:indexforineq}
\end{equation}
If $n\geqslant 0$, there is a cycle $C(W,S,c)\in C_{2n+1}(Y,K;\Delta_\sS)$ satisfying
	\[
		\delta_1 v^j(C(W,S,c)) = \begin{cases} 0, & 0\leqslant j < n \\  \eta(W,S,c), \qquad & j=n \end{cases}
	\]
\end{prop}

\begin{proof}
	The proof is a direct adaptation of the non-singular case, \cite[Proposition 1]{froyshov}. The invariant $C(W,S,c)$ is defined using a moduli space of singular instantons associated to  $(W,S,c)$.  For a given irreducible critical point $\alpha$ defining a basis element of $C_\ast(Y,K;\Delta_\sS)$, we define the coefficient of $\alpha$ in $C(W,S,c)$ to be
\[
	\langle \alpha,C(W,S,c) \rangle = \sum_{[A]\in M(W,S;\alpha)_0} \epsilon(A)U^{2\kappa_\text{min}(W,S,c)-2\kappa(A)} T^{\nu(A)}
\]
where $\epsilon(A)\in\{\pm 1\}$ is determined by the orientation of the moduli space. The space $M(W,S;\alpha)_0$ is the $0$-dimensional moduli space of instantons on $(W,S,c)$ with limiting flat connection $\alpha$. The usual argument involving 1-dimensional moduli spaces shows that $dC(W,S,c)=0$. The key point is that reducibles on $(W,S,c)$ do not show up in the argument because of the assumption on $n$.

Now consider the moduli space $M(W,S,c;\theta)_{2n+1}$ of instantons for $(W,S,c)$ which are asymptotically reducible and have energy $\kappa_\text{min}(W,S,c)$. This has dimension $2n+1$. By following the perturbation scheme of \cite[Section 7.3]{dcx} (here we use $n\geqslant 0$), we can ensure that the reducibles have neighborhoods homeomorphic to cones on $\pm\bC\bP^{n}$, and that $M(W,S,c;\theta)_{2n+1}$ is a (possibly non-compact) smooth manifold away from the singularities corresponding to reducibles.

Let $M'$ be the result of first removing small open $(2n+1)$-ball neighborhoods of the reducibles in $M(W,S,c;\theta)_{2n+1}$, and then cutting down by the codimension 2 divisors associated to the euler class of the $U(1)$ basepoint fibrations determined by the reduction in neighborhoods of $n$ points in $S$. (See \cite{Kr:obs} for more details on these codimension 2 divisors. We use the same convention as in \cite[Remark 8.3]{DS} to define these divisors.) Then $M'$ is a 1-manifold with boundary. The argument in \cite[Section 5.1.4]{dk} shows that the signed count of the boundary points of the cut down moduli space on the links $\pm \C\P^n$ of the reducibles is $\pm 1$. Thus counting the boundary points of $M'$ yields $\eta(W,S,c)$. 

However, in general $M'$ also has ends, which come from instantons with breakings along the cylinder, giving rise to a term of the form $\delta_1\psi$. Here $\psi\in C_1(Y,K;\Delta_\sS)$ and the coefficient of a generator $\alpha$ in $\psi$ is obtained by counting the instantons in $M(W,S,c;\alpha)_{2n}$ after cutting down the moduli space by the codimension 2 divisors induced by the basepoint fibrations over the $n$ points in $S$. Moving these basepoints along the cylindrical end, one by one, shows that $\psi$ is homologous to $v^nC(W,S,c)$. In particular, we obtain $\delta_1 v^n \psi = \eta(W,S,c)$. After replacing $M(W,S,c;\theta)_{2n+1}$ with the moduli space of lower energy $M(W,S,c;\theta)_{2j+1}$ for $j< n$, which contains no reducibles, and is empty if $j$ does not have the same parity as $n$ for index reasons, the same argument shows that $\delta_1 v^j \psi = 0$.
\end{proof}

\begin{lemma}\label{lemma:inequalityforh}
Theorem \ref{thm:hineq} holds when the right side of \eqref{eq:hinvinequality} is non-negative, i.e. $n\geqslant -1$.
\end{lemma}

\begin{proof}
		First suppose $n= -1$. Then the cobordism $(W,S)$, with a small $(B^4,D^2)$ deleted in the interior, along with the element induced by $c$, is negative definite over $\sS$. By Proposition \ref{prop:morphism} we have an induced morphism $\widetilde C_\ast(S^3,U_1) \to \widetilde C_\ast(Y,K)$ of $\cS$-complexes. 
		This implies $h_\sS(Y,K)\geqslant h_\sS(S^3,U_1)=0$.

Next suppose $n\geqslant 0$. Then Proposition \ref{prop:relinv}, with our assumptions, implies that there is a cycle $C(W,S,c)\in C_{2n+1}(Y,K;\Delta_\sS)$ such that
	\[
		\delta_1 v^n C(W,S,c) \neq 0, \quad \delta_1 v^j C(W,S,c) = 0 \;\; (0\leqslant j < n).
	\]
	Then by Proposition \ref{h-reinterpret} we have $h_\sS(Y,K) \geqslant n + 1$. This completes the proof.
\end{proof}

\begin{remark}
	In Subsection \ref{sec:higherlevelmorphisms}, we introduce a more general notion of morphisms between $\cS$-complexes. 
	In Lemma \ref{lemma:inequalityforh}, it is possible to construct such generalized
	 morphisms $\widetilde C_\ast(S^3,U_1) \to \widetilde C_\ast(Y,K)$ in the case that $n\geq 0$ and then conclude the claim in the same way as in the case $n=-1$. $\diamd$
\end{remark}

Our approach to remove the condition $n\geqslant -1$ divides into two cases, depending on whether $T^4-1$ vanishes in $\sS$ or not. We search for the simplest non-trivial instances of the inequality for $h_\sS(Y,K)$ in each of these two cases. In the first case we will see that $T_{3,4}$ provides a suitable example, and in the second case $T_{2,3}$.

\begin{lemma}\label{blowuptwist}
	Let $K'\subset Y$ be obtained from a knot $K$ by a full twist around a collection of strands with linking number $d$, as in Figure \ref{fig:blowuptwist}. Then there is a cobordism 
	$(W,S):(Y,K) \to (Y,K')$ where $W$ is the connected sum of the product cobordism $I \times Y$ and $\overline{\C\P}^2$, and $S$ is a cylinder with $S\cdot S=-d^2$. 
	In particular, there is a compact oriented connected genus $0$ surface $S_m \subset \overline{\C\P}^2\setminus B^4$ of degree $m$ with boundary the 
	torus knot $T_{m,m+1}\subset S^3$.
\end{lemma}
\begin{proof}
	The first claim is elementary. To see the second claim, note that the first part gives a cobordism from the unknot $U_1=T_{m,1}$ to the torus knot $T_{m,m+1}$ embedded in the connected sum of 
	$I\times S^3$ and $\overline{\C\P}^2$ which has degree $m$. By capping off the incoming end with a disc embedded in the $4$-ball we obtain the desired surface $S_m$.
\end{proof}

\begin{figure}[t]
\centering
\includegraphics[scale=.9]{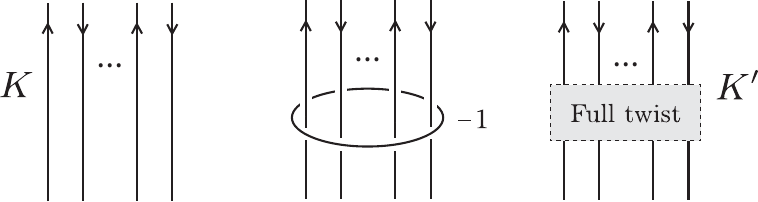}
\caption{The knot $K'$ is obtained from $K$ by blowing up around the strands indicated, which introduces one full right-handed twist. The linking number $d$ is the number of upwards pointing strands minus the number of downwards pointing strands.}
\label{fig:blowuptwist}
\end{figure}

\begin{cor}\label{cor:mmplus1}
	For any $\sS$ and $m$ odd we have $h_\sS(T_{m,m+1})\geqslant (m-1)/2$.
\end{cor}

\begin{proof}

We have $\sigma(T_{m,m+1}) = -(m-1)(m+3)/2$. Furthermore, for the pair $(\overline{\C\P}^2\setminus B^4 ,S_m)$ in Lemma \ref{blowuptwist} with $c=0$, we have
\begin{equation}
	\kappa_\text{min}(\overline{\C\P}^2\setminus B^4, S_m) = \min_{a\in \Z} \left( a + \frac{m}{4} \right)^2 = \frac{1}{16} \label{eq:kappminodd}
\end{equation}
realized uniquely by $a = (\pm 1 -m)/4$ when $m\equiv \pm 1 \pmod 4$. Thus $\eta(W,S)$ is a power of $T$, up to sign. The right side of \eqref{eq:hinvinequality} computes to $(m-1)/2$. The result then follows from Lemma \ref{lemma:inequalityforh}.
\end{proof}

\begin{prop}
\hangindent\leftmargini
{\phantom{.}}
	\begin{enumerate} \item[{\em (i)}] If $T^4=1$ in $\sS$ then $h_\sS(T_{3,4})=1$. \item[{\em (ii)}] If $T^4\neq 1$ in $\sS$ then $h_\sS(T_{2,3})=1$.
\end{enumerate}
\end{prop}

This proposition follows from the computations in \cite[Section 9]{DS}, which rely on the equivariant ADHM construction. Here we give an alternative route to these computations.

\begin{proof}	
	From Corollary \ref{cor:mmplus1}, we have $h_\sS(T_{3,4})\geqslant 1$, and 
	we just need to show $h_\sS(T_{3,4})\leqslant 1$. 
	The complex $C_*(T_{3,4};\Delta_\sS)$ has three generators $\alpha$, $\alpha'$ and $\beta$, 
	where the first two elements have degree one and $\beta$ has degree $3$ \cite[Section 9.4]{DS}. 
	We pick honest lifts 
	$\widetilde \alpha$, $\widetilde \alpha'$ and $\widetilde \beta$ 
	with 
	\[
	  \widetilde {\rm gr}(\widetilde \alpha)=\widetilde {\rm gr}(\widetilde \alpha')=1,\hspace{1cm}
	  \widetilde {\rm gr}(\widetilde \beta)=3.
	\]
	To prove the claim it suffices to show that $v(\beta)=0$. Since $T^4=1$, Lemma \ref{lift-morphisms} implies that 
	$\langle v(\widetilde \beta),\widetilde \alpha\rangle$ and $\langle v(\widetilde \beta),\widetilde \alpha'\rangle$ 
	are both integers. Therefore, it suffices to prove the claim for the case that $\sS=\Z$.
	If $v(\beta)\neq 0$, then it is easy to see from this and $h_\sS(T_{3,4})\geqslant 1$ that the homology of 
	the complex $\widetilde C_*(T_{3,4})$ over $\Z$ has rank at most $3$. This 
	is a contradiction because the homology of $\widetilde C_*(T_{3,4})$ is isomorphic to
	$I^\natural(T_{3,4})$ \cite{DS}, which has rank $5$.
	
	The complex $C_*(T_{2,3};\Delta_\sS)$ has only one generator $\alpha$ with  degree $1$. 
	Therefore, the maps $v$ and $\delta_2$ are trivial and $0\leq h_\sS(T_{2,3}) \leq 1$. To determine the exact value of $h_\sS(T_{2,3})$ in the case $T^4\neq 1$, 
	we will make use of the pair $(\overline{\C\P}^2\setminus B^4, S_2)$. Note that $S_2$ has genus $0$, and self-intersection $-4$. We find that 
	\[
	  \kappa_\text{min}(\overline{\C\P}^2\setminus B^4, S_2)=\min_{a\in \Z}\left(a - \frac{1}{2}\right)^2= \frac{1}{4},
	\]
	realized by $a=0$ and $a=1$. The corresponding monopole numbers for these reducibles are $-2$ and $2$. It follows that $\eta(\overline{\C\P}^2\setminus B^4, S_2)=1-T^4$, which is non-zero by assumption. 
	So Proposition \ref{prop:relinv} implies that the map $\delta_1$ is non-zero and $h_\sS(T_{2,3})=1$.
\end{proof}

\begin{proof}[Proof of Theorem \ref{thm:hineq}]
	Let $(W,S,c)$ and $(Y,K)$ be as in the statement of Theorem \ref{thm:hineq}, and $n$ as in \eqref{eq:indexforineq}. By Lemma \ref{lemma:inequalityforh} we may assume $n\leqslant -2$.

First suppose $T^4=1$ in $\sS$. Take the boundary sum of $(W,S)$ with $N$ copies of the pair $(\overline{\C\P}^2\setminus B^4,S_3)$ from Lemma \ref{blowuptwist} to obtain a new pair $(W',S')$ with boundary $(Y',K')=(Y,K)\#(S^3,NT_{3,4})$. For $(W',S',c)$ the left side of \eqref{eq:hinvinequality} is
\[
	h_\sS(Y',K') = h_\sS(Y,K) + Nh_\sS(T_{3,4})  = h_\sS(Y,K) + N,
\]
and, similar to the computation leading to Corollary \ref{cor:mmplus1}, the right side of \eqref{eq:hinvinequality} is $n+1+N$. Furthermore, the condition $\eta\neq 0$ is preserved under this operation. Thus the inequality for $(W',S',c)$ is equivalent to the one for $(W,S,c)$. Letting $N\geqslant -n-1$, we see that $(W',S',c)$ satisfies the condition of Lemma \ref{lemma:inequalityforh}, completing the proof in this case.

Next suppose $T^4\neq 1$ in $\sS$. We follow the same strategy, but using the trefoil and $(\overline{\C\P}^2\setminus B^4, S_2)$ in place of $T_{3,4}$ and $(\overline{\C\P}^2\setminus B^4, S_3)$. 
For the pair $(\overline{\C\P}^2\setminus B^4, S_2)$, the left hand side of \eqref{eq:hinvinequality} is $h_\sS(T_{2,3})=1$ and the right hand side is computed to $1$ as well. We then repeat the above strategy of taking the boundary sum of $(W,S)$ with enough copies of $(\overline{\C\P}^2\setminus B^4, S_2)$.
\end{proof}

Now suppose $K$ is an unknot $U_1\subset Y$ and let $(W,S)$ be such that $b^1(W)=b^+(W)=0$ and $S$ is an unknotted disk with boundary $U_1$. Then inequality \eqref{eq:hinvinequality} implies
\begin{equation}\label{eq:hunknot1}
	h_\sS(Y,U_1) \geqslant |w^2|
\end{equation}
whenever $w=2z-c$ minimizes $-(z-\frac{1}{2}c)^2$ over all $z$, and $\eta(W,S,c)\neq 0 \in \sS$. The cohomology class $w$ is called an {\em extremal vector} in the terminology of \cite{froyshov, froyshov-inequality}. Note that if $w$ is non-zero, $\eta(W,S,c)$ is twice the corresponding count of reducibles in the non-singular case. We will use the following in the next section.

\begin{lemma}\label{lemma:poincare}
Suppose $2\neq 0$ in $\sS$. Let $\Sigma$ be the Poincar\'{e} sphere oriented as the boundary of a negative definite plumbing  $W$ with intersection form $-E_8$. 
Then $h_\sS(\Sigma,U_1) \geqslant 4$.
\end{lemma}

\begin{proof}
The result follows from \eqref{eq:hunknot1} because $-E_8$, the lattice of $W$, has an extremal vector of square $-4$ with count of reducibles $\eta=16\neq 0 \in \sS$, see \cite[Section 4]{froyshov}.
\end{proof}

Theorem \ref{thm:hineq} can be generalized by incorporating $\mu$-classes of surfaces into the argument, as Fr\o yshov does in the non-singular case. The statement is as follows. Choose a non-negative integer $m$, a class $a\in H^2(W;\Z)$, and suppose
\[
	 \sum (-1)^{z^2} T^{(2z-c)\cdot S} (a\cdot (2z-c))^m \neq 0 \in \sS
\]
where the sum is over all $z$ minimal with respect to $c$. Then:
\begin{equation*}
		h_\sS(Y,K) \geqslant 4\kappa_\text{min}(W,S,c) - g(S) + \frac{1}{4}S\cdot S - \frac{1}{2}\sigma(Y,K)-m.
\end{equation*}
We leave the details to the interested reader. Upon establishing $h_\Q(Y,U_1)=4h_\Q(Y)$, this inequality recovers the inequalities for negative definite 4-manifolds from \cite{froyshov-inequality}.

\subsection{Higher height morphisms of $\cS$-complexes}\label{sec:higherlevelmorphisms}
In Proposition \ref{prop:relinv}, associated to $(W,S,c)$ with boundary $(Y,K)$, we constructed an element of the complex $C_*(Y,K;\Delta_\sS)$. Although this relative element does not come from a morphism of $\cS$-complexes, it interacts nicely with the structure of the $\cS$-complex associated to $(Y,K)$. The goal of this subsection is to introduce a general algebraic structure which captures the information of such relative elements in the context of $\cS$-complexes. We start with the following definition which is a generalization of morphisms of $\cS$-complexes.

\begin{definition}
	Let $R$ be an integral domain. Suppose $\widetilde C_\ast = C_\ast \oplus C_{\ast-1}\oplus R$ and 
	$\widetilde C'_\ast = C'_\ast \oplus C'_{\ast-1}\oplus R$ are $\cS$-complexes over $R$ with structure maps $d$, $v$, $\delta_1$, $\delta_2$ 
	and $d'$, $v'$, $\delta_1'$, $\delta_2'$ respectively. 
	For a positive integer $i$, 
	a {\emph{ height $i$ morphism}} from $\widetilde C_\ast$ to $\widetilde C_\ast'$ is a chain map 
	$\widetilde \lambda:(\widetilde C_\ast,\widetilde d)\to (\widetilde C_\ast',\widetilde d')$ which has the form
	\begin{equation}\label{eq:morphism-level}
		\widetilde \lambda = \left[ \begin{array}{ccc} \lambda & 0 & 0 \\ \mu & \lambda & \Delta_2 
		\\ \Delta_1 & 0 & 0  \end{array} \right].
	\end{equation}
	Furthermore, we require that there is a  {\em non-zero} $\eta\in R$ such that
	\begin{equation}\label{level-j-rel}
	  \delta_{1}'v'^j\Delta_2+\Delta_1v^j\delta_2+\sum_{l=0}^{j-1}\delta_{1}'v'^l\mu v^{j-1-l}\delta_2
	  =\left\{
	  \begin{array}{cc}
	  \eta&j=i-1\\
	  0&j<i-1
	  \end{array}
	  \right..
	\end{equation}
	In the case that $i=0$, a {\it height $0$ morphism} is defined to be a morphism of $\cS$-complexes in 
	the previous sense. $\diamd$
\end{definition}

The topological counterpart of height $i$ morphisms of $\cS$-complexes is given by the following definition, which generalizes Definition \ref{defn:negdefpair}.

\begin{definition}\label{defn:negdefpair-higher-level}
	Let $(W,S):(Y,K)\to (Y',K')$ be a cobordism of pairs between oriented knots in integer homology 
	3-spheres, where $S$ is embedded, connected and oriented, and $c\in H^2(W;\Z)$. 
	We say $(W,S,c)$ is {\emph{negative definite of height $i\in \Z_{\geq 0}$ over $\sS$}} if:
	\begin{itemize}
		\item[(i)] $b^1(W)=b^+(W)=0$;
		\item[(ii)] the index of one (and thus every) minimal reducible $A_L$ is $2i-1$;
		\item[(iii)] $\eta(W,S,c)$ is non-zero as an element in $\sS$. 
	\end{itemize}
	If $c=0$, we also say $(W,S)$ is a {\emph{negative definite pair of height $i$ over $\sS$}}.  $\diamd$
\end{definition}

The reader should note that, similar to Definition \ref{defn:negdefpair}, the conditions in Definition \ref{defn:negdefpair-higher-level} are about the cohomology groups of $W$, the homology class of $S$ and the cohomology class $c$. In particular, one can easily check it in many explicit cases.

\begin{prop}\label{level-n-cob-mor}
	Let $(W,S,c):(Y,K)\to (Y',K')$ be negative definite of height $i$ over $\sS$. Then $(W,S,c)$ induces a 
	height $i$ morphism 
	$\widetilde \lambda_{(W,S,c)}:\widetilde C_*(Y,K;\Delta_\sS) \to \widetilde C_*(Y',K';\Delta_\sS)$ 
	of $\cS$-complexes. At the level of $\Z\times \R$-graded enriched complexes, we have
	\begin{equation}\label{enriched-level-i}
	  \widetilde \lambda_{(W,S,c)} \left(\widetilde C_{l,j}(Y,K;\Delta_\sS)\right) \subseteq  
	  \bigcup_{k\leqslant j+2\kappa_{\rm min}(W,S,c)+\delta}\widetilde C_{l+2i,k}(Y',K';\Delta_\sS)
	\end{equation}
	where $\delta$ can be made arbitrary small by arranging the auxiliary perturbations in the definition of
	$\widetilde C_*(Y,K;\Delta_\sS)$, $\widetilde C_*(Y',K';\Delta_\sS)$
	and $\widetilde \lambda_{(W,S,c)}$ small enough.
\end{prop}
\begin{proof}
	The construction of the matrix components of $\widetilde \lambda_{(W,S,c)}$ and showing that it is a chain map are similar to the case of maps associated to 
	negative definite pairs as it 
	is reviewed in Subsection \ref{sec:morphisms}. To obtain the relations in \eqref{level-j-rel}, consider the moduli space $M(W,S,c;\theta,\theta')_{2j+1}$ for $0\leq j \leq i-1$
	where $\theta$ and $\theta'$ are reducibles associated to $(Y,K)$ and $(Y',K')$. The assumption on $(W,S,c)$ implies that for $j<i-1$, $M(W,S,c;\theta,\theta')_{2j+1}$
	does not contain any reducible, and only minimal reducibles contribute to this moduli space if $j=i-1$, which we can assume are regular \cite[Section 7.3]{dcx}. 
	Analogous to the proof of Proposition \ref{prop:relinv}, we
	cut down $M(W,S,c;\theta,\theta')_{2j+1}$ by $j$ codimension two divisors associated to $j$ points $x_1$, $x_2$, $\dots$, $x_j\subset S$, and then study the ends 
	(and boundary components) of the resulting 1-manifold to obtain the desired relations.
	
	To simplify our analysis it is convenient to make some additional assumptions on the points $x_1$, $\dots$, $x_j$. The surface cobordism $S$ (after adding cylindrical ends)
	has an end modeled on $(-\infty,0]\times K$. Let $y\in K$ be the basepoint of the knot $K$. Then we assume $x_l=(t_l,y)\in (-\infty,0]\times K$ are such that 
	\[
	  t_1<t_2<\dots<t_j<0
	\] 
	where $t_{l+1}-t_l$ ($1\leq l\leq j-1$) and $-t_j$ are very large. Later in the proof we make the latter condition more precise. From now on we	also assume that $j=i-1$. The proof for $j<i-1$ is similar. We also focus on the case that the coefficient ring is $\bF=\Z/2$. By taking into account orientations 
	of moduli spaces using the conventions from \cite{DS} and Subsection \ref{subsec:reducibles}, we can lift the coefficients to $\Z$. We recall that we used slightly different sign conventions in the definitions of $\delta_2$ and $\Delta_2$ in \cite{DS}, which should be taken into account when the signs of different terms appearing below are characterized. By working with local  systems, the proof can be
	easily adapted to obtain the extension to the coefficient ring $\Z[U^{\pm 1},T^{\pm1}]$ and the claim in \eqref{enriched-level-i} about morphisms at the level of enriched
	complexes (cf. Proposition \ref{prop:morphismigraded}). 
	
	Studying the ends and the boundary components of the intersection of the space $M(W,S,c;\theta,\theta')_{2i-1}$
	with the $i-1$ codimension $2$ divisors, after having removed small neighborhoods of the reducibles, implies that
	\begin{equation}\label{identity-red-ends}
	  \Delta_{1,i-1}\delta_2+\delta_1'\Delta_{2,i-1}=\eta(W,S,c).
	\end{equation}
	The term on the left hand side of the identity are the contributions of the ends of the moduli space and the term on the right hand side are the boundary points corresponding to 
	the minimal reducibles. Furthermore, $\Delta_{1,i-1}:C_*(Y,K)\to \bF$ and $\Delta_{2,i-1}:\bF \to C_*(Y',K')$ are homomorphisms defined as follows. For a generator $\alpha\in C_*(Y,K)$ we cut down $M(W,S,c;\alpha,\theta')_{2i-2}$ by the codimension 2
	divisors associated to $x_1, \ldots, x_{i-1}$, and $\Delta_{1,i-1}(\alpha)$ is the number of elements of this $0$-dimensional space.
	Similarly, the coefficient of a generator $\alpha'\in C_*(Y',K')$ in $\Delta_{2,i-1}(1)$ is the count of elements of 
	$M(W,S,c;\theta,\alpha')_{2i-2}$ after being cut down by the divisors associated to $x_1,\dots,x_{i-1}$. In particular, $C(W,S,c)$ in Proposition \ref{prop:relinv}
	can be regarded as a special case of the map $\Delta_{2,i-1}$.

	Next we identify the left hand side of \eqref{identity-red-ends} in terms of $\widetilde \lambda_{(W,S,c)}$ and the differentials of $\widetilde C_*(Y,K)$ and $\widetilde C_*(Y',K')$.
	Our assumption on the arrangement of $x_1$, $\dots$, $x_{i-1}$ implies that $\Delta_{1,i-1}=\Delta_1 v^{i-1}$. To analyze $\Delta_{2,i-1}$, we need to 
	move the points $x_i$ one by one to the other end of the cobordism. Part of the data of the cobordism map $\widetilde \lambda_{(W,S,c)}$ is a path $\gamma$ 
	from the basepoint of $K$ to that of $K'$. To be more precise, $\gamma:\R\to S$ is a map such that $\gamma(t)=(t,y)$ for $t\leq -1$ and $\gamma(t)=(t,y')$ for $t\geq 1$
	with $y'$ being the basepoint of $K'$. Take the homotopy which moves the point $x_{i-1}$ along the path $\gamma$ in the positive direction. 
	This defines a codimension $1$ divisor of $M(W,S,c;\theta,\alpha')_{2i-2}$ with boundary being the divisor associated to the point $x_{i-1}$.
	Further cutting down $M(W,S,c;\theta,\alpha')_{2i-2}$ by the codimension $2$ divisors associated to $x_1,\dots,x_{i-2}$ gives rise to a (possibly non-compact) 
	1-dimensional manifold $N_{\alpha'}$ where the count of boundary components is equal to $\Delta_{2,i-1}$.

\begin{figure}[t]
\centering
\includegraphics[scale=1]{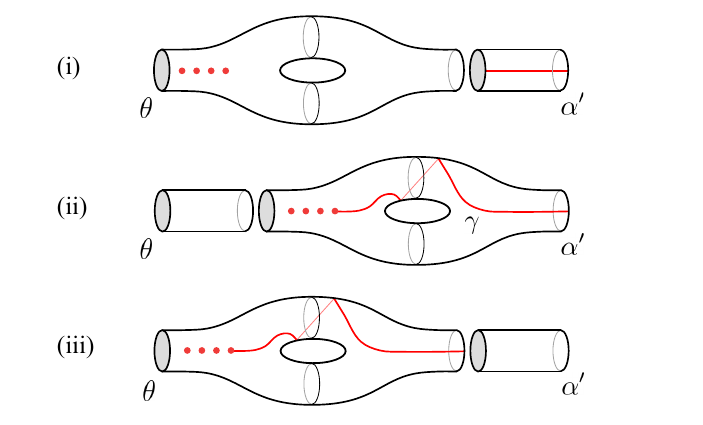}
\caption{The three types of ends of the moduli space $N_{\alpha'}$}
\label{fig:nalphaends}
\end{figure}

The 1-manifold $N_{\alpha'}$ can be compactified by compactifying its ends, which are of the three types demonstrated in Figure \ref{fig:nalphaends}. The count of the ends of type (i) is equal to the coefficient of $\alpha'$ in $v\Delta_{2,i-2}(1)$. Here 
	$\Delta_{2,i-2}$ is defined similar to $\Delta_{2,i-1}$ with the difference that we only use the points $x_1,\dots,x_{i-2}$ to cut down the moduli spaces
	$M(W,S,c;\theta,\alpha')_{2i-4}$. The count of the ends of type (ii) gives the coefficient of $\alpha'$ in $\mu_{i-1}\delta_2$ where $\mu_{i-1}:C_*(Y,K)\to C_*(Y',K')$
	is the map given by cutting the moduli spaces $M(W,S,c;\alpha,\alpha')_{2i-3}$ after cutting down by codimension divisors $2$ associated to the points 
	$x_1,\ldots,x_{i-2}$ and the codimension $1$ divisor associated to the homotopy of the point $x_{i-1}$. Our assumption about the points 
	$x_1,\ldots,x_{i-1}$ implies that $\mu_{i-1}=\mu v^{i-2}$. Finally, the ends of type (iii) gives the coefficient of $\alpha'$ in a term of the form $d'\Phi(1)$ where
	$\Phi:\bF\to C_*(Y,K)$ is defined in terms of the moduli spaces $M(W,S,c;\theta,\alpha')_{2i-3}$ after cutting down by codimension $2$ divisors associated to $x_1,\dots,x_{i-2}$ and the codimension $1$ divisor associated to the homotopy of the point $x_{i-1}$. Since we are eventually interested in 
	$\delta_1'\Delta_{2,i-1}$, the term $d'\Phi(1)$ does not play any role as a consequence of $\delta_1'd'=0$. In summary, we obtain the relation
	\[
	  \delta_1'\Delta_{2,i-1}=\delta_1'v'\Delta_{2,i-2}+\delta_1'\mu v^{i-2}\delta_2.
	\]
	We iterate this process to study the term $\delta_1'v'\Delta_{2,i-2}$ by moving the remaining points $x_1,\ldots,x_{i-2}$ to the outgoing end of the cobordism, one by one.
	When we apply this iteration to the point $x_{k}$, we obtain an additional term of the form $\delta_1'v^{i-1-k}\mu v^{k-1}\delta_2$. 
	Thus the desired relation is given at the end of this process.
\end{proof}

\subsection{Immersed cobordisms, crossing changes and twisting}\label{subsec:immersedcrossing}

In this subsection, we firstly extend the scope of Theorem \ref{thm:hineq} to the case of normally immersed cobordisms. Then we apply this extension to some special cobordisms to obtain further information about the invariant $h_\sS(Y,K)$. As already manifested in the proof of Theorem \ref{thm:hineq}, the behavior of $h_\sS(Y,K)$ depends on whether $T^4=1$ or not.

Let $(W,S):(Y,K) \to (Y',K')$ be a cobordism of pairs such that $b^1(W)=b^+(W)=0$, where $S$ is a normally immersed oriented cobordism of genus $g$ with $s_+$ and $s_-$ positive and negative double points, and $c\in H^2(W;\Z)$.  We can extend the definition of $\kappa_{\rm min}(W,S,c)$ and $\eta(W,S,c)$ of Subsection \ref{subsec:reducibles} to the immersed case because they only depend on the cohomology ring of $W$ and the homology class of $S$ and the cohomology class $c$. After blowing up we have the pair $(\overline{W},\overline{S})$. Let $e_1,\ldots, e_{s_+}$ denote the homology classes of exceptional spheres at the positive double points, and $f_1,\ldots,f_{s_-}$ the remaining exceptional spheres. The homology class of $\overline S$ is 
\[
  [\overline S] = [S]-2e_1-2e_2-\cdots-2e_{s_+}.
\]
Let $\overline c$ be the class in $H^2(\overline W,\Z)$ determined by $c$. Then any reducible connection $A_L$ associated to $(\overline{W},\overline{S},\overline c)$ is determined by its Chern class $c_1(L)$, given as
\begin{equation}\label{first-Chern-class}
 	z+\sum_{i=1}^{s_+} k_i e_i + \sum_{j=1}^{s_-} l_j f_j \in H^2(W;\Z)\oplus \Z^{s_++s_-}.
\end{equation}
The topological energy and the monopole number of any such reducible are
\[
  \kappa(A_L)=-\left(z+\frac{1}{4}S-\frac{1}{2}c\right)^2+\sum_{i=1}^{s_+}\left(k_i-\frac{1}{2}\right)^2+\sum_{j=1}^{s_-}l_j^2,
\]
\[  
  \nu(A_L) = (2z-c) \cdot S+4\sum_{i=1}^{s_+}k_i.
\]
In particular, the minimum topological energy $\kappa_\text{min}(\overline W,\overline S,\overline c)$ is equal to $\kappa_\text{min}(W,S,c)+s_+/4$ and is realized by the cohomology classes as in \eqref{first-Chern-class} such that $z$ is minimal for $(W,S,c)$, $k_i\in \{0,1\}$ and $l_j=0$. Note $\overline S\cdot \overline S = S\cdot S-4s_+$. Moreover, it is straightforward to check that in this situation we have
\begin{equation}
  \eta(\overline W, \overline S,\overline c) =(1-T^4)^{s_+}\eta(W, S,c).\label{eq:etacrossing}
\end{equation}
In particular, non-vanishing of $\eta(W, S,c)$ implies the non-vanishing of $\eta(\overline W, \overline S,\overline c)$ in the case that $T^4\neq 1$.

If $T^4=1$, the class $\overline c$ needs to be modified to avoid vanishing of \eqref{eq:etacrossing}. Define
\[
   \overline c=c-e_1-\dots-e_{s_+}.
\]
Then $\eta(\overline W, \overline S,\overline c)=\eta(W, S,c)$, $\kappa_\text{min}(\overline W,\overline S,\overline c)=\kappa_\text{min}(W,S,c)$ and the minimum topological energy is realized by the cohomology classes as in \eqref{first-Chern-class} such that $z$ is minimal for $(W,S,c)$ and $k_i=l_j=0$.

\begin{theorem}\label{thm:hineq-cob}
	Suppose $(Y,K)$, $(Y',K')$ and $(W,S,c)$ are given as above. Suppose that $\eta(W,S,c)\neq 0 \in \sS$. Then we have the following inequality:
	\begin{align}
		h_\sS(Y',K')-h_\sS(Y,K) &\geqslant\nonumber\\ 
		4\kappa_\text{\emph{min}}(W,S,c)& - g(S) + \frac{1}{4}S\cdot S - \varepsilon_\sS -
		\frac{1}{2}(\sigma(Y',K')-\sigma(Y,K)), \label{eq:hinvinequality-cob}
	\end{align}
where $\epsilon_\sS=0$ if $T^4\neq 1$ in $\sS$ and $\epsilon_\sS = s_+$ if $T^4=1$.
\end{theorem}
\begin{proof}
	This is a slight variation of Theorem \ref{thm:hineq} and in fact it can be reduced to that theorem.
	By firstly blowing up $(W,S)$ to obtain $(\overline W, \overline S)$ and then removing a neighborhood 
	of the path from the basepoint of $K$ to the basepoint of $K'$ one obtains a pair 
	$(\overline W', \overline S')$ with boundary $(-Y\#Y,-K\#K')$. Then we apply Theorem \ref{thm:hineq}
	to $(\overline W', \overline S')$ and a choice of $\overline c'\in H^2(\overline W';\Z)$ which is determined 
	by the above discussion. This and the additivity of $h_\sS$ with respect to connected sum
	gives the desired inequality.
	
	It is instructive to consider the cobordism version of the proof of Theorem \ref{thm:hineq} at least in the 
	case that the right side of \eqref{eq:hinvinequality-cob}, denoted by $i$, is positive. 
	The argument will be useful later
	when we study the behavior of $\Gamma_{(Y,K)}$ with respect to cobordisms.
	By Proposition \ref{level-n-cob-mor} there is a height $i$ morphism 
	$\widetilde \lambda_{(\overline W,\overline S,\overline c)}:(\widetilde C_*(Y,K;\Delta_\sS),\widetilde d)\to 
	(\widetilde C_*(Y',K';\Delta_\sS),\widetilde d')$
	with the matrix components $\lambda$, $\mu$, $\Delta_1$ and $\Delta_2$. Let also the matrix 
	components of $\widetilde d$ and $\widetilde d'$ be given by $d$, $v$, $\delta_1$, $\delta_2$,
	and $d'$, $v'$, $\delta_1'$, $\delta_2'$ respectively. Using the defining relations of height $i$ morphisms, it is straightforward to check 
	\begin{align}
	  \delta_1'v'^l\lambda=&\left(\Delta_1v^l+\sum_{j=0}^{l-1}\delta_1'v'^j\mu v^{l-1-j}\right) d\nonumber\\
	  &+\sum_{j=0}^{l-1}\left(\Delta_1v^j\delta_2+\delta_1'v'^j\Delta_2
	  +\sum_{k=0}^{j-1}\delta_1'v'^k\mu v^{j-1-k}\delta_2\right)\delta_1v^{l-1-j}, \label{deltavlambda}
	\end{align}
	\begin{align}
	  \lambda v^l\delta_2=&-d'\left(v'^l\Delta_2+\sum_{j=0}^{l-1}v'^j\mu v^{l-1-j}\delta_2\right)\nonumber\\
	  &+\sum_{j=0}^{l-1}v'^{l-1-j}\delta_2'\left(\delta_1'v'^{j}\Delta_2+\Delta_1v^{j}\delta_2
	  +\sum_{k=0}^{j-1}\delta_1'v'^k\mu v^{j-1-k}\delta_2\right).
	\end{align}
	
	Suppose $h=h_\sS(Y,K)\geq 1$. There is $\alpha\in C_*(Y,K;\Delta_\sS)$ such that $d\alpha=0$ and $h-1$
	is the smallest integer $l$ such that $\delta_1v^l(\alpha)$ is non-zero. For $\alpha'=\lambda(\alpha)$, 
	\eqref{deltavlambda} implies that
	\[
	  \delta_1'v'^l(\alpha')=\left\{\begin{array}{ll}0&l\leq h+i-2\\ \eta(\overline W,\overline S,\overline c)\delta_1v^{h-1}(\alpha)&l=h+i-1 \end{array}\right.
	\]
	Since $d'\alpha'=0$, we conclude that $h_\sS(Y',K')\geq h+i$. In the case $h\leq 0$, there are 
	$\alpha\in C_*(Y',K';\Delta_\sS)$, $a_0$, $\dots$, $a_{-h}\in \sS$ with $a_{-h}\neq 0$ such that
	\begin{equation}\label{rel-h-neg}
	  d\alpha=\sum_{l=0}^{-h}v^{l}\delta_2(a_{l}).
	\end{equation}
	If $-i<h\leq 0$, then define $\alpha'$ as follows:
	\begin{equation}\label{alphap}
	  \alpha'=\lambda(\alpha)+\sum_{l=0}^{-h}\left(v'^{l}\Delta_2(a_{l})+\sum_{j=0}^{l-1}v'^j\mu v^{l-1-j}\delta_2(a_l) \right).
	\end{equation}
	By \eqref{rel-h-neg} we have $d\alpha'=0$ and \eqref{deltavlambda} implies that
	\[
	  \delta_1'v'^l(\alpha')=\left\{\begin{array}{ll}0&l\leq h+i-2\\ \eta(\overline W,\overline S,\overline c)a_{-h}&l=h+i-1 \end{array}\right..
	\]
	Finally for $h\leq -i$, applying $\lambda$ to \eqref{rel-h-neg} implies that 
	\[
	  d\alpha'=\delta_2(a_0')+v\delta_2(a_1')+\dots+v^{-h-i}\delta_2(a_{-h-i}')
	\]
	with $\alpha'$ again defined as in \eqref{alphap}, and where each $a_k'$ is given by
	\[
	  a_k'=\sum_{j=i-1}^{-h-k-1}\left(\delta_1'v'^{j}\Delta_2(a_{k+j+1})+\Delta_1v^{j}\delta_2(a_{k+j+1})
	  +\sum_{m=0}^{j-1}\delta_1'v'^m\mu v^{j-1-m}\delta_2(a_{k+j+1})\right).
	\]
	In particular, $a_{-h-i}'=\eta(\overline W,\overline S,\overline c)a_{-h}$, which is non-zero. This completes the proof.
\end{proof}

Examples of cobordisms as in Theorem \ref{thm:hineq-cob} arise from crossing changes. Namely, suppose the knot $(Y,K)$ is obtained from $(Y,K')$ by replacing $s_+$ positive crossings with negative crossings and $s_-$  negative crossings with positive crossings. Then there is an annulus immersed in $W=I\times Y$ giving a cobordism $S:K\to K'$ with the properties discussed above. We refer to $(\overline W,\overline S)$ in this situation as a {\em crossing change cobordism}. Since $W$ is the product cobordism, we have $\eta(W,S,c)=1$ and $\kappa_\text{\rm{min}}(W,S,c)=0$ for $c=0$.

We already made use of one crossing change cobordism in the proof of Theorem \ref{thm:hineq}:
\[
	(S^3,U_1) \to (S^3,T_{2,3}),
\]
obtained from deleting a 4-ball from the pair $(\overline{\C\P}^2\setminus B^4, S_2)$ of Lemma \ref{blowuptwist}, is the same as the cobordism induced by turning the unknot into the trefoil by a positive crossing change. Thus $s_+=1$ and $s_-=0$ for this example. We now apply Theorem \ref{thm:hineq-cob} to crossing change cobordisms to obtain the main part of Theorem \ref{thm:htwistedsign}:

\begin{prop}\label{prop:hinvcrossing}
	Suppose $K$ and $K'$ are homotopic knots in $Y$, and $T^4\neq 1$. Then 
	\[
	  h_\sS(Y,K)+\sigma(Y,K)/2=h_\sS(Y,K')+\sigma(Y,K')/2. 
	\]
	In particular, if $K$ is a null-homotopic knot in $Y$, then
	\[
	  h_\sS(Y,K)-h_\sS(Y,U_1)= -\sigma(Y,K)/2.
	\]
\end{prop}

\begin{proof}
	Take a crossing change cobordism $S:K\to K'$, which exists because $K$ and $K'$ are 
	homotopic knots. Then Theorem \ref{thm:hineq-cob} applied to this cobordism implies that
	\[
	   h_\sS(Y,K)+\sigma(Y,K)/2\leq h_\sS(Y,K')+\sigma(Y,K')/2.
	\]
	We obtain the reverse inequality by applying Theorem \ref{thm:hineq-cob} to the reverse cobordism.
\end{proof}

\begin{remark}
	In the next section we shall show that $h_\sS(Y,U_1)=4h_{\sS}(Y)$, where $h_{\sS}(Y)$ is the (non-singular) Fr\o yshov invariant for the integral homology sphere $Y$. This will complete the proof of the claim 
	$ h_\sS(Y,K)= -\sigma(Y,K)/2+4h_\sS(Y)$ for a null-homotopic knot. In general, the expression $h_\sS(Y,K)+\sigma(Y,K)/2$ determines 
	a map on conjugacy classes of $\pi_1(Y)$, i.e. free homotopy classes of loops in $Y$, which is equal to $4h_{\sS}(Y)$ at the trivial conjugacy class. 
	However, the authors do not know if the invariant is ever different for a non-trivial conjugacy class. $\diamd$
\end{remark}

Crossing change cobordisms form a special case of twist cobordisms which appeared in Lemma \ref{blowuptwist}. In the light of Proposition \ref{prop:hinvcrossing}, we study the behavior of $h_\sS$ with respect to twist cobordisms in the case that $T^4= 1$.

\begin{prop}\label{eq:blowuptwistineq}
	Let $K$ and $K'$ be knots in $Y$ related by a full twist around a collection of strands with 
	linking number $d$, as in Figure \ref{fig:blowuptwist}. If $T^4=1$ in $\sS$ then
	\begin{equation}\label{ineq-twist}
	  h_\sS(Y,K') - h_\sS(Y,K) \geqslant -\left \lfloor\frac{d^2}{4}\right\rfloor - \frac{1}{2}(\sigma(Y,K')-\sigma(Y,K)).
	\end{equation}
\end{prop}
\begin{proof}
	In the case that $d$ is odd, the twist cobordism $(W,S):(Y,K)\to (Y,K')$ with $c=0$ satisfies the 
	assumption of Theorem \ref{thm:hineq-cob} with $\kappa_\text{min}=1/16$. Since $S$ has $g=0$
	and self-intersection $-d^2$, the claim follows from Theorem \ref{thm:hineq-cob}. The same
	argument applies in the case $d$ is congruent to $0$ modulo $4$, with the difference that 
	$\kappa_\text{min}=0$. If $d$ is congruent to $2$ modulo $4$, then we need to modify $c$ to 
	a generator of $H^2(W;\Z)$ to guarantee $\eta\neq 0$. This produces $(W,S,c)$ with 
	$\kappa_\text{min}=0$ again, which in turn gives \eqref{ineq-twist}.
\end{proof}

\begin{remark}
	Although it is possible to apply the argument of Proposition \ref{eq:blowuptwistineq} 
	to the case that $T^4\neq1$, this does not give rise to any interesting information
	because the knots $K$ and $K'$ in Proposition \ref{eq:blowuptwistineq} are homotopic
	and Proposition \ref{prop:hinvcrossing} already determines $h_\sS(Y,K') - h_\sS(Y,K)$. $\diamd$
\end{remark}

\subsection{An inequality for the $\Gamma$-invariant}

We now turn to an inequality for the invariant $\Gamma_{(Y,K)}^R$. Here we fix an integral domain $R$ which is an algebra over $\Z[T^{\pm 1}]$. The following is an analogue of an inequality for the $\Gamma_Y$ invariant in the non-singular setting \cite{AD:CS-Th}.

\begin{prop}\label{Gamma-ineq}
	Let $(W,S):(Y,K)\to (Y',K')$ be a cobordism of pairs with $b^1(W)=b^+(W)=0$ such that $S$ is normally immersed
	with $s_+$ positive and $s_-$ negative double points, and
	let $c\in H^2(W;\Z)$. Suppose that $\eta(W,S,c)\neq 0\in R[U^{\pm 1}]$. Let
	\begin{equation}
		i := 4\kappa_\text{\em min}(W,S,c) - g(S)+ \frac{1}{4}S\cdot S-\epsilon_R  
		+ \frac{1}{2}\sigma(Y,K)- \frac{1}{2}\sigma(Y',K') \label{eq:iforgammauntwisted}
	\end{equation}
	where $\epsilon_R=0$ if $T^4\neq 1$ in $R$, and $\epsilon_R=s_+$ if $T^4=1$. Let $k\in \Z$. If $i\geqslant 0$, then
	\begin{equation}\label{eq:Gamma-ineq}
	  \Gamma_{(Y',K')}^{R}(k+i) \leqslant 2\kappa_\text{\em min}(W,S,c) + \frac{1}{2}(s_+ - \epsilon_R)  + 
	  \Gamma_{(Y,K)}^{R}(k) .
	\end{equation}
\end{prop}

\begin{proof}
	Let $(\overline W,\overline S)$ be given by blowing up $(W,S)$ and $\overline c$ be the cohomology class given 
	in Subsection \ref{subsec:immersedcrossing}, which depends on whether $T^4=1$ or not. 
	Proposition \ref{level-n-cob-mor} provides a height $i$ morphism 
	$\widetilde \lambda_{(\overline W,\overline S,\overline c)}:\widetilde C_\ast(Y,K;\Delta_\sS) \to 
	\widetilde C_\ast(Y',K';\Delta_\sS)$.
	We suppose no perturbations are needed in these constructions, so that the $\cS$-complexes $\widetilde C_\ast(Y,K;\Delta_\sS)$ and 
	$\widetilde C_\ast(Y',K';\Delta_\sS)$ are I-graded and the cobordism map 
	$\widetilde \lambda_{(\overline W,\overline S,\overline c)}$ 
	increases the $\Z$-grading $\widetilde {\rm gr}$ by $2i$ and is of level 
	$2\kappa_{\rm min}(\overline W,\overline S,\overline c)$ with respect to the $\R$-grading. 
	The proof can be adapted to the general case of enriched complexes by a limiting argument as in \cite{AD:CS-Th}.
	
	For each integer $k$, $\Gamma_{(Y,K)}^{R}(k) $ is defined by taking the minimum of 
	the instanton grading of all $\alpha\in C_{*}(Y,K;\Delta_\sS)$ 
	with $\widetilde {gr}=2k-1$ which satisfy an appropriate condition depending on $k$. The proof of Theorem 
	\ref{thm:hineq-cob} explains how one can define $\alpha'\in C_{*}(Y',K';\Delta_\sS)$ with $\widetilde {\rm gr}=2k+2i-1$
	which satisfies the condition relevant for the definition of $\Gamma_{(Y',K')}^{R}(k+i)$. Using the behavior of 
	$\widetilde \lambda_{(\overline W,\overline S,\overline c)}$ with respect to the $\Z\times \R$ bigrading, we see that 
	\begin{equation}\label{ineq-alpha-alphap}
	  \deg_I(\alpha')\leq \max(\deg_I(\alpha),0)+
	  2\kappa_{\rm min}(\overline W,\overline S,\overline c).
	\end{equation}
	Moreover, if $k>0$, then $\max(\deg_I(\alpha),0)$ can be replaced with $\deg_I(\alpha)$
	because $\delta_1v^{k-1}(\alpha)$ is non-trivial and hence $\deg_I(\alpha)>0$.
	The inequality in \eqref{ineq-alpha-alphap} and its modification for positive values of $k$	
	give rise to \eqref{eq:Gamma-ineq}.
\end{proof}

\begin{remark}
	The reader might wonder to what extent the assumption on non-negativity of $i$ in Proposition \ref{Gamma-ineq} is necessary. We believe the same result holds even after 
	removing this assumption. This will be discussed elsewhere. $\diamd$
\end{remark}

We have the following immediate corollary of Proposition \ref{Gamma-ineq}:

\begin{cor}\label{Gamma-ineq-classical-cob}
	Suppose $S:K\to K'$ is a normally immersed cobordism in $I\times S^3$ with $s_+$ 
	positive and $s_-$ negative	 double points. Suppose $T^4\neq 1$ in $R$, and
	 \[
	   i :=\frac{1}{2}\sigma(K)- \frac{1}{2}\sigma(K') - g(S) \geq 0.
	 \]
	 Then we have the following inequality:
	 \[
	   \Gamma_{K'}^{R}(k+i) \leqslant \frac{s_+}{2} + \Gamma_{K}^{R}(k).
	 \]
	 In particular, if $-\sigma(K)/2\geqslant 0$, then we have the inequality
	\[
	  \Gamma_K^R\left(-\frac{1}{2}\sigma(K)\right) \leqslant \frac{1}{2} c_s^+(K).
	\]
\end{cor}
\noindent This corollary implies Theorem \ref{thm:gammaintro}, and thus completes the proof of Theorems \ref{thm:clasp} and \ref{thm:doubletwist}.

\begin{proof}
	For any immersed disc $D$ in the $4$-dimensional ball with boundary $K\subset S^3$, we obtain a
	cobordism $S:U_1\to K$ of genus $0$ immersed into $I\times S^3$ by removing a neighborhood of a point on the interior of $D$
	away from the immersed points. This allows us to obtain the second part from the first part.
\end{proof}

The following is a generalization of Corollary \ref{Gamma-ineq-classical-cob} to the case that $W:S^3\to S^3$ is the blow up of the product cobordism.

\begin{cor}\label{Gamma-ineq-blownup-classical-cob}
	Suppose $(W,S):(S^3,K)\to (S^3,K')$ is a cobordism of pairs where $W$ 
	is the connected sum of $I\times S^3$ and $\overline{\C\P}^2$, and $S$ is an embedded cobordism of degree $d$. 
	Let
	\begin{equation}
		i :=  \epsilon_R(d)- \frac{d^2}{4}-s_+ 
		+ \frac{1}{2}\sigma(Y,K)- \frac{1}{2}\sigma(Y',K')- g(S)
		\label{eq:iforgamma-blownup-classical-cob}
	\end{equation}
	where $\epsilon_R(d)\in \frac{1}{4}\Z$ is defined as follows:
	\begin{equation}\label{epsilon}
		\epsilon_R(d)=\left\{
		\begin{array}{ll}
			1+s_+ & \quad \text{$T^4\neq 1$, \em $d$ even, $d\neq 0$;}\\
			 s_+& \quad\text{$T^4\neq 1$, $d=0$;}\\
			\frac{1}{4}+s_+& \quad\text{$T^4\neq 1$, \em $d$ odd;}\\
			 0& \quad\text{$T^4= 1$, \em $d$ even;}\\
			\frac{1}{4}& \quad\text{$T^4= 1$, \em $d$ odd}.
		\end{array}
		\right.
	\end{equation}
	If $i\geq 0$, then for any integer $k$ we have:
	\begin{equation}  \label{eq:Gamma-ineq-CP}
		\Gamma_{K'}^{R}(k+i) \leqslant \frac{\epsilon_R(d)}{2} + \Gamma_{K}^{R}(k).
	\end{equation}
\end{cor}

We give the following application of Corollary \ref{Gamma-ineq-blownup-classical-cob} to the {\it $\overline{\C\P}^2$-genus} of the knot $7_4$.

\begin{cor}\label{sevenfour-CP}
	The knot $7_4$ is not slice in $\overline{\C\P}^2$. That is to say, there is 
	no cobordism $(W,S):(S^3,U_1)\to (S^3,7_4)$ where $W$ is the blow up of the product
	cobordism and $S$ is diffeomorphic to an annulus.
\end{cor}
\begin{proof}
	Most of this problem was addressed in \cite{Pic:CP-genus}. 
	It is shown there that there is no genus 
	zero cobordism $S:U_1\to 7_4$ in the cobordism $W:S^3\to S^3$  unless the degree of
	$S$ is equal to $2$. If there is one such degree two cobordism, then the quantity
	$i$ and $\epsilon_R(2)$ in \eqref{eq:iforgamma-blownup-classical-cob} and 
	\eqref{epsilon} for $R=\Z[T^{\pm1}]$ are both equal to $1$. 
	Thus \eqref{eq:Gamma-ineq-CP} for $k=0$ gives
	\[
	  \Gamma_{7_4}(1) \leqslant \frac{1}{2}
	\]
	which contradicts with our computation of $\Gamma_{7_4}(1)$ in Section \ref{sec:comps}.
\end{proof}

\begin{remark}
	Given a knot $K$, one can ask if it can be obtained from the unknot by a full-twist of linking 
	number $d$. This problem has been recently studied in \cite{All-Liv:twisting}. 
	Corollary \ref{sevenfour-CP} implies that the knot $7_4$ cannot be obtained from the unknot
	by a twist of linking number $2$. This determines one of the unknown values in 
	\cite{All-Liv:twisting}. $\diamd$
\end{remark}

\subsection{Torus knots}

In this subsection we study how the special structure of the singular instanton chain complexes for torus knots imposes constraints on cobordisms and their induced maps.

\begin{theorem}\label{conc-torus}
	Suppose $(W,S):(S^3,T_{p,q}) \to (S^3,T_{p,q})$ is a homology concordance and $\sS=\Z[U^{\pm1},T^{\pm 1}]$. Then $\widetilde \lambda_{(W,S)}:\widetilde C_*(T_{p,q};\Delta_\sS) \to \widetilde C_*(T_{p,q};\Delta_\sS)$ is 
	$\cS$-chain homotopy equivalent to an isomorphism.
\end{theorem}
\begin{proof}
	Since the Chern-Simons functional of a torus knot $T_{p,q}$ is non-degenerate, we may define the enriched complex $\widetilde C_*(T_{p,q};\Delta_\sS)$
	using a perturbation of 
	the Chern-Simons functional of $T_{p,q}$ which is trivial in a neighborhood of singular flat connections, the critical points of the Chern-Simons functional. 
	In particular, $C_*(T_{p,q};\Delta_\sS)$
	is generated by lifted irreducible singular flat connections $\widetilde \alpha_1,\ldots,\widetilde \alpha_m$ associated to $T_{p,q}$ where $m=-\sigma(T_{p,q})/2$, 
	and the differential $d:C_*(T_{p,q};\Delta_\sS) \to C_*(T_{p,q};\Delta_\sS)$ vanishes. Let $\widetilde \alpha_i$ be the lift of the flat connection $\alpha_i$ such that 
	$\widetilde {\rm gr}(\alpha_i)=1$ or $3$ and ${\rm hol}_{T_{p,q}}(\widetilde \alpha_i)\in [0,1)$.
	Define $\widetilde \lambda_{(W,S)}$ using a compatible perturbation of the ASD 
	equation for singular connections on $(W,S)$. 
	
	Next, we study the map $\lambda_{(W,S)}:C_*(T_{p,q};\Delta_\sS) \to C_*(T_{p,q};\Delta_\sS)$ by looking at the coefficient of 
	$\widetilde \alpha_j$ in $\lambda_{(W,S)}(\widetilde \alpha_i)$
	for any $i$ and $j$, which is an element of $\Z[U^{\pm1}, T^{\pm1}]$ denoted by $\langle \lambda_{(W,S)}(\widetilde \alpha_i),\widetilde \alpha_j\rangle $. 
	This coefficient is zero unless $\widetilde {\rm gr}(\alpha_i)=\widetilde {\rm gr}(\alpha_j)$. Furthermore, if $\deg_I(\widetilde \alpha_j)>\deg_I(\widetilde \alpha_i)$,
	then $\langle \lambda_{(W,S)}(\widetilde \alpha_i),\widetilde \alpha_j\rangle $ is again zero if we take the perturbations involved in the definition of $\widetilde C_*(T_{p,q};\Delta_\sS)$
	and $\widetilde \lambda_{(W,S)}$ small enough. In the case that the $\Z\times \R$-bigradings of $\widetilde \alpha_i$ and $\widetilde \alpha_j$ agree with each other, 
	$\langle \lambda_{(W,S)}(\widetilde \alpha_i),\widetilde \alpha_j\rangle $ is an element of $\Z[T^{\pm 1}]$. This Laurent polynomial is in fact an integer
	if the perturbations are small enough. To see this, we momentarily assume that all perturbations are trivial. Then the instantons involved in the definition of  
	$\langle \lambda_{(W,S)}(\widetilde \alpha_i),\widetilde \alpha_j\rangle $ are flat singular connections on $(W,S)$. In particular, their monopole numbers are equal to zero. 
	This together with the assumption ${\rm hol}_{T_{p,q}}(\widetilde \alpha_i), {\rm hol}_{T_{p,q}}(\widetilde \alpha_j)\in [0,1)$ imply that 
	$\langle \lambda_{(W,S)}(\widetilde \alpha_i),\widetilde \alpha_j\rangle \in \Z$. In the case that we need to use perturbations, we can obtain the same conclusion 
	by taking a perturbation which is small enough (see \cite{AD:CS-Th} for similar arguments). In summary, if we order $\widetilde \alpha_1, \ldots, \widetilde \alpha_m$
	based on their $\Z$-gradings and then their
	$\R$-gradings, the map $\lambda_{(W,S)}$ is given by an upper-triangular block matrix such that the diagonal blocks take values in integers. To prove our claim 
	it suffices to show that the diagonal blocks are invertible.
	
	If the diagonal blocks of $\lambda_{(W,S)}$ are not invertible, then $\lambda_{(W,S)}$ is not an isomorphism over the field 
	$\sS'={\rm Frac}((\Z/r)[T^{\pm 1}])$ where $r$ is a prime number. Since $h_{\sS'}(T_{p,q})$ equals $m=-\sigma(T_{p,q})/2$ as a consequence of Proposition \ref{prop:hinvcrossing}, there is some
	$\zeta \in C_*(T_{p,q};\Delta_{\sS'})$ such that 
	\[
	  m-1=\min\{j\mid \delta_1v^j(\zeta)\neq 0\}.
	\]
	Let $\zeta'=\lambda_{(W,S)}(\zeta)$. Since $\widetilde \lambda_{(W,S)}$
	induces a morphism of $\cS$-complexes, we have
	\[
	  m-1=\min\{j\mid \delta_1v^j(\zeta')\neq 0\}.
	\]
	In particular, this shows that $v^k\zeta'$ for $0\leq k\leq m-1$ gives a basis for $C_*(T_{p,q};\Delta_{\sS'})$, which is a contradiction, and hence 
	$\lambda_{(W,S)}$ is an isomorphism.
	
	The map $\widetilde \lambda_{(W,S)}$ is filtered with respect to the filtration
	\[
	  0\subseteq C_{*-1}(T_{p,q};\Delta_{\sS})\subseteq C_{*-1}(T_{p,q};\Delta_{\sS})\oplus \sS\subseteq \widetilde C_*(T_{p,q};\Delta_\sS)
	\]
	and the induced map by $\widetilde \lambda_{(W,S)}$ on the graded part of the above filtration is given by the maps $\lambda_{(W,S)}$, $1$ and $\lambda_{(W,S)}$.
	Thus $\widetilde \lambda_{(W,S)}$ is an isomorphism, too.
\end{proof}

The following is a consequence of the proof of Theorem \ref{conc-torus}, and is a slight generalization of Theorem \ref{torus-conc-reps} from the introduction.
\begin{cor}\label{torus-conc-reps-gen}
	Let $(W,S):(S^3,T_{p,q}) \to (Y,K)$ be a homology concordance. Then any traceless $SU(2)$ representation of $\pi_1(S^3\setminus T_{p,q})$ extends over 
	the concordance complement. 
\end{cor}

\begin{proof}
	Since we can compose $(W,S)$ with its flipped copy as a cobordism from $(Y,K)$ to $(S^3,T_{p,q})$, it suffices to prove this for the case that $(Y,K)=(S^3,T_{p,q})$.
	Firstly suppose $\widetilde \lambda_{(W,S)}$ can be defined with trivial perturbations. Following the terminology of the proof of Theorem \ref{conc-torus}, in the upper-triangular
	 block representation of $\lambda_{(W,S)}$ all diagonal blocks have to be invertible. However, if for a given $\alpha_i$ there is no singular flat connection on $(W,S)$ extending
	 $\alpha_i$ on the incoming end, then a column of one of the diagonal blocks of $\lambda_{(W,S)}$ vanishes, which is a contradiction. 
	 In the presence of perturbations, one still obtains a similar contradiction for small perturbations using a limiting argument.
\end{proof}

\begin{cor}\label{instanton-map-summand-torus}
	Suppose $(W,S):(S^3,T_{p,q}) \to (Y,K)$ is a homology concordance. Then for any of the invariants $I_*$, $\widehat I_*$, $\widecheck I_*$, $I_*^\sharp$ and 
	$I_*^\natural$ with any choice of 
	coefficient ring (twisted or untwisted), the invariant of $T_{p,q}$ is a summand of the invariant of $K$.
\end{cor}
\begin{proof}
	Let $(W',S'):(Y,K) \to (S^3,T_{p,q}) $ be given by flipping the cobordism $(W,S)$. Theorem \ref{conc-torus} implies that 
	$\widetilde \lambda_{(W',S')}\circ \widetilde \lambda_{(W,S)}$ is $\cS$-chain homotopy equivalent to an isomorphism. In particular, $\lambda_{(W',S')}\circ \lambda_{(W,S)}$
	is chain homotopy equivalent to an isomorphism. This implies that $\lambda_{(W,S)}:I_*(S^3,T_{p,q})\to I_*(Y,K)$ is an isomorphism. Similar arguments apply to the other flavors 
	of singular instanton homology using the definitions and results of \cite{DS}.
\end{proof}

\begin{remark}
	In the case that $(W,S):(S^3,T_{p,q})\to (Y,K)$ is a homology concordance such that the concordance complement $W\backslash S$ has a handle decomposition with no
	3-handles (e.g. $W=I\times S^3$ and $S$ is a ribbon concordance), Corollary \ref{torus-conc-reps-gen} can be proved using an elementary argument without appealing
	to Yang-Mills gauge theory \cite{Gor:Rib,dlvw}. On the other hand, the proof of Theorem \ref{torus-conc-reps-gen} relies 
	heavily on the properties of equivariant singular instanton homology developed in the current paper. 
	To see whether Corollary \ref{torus-conc-reps-gen} is a genuine application of gauge theory, it would be desirable to find 
	a knot concordant to a torus knot $T_{p,q}$ such that there is no ribbon concordance from $T_{p,q}$ to $K$. As pointed out in the introduction, 
	the authors do not know whether any such knot exists. 
	One possible approach to obtain such a knot is to find a $K$ concordant to $T_{p,q}$ and then use one of the known obstructions to rule out the existence of a 
	ribbon concordance from $T_{p,q}$ to $K$. Various versions of instanton Floer homology provide such obstructions \cite{dlvw,km-concordance}. 
	However, Corollary \ref{instanton-map-summand-torus} asserts that they are not useful in this direction. $\diamd$
\end{remark}

\begin{remark}
	Another obstruction to the existence of ribbon concordances is given by the Alexander polynomial \cite{Gor:Rib,Gil:Rib}. If there is a ribbon concordance from $K$ to $K'$,
	then $\Delta_{K'}(t)=\Delta_{K}(t)f(t)f(t^{-1})$ where $\Delta_K(t)$ and $\Delta_{K'}(t)$ are the (symmetrized) Alexander polynomials of $K$ and $K'$, and $f(t)$ is a 
	polynomial with integer coefficients. If a torus knot $T_{p,q}$ is concordant to a knot $K'$, then the Fox-Milnor theorem implies that $\Delta_{T_{p,q}}(t)\Delta_{K'}(t)$
	is equal to $g(t)g(t^{-1})$ for a polynomial $g$ with integer coefficients. Any root $\omega$ of $\Delta_{T_{p,q}}(t)$ is simple with norm $1$. 
	For any such root $g(\omega)g(\omega^{-1})$ vanishes and hence $g(\omega)=0$ or $g(\omega^{-1})=0$. The latter case is also equivalent to $g(\omega)=0$ because
	$\omega$ has norm one and $g$ has integer coefficients. In summary $g(t)=\Delta_{T_{p,q}}(t)f(t)$ for a polynomial $f(t)$ with integer coefficients. 
	Thus $\Delta_{K}(t)=\Delta_{T_{p,q}}(t)f(t)f(t^{-1})$, i.e., the Alexander polynomial cannot be used to rule out the existence of a ribbon concordance from $T_{p,q}$ to $K$. $\diamd$
\end{remark}



\section{The twisted Fr\o yshov invariant and suspension}\label{sec:froyshovandsuspension}

In this section we prove the following, which completes the proof of Theorem \ref{thm:htwistedsign}:

\begin{theorem}\label{thm:hinvissign}
	Suppose $T^4\neq 1$ and $2\neq 0$ in $\sS$. Let $K$ be a null-homotopic knot in an integer homology 3-sphere $Y$. Then we have
\[
	h_\sS(Y,K) = -\frac{1}{2}\sigma(Y,K) + 4h_\sS(Y)
\]
where $h_\sS(Y)$ is the Fr\o yshov invariant for integer homology 3-spheres.
\end{theorem}

\noindent The Fr\o yshov invariant $h_\sS(Y)$ is defined following the construction in \cite{froyshov}, using the chain complex $\widetilde C_\ast(Y;\Delta_\sS)$. In general we can form such a chain complex with local coefficients for an algebra over $\Z[U^{\pm 1}]$ where $U$ is tied to the Chern--Simons functional. Thus in defining $h_\sS(Y)$ the local coefficient system is formed by viewing $\sS$ as an algebra over $\Z[U^{\pm 1}]$, and the $T$ variable no longer plays an essential role.

By Proposition \ref{prop:hinvcrossing}, to prove Theorem \ref{thm:hinvissign} we need only establish
\begin{equation}
	h_\sS(Y,U_1) = 4h_\sS(Y) \label{eq:4hunknot}
\end{equation}
whenever $2\neq 0 $ in $\sS$ where $U_1$ is an unknot in a small 3-ball contained in $Y$. The condition That $2\neq 0$ in $\sS$ is required by the constructions of \cite{froyshov}.

The second subsection establishes a kind of categorification of Theorem \ref{thm:hinvissign}. Let $Y=S^3$. If $\sS$ is a field, then the Fr\o yshov invariant $h_\sS(K)$ determines the local equivalence class of the $\cS$-complex $\widetilde C_\ast(K;\Delta_\sS)$, and Theorem \ref{thm:hinvissign} implies that this local equivalence class is determined by the signature of $K$. But in fact more is true: the signature determines the $\cS$-chain homotopy type of $\widetilde C_\ast(K;\Delta_\sS)$. In establishing this we are lead to the notion of suspending cobordisms by the trefoil cobordism. This idea was already utilized to an extent in Section \ref{sec:inequalities}, and will play an important role in the sequel.

In the final subsection we formulate a connected sum theorem which further explains the relation \eqref{eq:4hunknot}. We only sketch the ideas and leave some details for future investigation.

\subsection{The Fr\o yshov invariant for an unknot}

Here we establish \eqref{eq:4hunknot}, which completes the proof of Theorem \ref{thm:hinvissign}.

\begin{prop}\label{prop:hinvunknot}
	Suppose $2\neq 0 $ in $\sS$. Then $h_\sS(Y,U_1)=4h_\sS(Y)$.
\end{prop}

\begin{proof}
	Suppose $S$ is an unknotted disc in $W=I\times Y$ that fills the unknot $U_1$ in $\{0\}\times Y$. Then $(W,S):(Y,U_1) \to (Y,\emptyset)$ is a cobordism of pairs, which admits a reducible singular flat connection with index $-1$. The moduli space of (perturbed) singular ASD connections on $(W,S)$ which are asymptotic to irreducible (perturbed) flat connections and have dimension zero 
	gives rise to a chain map $\phi: C_*(Y,U_1;\Delta_\sS)\to C_*(Y;\Delta_\sS)$ where $(C_*(Y;\Delta_\sS),d_{Y})$ is Floer's instanton complex with local coefficient system $\Delta_\sS$. Considering the $0$-dimensional moduli spaces on $(W,S)$ which are asymptotic to an irreducible (perturbed) flat connection on the end $(Y,U_1)$ and the trivial connection on the end $Y$ determines a map $\fD: C_*(Y,U_1;\sS)\to \sS$ which satisfies the following property:
	\begin{equation}\label{fD-chain}
	  \fD\circ d+\delta_1-D_1\circ \phi=0.
	\end{equation}
	Here $D_1:C_*(Y;\sS)\to \sS$ is the map defined by counting (perturbed non-singular) ASD instantons over the cylinder $\R\times Y$ which are asymptotic to an irreducible (perturbed) flat connection on the incoming end and 
	the trivial connection on the outgoing end. The above identity follows from a standard argument by considering 1-dimensional moduli spaces of singular ASD connections over $(W,S)$. In particular, the term $\delta_1$ comes from gluing
	the reducible flat connection on $(W,S)$ to singular instantons on $(\R\times Y,\R\times U_1)$ from an irreducible (perturbed) flat connection to the reducible flat connection on $(Y,U_1)$.

	The identity in \eqref{fD-chain} and the chain map property of $\phi$ implies that 
	\begin{equation}\label{prop-h-hs}
		h_\sS(Y,U_1) > 0  \qquad \Longrightarrow \qquad h_\sS(Y) > 0
	\end{equation}
	The two invariants $h_\sS(\,\cdot\, )$ and $h_\sS(\,\cdot\, ,U_1\,)$ define homomorphisms $\Theta^3_{\Z}\to \Z$ from the integer homology cobordism group to integers. The implication in \eqref{prop-h-hs} implies that 
	the kernel of $h_\sS(Y)$ is a subset of the kernel of $h_\sS(Y,U_1)$. On the other hand, if $h_\sS(Y,U_1)=0$, then $h_\sS(Y\#\Sigma,U_1)>0$ and $h_\sS(Y\#-\Sigma,U_1)<0$ where $\Sigma$ is the Poincar\'{e} homology sphere. 
	This is a consequence of Corollary \ref{lemma:poincare}. Using \eqref{prop-h-hs}, we conclude that $h_\sS(Y\#\Sigma)>0$ and $h_\sS(Y\#-\Sigma)<0$. 
	Since $h_\sS(\Sigma)=1$  \cite{froyshov}, we have $h_\sS(Y)=0$. Thus the two invariants $h_\sS(\,\cdot\, )$ and $h_\sS(\,\cdot\, ,U_1\,)$ have equal kernels, 
	and to complete the proof it suffices to show that $h_\sS(\Sigma,U_1)=4$.
	
	The Chern-Simons functional of the pair $(\Sigma,U_1)$ is Morse-Bott and its irreducible critical set consists of two copies of $S^2$. Thus $\text{rk} (I_1(\Sigma,U_1;\Delta_\sS))$
	is at most 2. From this, \eqref{h-twisted-ineq-irr-I} implies that $h_\sS(\Sigma,U_1) \leqslant  4$. This inequality and Lemma \ref{lemma:poincare} shows that $h_\sS(\Sigma,U_1)=4$, completing the proof.
\end{proof}

In the statement of Theorem \ref{thm:hinvissign} with $\sS=\Z[U^{\pm 1},T^{\pm 1}]$ the non-singular Fr\o yshov invariant $h_\sS(Y)$ for the homology 3-sphere $Y$ that appears is a priori different from the one defined by Fr\o yshov. Indeed, we are using a local coefficient system over the ring $\Z[U^{\pm 1}]$ while Fr\o yshov works over $\Z$. Similarly, we might inquire about the role of the variable $U$ for singular Fr\o yshov invariants. The following shows that this role is inessential. 

\begin{prop}
	Let $R$ be an integral domain algebra over $\Z[T^{\pm 1}]$ and $\sS=R[U^{\pm 1}]$. Then $h_\sS(Y,K)= h_R(Y,K)$. 
\end{prop}

\begin{proof}
	Let $\alpha_1,\ldots,\alpha_n$ be singular flat connections which generate $C_1(Y,K;\Delta_R)$. Suppose that $d\alpha = 0$ where $\alpha=\sum c_i \alpha_i$ where $c_i\in R$ and $\delta_1\alpha\neq 0$. The flat connections $\alpha_i$ have lifts $\widetilde \alpha_i$ each having $\Z$-grading $1$. Let $\widetilde \alpha = \sum c_i \widetilde \alpha_i$. Then in the enriched complex over $R[U^{\pm 1}]$,
\[
	d\widetilde \alpha = \sum a_i \widetilde \beta_i
\]
where the $\widetilde\beta_i$ are lifts of flat connections and have $\Z$-grading $0$. The important observation here is that $a_i\in R$ and contains no powers of $U$. This is in fact a special case of Lemma \ref{lift-morphisms}. From this we see that $d\widetilde \alpha=0$. Since $\delta_1\alpha\neq 0$ over $R$ the same is true over $\sS$. Thus $h_R(Y,K)>0$ implies $h_\sS(Y,K) >0$. Similar to the proof of Proposition \ref{prop:hinvunknot}, equality then follows by the fact that the invariants agree for the trefoil (see \cite{DS} for the computation of $h_\sS(T_{2,3})$ for all $\sS$). 
\end{proof}

A similar argument shows that $U$ plays essentially no role in the determination of Fr\o yshov's $h$-invariant for homology 3-spheres.

\subsection{Suspension and $\cS$-complexes}

Consider an immersed cobordism $(W,S):(Y,K)\to (Y',K')$ with normal crossings as in Subsection \ref{subsec:immersedcrossing}. For simplicity we restrict to the case in which $W=I\times Y$ is a product and so $Y=Y'$. Let $(\overline W,\overline S)$ be the resulting blown-up cobordism. We determine when such a cobordism induces a morphism of $\cS$-complexes. The discussion in Subsection \ref{subsec:immersedcrossing} and the index formula \eqref{eq:index} implies that the index of a minimal reducible $A_L$ on $(W,S)$ is given by
\[
  \ind(A_L) = \sigma(Y,K)-\sigma(Y',K')-2g(S)-1.
\]
Moreover we recall from \eqref{eq:etacrossing} that $\eta(W,S)\neq 0$ under the assumption that $T^4\neq 1$ in $\sS$. If $T^4=1$ then $\eta(W,S)\neq 0$ if and only if the number of positive double points $s_+$ is zero. From this we obtain:

\begin{cor}\label{neg-pair}
	The pair $(W,S)$ is negative definite over $\sS$ if and only if
	\begin{equation} \label{genus-cond}
	   \sigma(Y,K)-\sigma(Y,K')=2g(S)
	\end{equation}
	and either $T^4\neq 1$ as elements of $\sS$, or $s_+=0$.
\end{cor}

This observation, together with the cobordism relations of Subsection \ref{subsec:cobrels}, leads to:

\begin{theorem}\label{knot-inv-twisted}
	Suppose $T^2-T^{-2}$ is invertible in $\sS$. For any null-homotopic knot $K\subset Y$ in an integer homology 3-sphere, we have an $\cS$-chain homotopy equivalence
	\begin{equation}
	\widetilde C(Y,K,\Delta_{\sS})\simeq\widetilde C(Y,U_1,\Delta_{\sS})\otimes \widetilde C(S^3,{-{\tfrac{\sigma(Y,K)}{2}}}T_{2,3},\Delta_{\sS})\label{eq:schaineq}
\end{equation}
	In particular, if $K$ is a knot in the 3-sphere, we have
	\[
\widetilde C(K;\Delta_{\sS})\simeq \widetilde C(S^3,{-{\tfrac{\sigma(K)}{2}}}T_{2,3};\Delta_{\sS}).
\]
\end{theorem}

\begin{proof}
	Suppose $K$ and $K'$ are two homotopic knots in an integer homology sphere $Y$ with the same signature. A generic smooth homotopy from $K$ to $K'$ determines an immersed cylinder 
	$S:K\to K'$ in $W=I\times Y$ with double points. Suppose $S':K'\to K$ is obtained from $S$ by flipping and then changing the orientation. 
	According to Corollary \ref{neg-pair}, $S$ and $S'$ are negative definite over $\sS$.
	By applying a sequence of finger moves, 
	the composed cobordisms $S'\circ S:K\to K$ and $S\circ S':K'\to K'$ can be turned into the product cobordisms. 
	Therefore, Proposition \ref{homotopy-moves} implies that $S$ and $S'$ induce $\cS$-chain homotopy equivalences 
	if  $T^2-T^{-2}$ is invertible. 
	An arbitrary null-homotopic knot $K\subset Y$ is homotopic to and has the same signature as $-\frac{\sigma(Y,K)}{2}T_{2,3}\subset Y$. An application of Theorem \ref{thm:connectedsum} completes the proof.
\end{proof}

\begin{remark}
	One of the implications of Theorem \ref{knot-inv-twisted} is that the Euler characteristic of $I_*(K)$ for a knot $K\subset S^3$ is $-\sigma(K)/2$. This approach to compute the Euler characteristic of $I_*(K)$
	 does not appeal to \cite{herald}. $\diamd$
\end{remark}

The cobordisms used in the above proof were rather special. Indeed, Corollary \ref{neg-pair} indicates, for example, that most crossing change cobordisms are not negative definite. However, keeping the assumption that $T^4\neq 1 \in \sS$, given any cobordism we may ``suspend'' it so that the result is negative definite over $\sS$. 

First, for any knot $K\subset Y$ we define the suspension $\Sigma^i K\subset Y$ to be the knot $(Y,K)\# (S^3, i T_{2,3})$. When $i$ is negative it should be understood that we are connect summing $K$ with $i$ copies of the left-handed trefoil. Next, let $(W,S):(Y,K)\to (Y,K')$ be any pair as above where $W$ is a product and $S$ is a connected, oriented surface which is possibly immersed with normal crossings. Let $s_+$ and $s_-$ be the number of positive and negative double points of $S$, respectively. For any pair of positive integers $i,j$ define the suspension
\[
	(W,\Sigma^{i,j} S):(Y,\Sigma^iK) \to (Y,\Sigma^jK)
\] 
to be the cobordism obtained by the composition of three parts: (i) the crossing change cobordism $\Sigma^iK\to K$ obtained by changing a crossing in a standard diagram in each right or left handed trefoil summand; (ii) the original cobordism $S\subset W$; and (iii) a crossing change cobordism $K\to \Sigma^j K$ which is the reverse of the kind in (i).

\begin{lemma}
	Assume $T^4\neq 1$ in $\sS$. If $i-j = \frac{1}{2}(\sigma(Y,K)-\sigma(Y,K'))-g(S)$ then the suspension $(W,\Sigma^{i,j}S)$ is negative definite over $\sS$.
\end{lemma}

This follows from Corollary \ref{neg-pair} by direct computation. In particular, we can always choose a suspension so that the result is negative definite over $\sS$. We note that the number of positive and negative double points of the suspended surface $\Sigma^{i,j} S$ are given by
\[
  s_\pm(\Sigma^{i,j} S)=s_\pm+\max(\mp i,0)+\max(\pm j,0).
\]

\subsection{Further remarks on the invariants for $(Y,U_1)$}

The equivalence in \eqref{eq:schaineq} leads us to ask for the $\cS$-chain homotopy type of the complex $\widetilde C_\ast(Y,U_1;\Delta_\sS)$. It is natural to expect this complex to be related to Floer's instanton chain complex for the integer homology 3-sphere $Y$, and this is validated in part by Proposition \ref{prop:hinvunknot}. Here we partially explain this relationship, outlining some ideas without proof. The content here is not used elsewhere in the paper.

For simplicity we restrict to $\Q$ coefficients throughout this section, although everything works for a local system induced by an algebra over $R[U^{\pm 1}]$ so long as $\frac{1}{2}\in R$. To an integer homology 3-sphere $Y$ with basepoint $y$ there is a $\Z/8$-graded chain complex which is an $\cS$-complex exept for a grading difference:
\begin{equation}
  \widetilde C_*(Y,y)=C_*(Y) \oplus C_{*-3}(Y) \oplus \Q, \qquad   \widetilde d =\left[
  \begin{array}{ccc}
  d&0&0\\
  U&-d&D_2\\
  D_1&0&0
  \end{array}\right].\label{eq:socomplex}
\end{equation}
See \cite[Chapter 7]{donaldson-book} for the construction of this complex. (The normalization of the map $U$ above is the same as the one in \cite{AD:CS-Th}.) Here $(C_\ast(Y),d)$ is Floer's $\Z/8$-graded instanton chain complex generated by irreducibles. To distinguish our terminology, we call any such chain complex as above an $\cSO$-complex. In terms of our definition of an $\cS$-complex $(\widetilde C_\ast, \widetilde d , \chi)$ in Subsection \ref{subsec:scomplexes}, an $\cSO$-complex only differs in that $\chi$ has degree $3$ instead of $1$. Morphisms and homotopies are defined as before. The $\cSO$-chain homotopy type of $(\widetilde C_*(Y,y),\widetilde d)$ is a {\it natural} invariant of $(Y,y)$.

Roughly speaking, $\widetilde C_\ast(Y,y)$ is the Morse-Bott complex associated to the Chern--Simons functional on the space of framed $SU(2)$ connections on $Y$. There is an $SO(3)=SU(2)/\pm 1$ action on this space. After a perturbation we have a finite set of critical orbits of type $SO(3)$, and an orbit which is a point, corresponding to the trivial connection. 

The $\cS$-complex $\widetilde C_\ast(Y,K)=\widetilde C_\ast(Y,K,p)$ for a knot introduced in \cite{DS} and studied in this paper has a similar interpretation. It also requires a choice of basepoint $p$, where $p$ is on the knot $K$, although this choice is usually supressed. Then $\widetilde C_\ast(Y,K,p)$ is morally the Morse-Bott complex for the Chern--Simons functional on the space of singular $SU(2)$ connections on $(Y,K)$ which are framed at $p$. Recall that at $p\in K$ the bundle splits as $L\oplus L^\ast$, and a framing here is an identification of $L_p$ with $\C$. This space has an $S^1$-action. After a perturbation we have a finite set of critical orbits of type $S^1$, and an orbit which is a point, corresponding to the reducible singular connection $\theta$.

There is a third construction which is a hybrid of the above two. Choose a basepoint $y\in Y\setminus K$. The space of singular $SU(2)$ connections on $(Y,K)$ framed at $y$ now has an $SO(3)$ action; there is no preferred reduction at $y$. After a perturbation the Chern--Simons functional has a finite set of critical orbits of type $SO(3)$ and an orbit of type $S^2=SO(3)/S^1$ generated by the reducible $\theta$. We may form a $\Z/4$-graded complex of the form
\begin{equation}\label{eq:newsocomplex}
	\widetilde C_\ast(Y,K,y) = C_\ast \oplus C_{\ast-3} \oplus \Q_{(0)} \oplus \Q_{(2)},\quad   \widetilde d =\left[
  \begin{array}{cccc}
  d&0&0&0\\
  U&-d&V_2&\delta_2\\
  \delta_1&0&0&0\\
  V_1&0&0&0
  \end{array}\right]
\end{equation}
wher $C_\ast=C_\ast(Y,K)$. Here $\Q_{(0)}\oplus \Q_{(2)}$ corresponds to the homology of the reducible 2-sphere orbit. The maps $d$, $\delta_1$, $\delta_2$ are the usual types of maps which count $0$-dimensional moduli spaces of unparametrized instantons. The map $U$ is defined as for $\widetilde C_\ast(Y,y)$, while $V_1$ and $V_2$ also use holonomy, but involve the reducible $S^2$ orbit. For more on these kinds of constructions (in a more general context) see e.g. \cite[Section 7]{miller}.

The complex $\widetilde C_\ast(Y,K,y)$ is not strictly what we have been calling an $\cSO$-complex, but it does have a similar structure in so far as there is a degree $3$ endomorphism $\chi$ such that $\chi^2=0$ and $\widetilde d \chi + \chi  \widetilde d=0$: in the decomposition \eqref{eq:newsocomplex}, $\chi$ maps $C_\ast$ to $C_{\ast-3}$ identically and otherwise is zero. Let us use the following terminology: an $\cSO$-complex of the form \eqref{eq:socomplex} is {\em type (i)}, while one of the form \eqref{eq:newsocomplex} is {\em type (ii)}.

We can then form the tensor products of $\cSO$-complexes of either type: given $(\widetilde C_\ast, \widetilde d, \chi)$ and $(\widetilde C'_\ast, \widetilde d', \chi')$ we proceed as in \eqref{eq:tensorprod}. A straightforward calculation shows that the tensor product of a type (i) complex with a type (ii) complex is type (ii).

We can now formulate a connected sum formula, as follows. Suppose that $(Y,y)$ and $(Y',y')$ are based integer homology 3-spheres and $K'\subset Y'$ is a knot disjoint from $y'$. Form the connected sum $(Y\#Y', K',y_\#)$ so that $y_\#$ is disjoint from $K'$. Then there is an $\cSO$-chain homotopy equivalence of type (ii) $\cSO$-complexes:
\begin{equation}\label{eq:neqconnectedsum}
	  \widetilde C_*(Y\#Y',K',y_\#)\simeq  \widetilde C_*(Y,y)\otimes \widetilde C_*(Y',K',y').
\end{equation}
A proof of \eqref{eq:neqconnectedsum} should follow by adapting the ideas in the proof of \cite[Theorem 6.1]{DS}.

Let $(Y',K')=(S^3,U_1)$ with $y'$ a basepoint away from $U_1$. The complex $\widetilde C_\ast(S^3,U_1,y')$ is very simple: it has $C_\ast=C_{\ast-3}=0$ in \eqref{eq:newsocomplex}, as there is only the critical $S^2$ orbit of the reducible. Then \eqref{eq:neqconnectedsum} computes the type (ii) complex $C_\ast(Y,U_1,y)$ to be homotopy equivalent to $C_\ast\oplus C_{\ast-3}\oplus \Q_{(0)}\oplus \Q_{(2)}$ where $C_\ast=C_\ast(Y)\oplus C_{\ast-2}(Y)$ and $\widetilde d$ is given by
\[
	\left[  \begin{array}{cc|cc|cc} d & 0 & 0 & 0 &  0 & 0\\  0 & d & 0 & 0 & 0 & 0 \\ \midrule U & 0 &-d & 0 & D_2 & 0 \\ 0 & U & 0 & -d & 0 & D_2 \\ \midrule D_1 & 0 & 0 & 0 & 0 & 0\\  0 & D_1 & 0 & 0 & 0 & 0\end{array}  \right]
\]
This then immediately implies that the kernels of $h_\Q(\cdot,U_1)$ and $h_\Q(\cdot)$ as homomorphisms $\Theta_\Z^3\to \Z$ are equal. Furthermore, it implies that we have a $\Z/4$-graded isomorphism of irreducible theories: $I_\ast(Y,U_1)\cong \oplus^2 I_\ast(Y)$.

Further investigation should relate the type (ii) $\cSO$-complex $\widetilde C_\ast(Y,K,y)$, where $y\not\in K$, to the $\cS$-complex $\widetilde C_\ast(Y,K,p)$, where $p\in K$. It is likely that the relationship between the degree $4$ map $U$ and the degree 2 map $v$ discussed in \cite[Section 4.1]{Kr:obs} is relevant to this investigation.



\section{Singular Fr\o yshov invariants and the unoriented 4-ball genus}\label{sec:unoriented}

In this section we study the invariant $h_\sS(K)$, the singular instanton Fr\o yshov invariant defined over the coefficient ring $\sS$ where $T^4=1$. We show that $h_\sS(K)$ provides a lower bound for the unoriented 4-ball genus of $K$, also called the 4-dimensional crosscap number of $K$. We also compute $h_\sS(K)$ for torus knots and quasi-alternating knots.

\subsection{Singular instantons and non-orientable surfaces}

The goal of this section is to extend the construction of maps associated to knot cobordisms to the non-orientable case. We provide this extension for a limited family of cobordisms $(W,S):(Y,K)\to (Y,K')$ where $W$ is the product cobordism $I\times Y$ which is convenient for our applications. Another restriction, which is more essential, is the choice of the coefficient ring. In this section, we work over $\Z$. The discussion can be adapted to enriched complexes and one can work over any integral domain algebra over the ring $\Z[U^{\pm 1}]$. However, non-orientable cobordisms do not give rise to morphisms between $\cS$-complexes over the ring $\Z[U^{\pm 1},T^{\pm 1}]$ as in the case of orientable cobordisms in Section \ref{sec:prelims}.

Suppose $K$, $K'$ are knots in an integer homology sphere $Y$, and $S:K\to K'$ is a connected non-orientable cobordism embedded in $W=I\times Y$. Assuming $S$ satisfies a certain topological constraint, there is a morphism of $\cS$-complexes $\widetilde C_*(S):\widetilde C_*(K)\to \widetilde C_*(K')$ where $\widetilde C_*(K)$, $\widetilde C_*(K')$ are shorthand for the $\cS$-complexes of $(Y,K)$ and $(Y,K')$. The morphism $\widetilde C_*(S)$ is defined in terms of singular connections on $(W,S)$. However, it will be convenient to describe the situation after passing to double branched covers.

There is a unique element of $H^1(W\setminus S;\Z/2)$ which does not extend to $W$. This gives rise to $\widetilde W$, the double cover of $W$ branched along $S$. We have a covering involution $\tau:\widetilde W\to \widetilde W$ with fixed point set the lift $\widetilde S$ of $S$ to the branched cover. Fix the trivial $SO(3)$-bundle $\widetilde E_0:=\underline\R^3$ on $\widetilde W$ with the lift of $\tau$ to the involution $\ft_0:\widetilde E_0 \to \widetilde E_0$ given as
\begin{equation}\label{trivial-singular-bundle}
  ((t_1,t_2,t_3),x)\in \R^3\times \widetilde W\xrightarrow{\ft_0}((-t_1,-t_2,t_3),\tau(x)).
\end{equation}
A singular connection on $(W,S)$ can be lifted to an equivariant connection on the double branched cover. We proceed to (re)define the notion of a singular connection in terms of this latter description.

A singular connection for the pair $(W,S)$ is a connection $\widetilde A$ on an $SO(3)$-bundle $\widetilde E$ over $\widetilde W$ together with an involution $\ft:\widetilde E\to \widetilde E$ lifting the involution $\tau$ such that $\ft^*\widetilde A=\widetilde A$. The map $\ft$ induces an automorphism of order two on $\widetilde E|_{\widetilde S}$ and we require that this automorphism to be non-trivial. In particular, the involution $\ft$ induces a reduction of $\widetilde E$ into a direct sum $Q\oplus l$ over $\widetilde S$ where $l$ is a real line bundle and $Q$ is an $O(2)$-bundle with orientation bundle $l$. We require that $w_2(\widetilde E)=0$ and $w_1(l)=0$. That is to say, $l$ is isomorphic to the trivial line bundle and $\widetilde E$ can be lifted into an $SU(2)$-bundle. The final part of the data of a singular connection is an automorphism $\sigma:\widetilde E|_{\widetilde W\setminus \widetilde S}\to \widetilde E_0|_{\widetilde W\setminus \widetilde S}$ such that $\sigma\ft=\ft_0\sigma$. 

Two singular connections $A_1=(\widetilde A_1,\widetilde E_1,\ft_1,\sigma_1)$ and $A_2=(\widetilde A_2,\widetilde E_2,\ft_2,\sigma_2)$ are isomorphic if there is an isomorphism $\widetilde u:\widetilde E_1\to \widetilde E_2$ of $SO(3)$-bundles such that $\widetilde u^*\widetilde A_2=\widetilde A_1$, $\widetilde u\ft_1=\ft_2\widetilde u$, and $\sigma_2\widetilde u\sigma_1^{-1}$ lifts to a determinant $1$ automorphism of $\widetilde E_0|_{\widetilde W\setminus \widetilde S}$.

Given singular flat connections $\alpha$, $\alpha'$ on $(Y,K)$, $(Y,K')$ we may define moduli space of singular ASD connections $M(W,S;\alpha,\alpha')$. The pull-back of $\alpha$  to $\widetilde Y$, the double cover of $Y$ branched along $K$, defines a connection $\widetilde \alpha$ which is invariant with respect to the covering involution. Similarly, we have $\widetilde \alpha'$, a connection on $\widetilde Y'$. Then $M(W,S;\alpha,\alpha')$ is the space of isomorphism classes of finite energy singular connections $(\widetilde A,\widetilde E,\ft,\sigma)$ such that $\widetilde A$ is a singular connection on $\widetilde W^+$, the result of attaching cylindrical ends to $\widetilde W$, asymptotic to $\widetilde \alpha$ and $\widetilde \alpha'$, satisfying the ASD equation. 

As usual, we might have to perturb the Chern-Simons functionals for $(Y,K)$, $(Y,K')$ to form the $\cS$-complexes $\widetilde C_*(K)$, $\widetilde C_*(K')$. Similar arguments as in the orientable case show that there are equivariant perturbations of the ASD equation on $\widetilde W$ which are compatible with perturbations of the Chern-Simons functionals such that the elements of the moduli spaces $M(W,S;\alpha,\alpha')$ are regular. In a neighborhood of an element of $M(W,S;\alpha,\alpha')$ represented by $A=(\widetilde A,\widetilde E,\ft,\sigma)$ the moduli space is a smooth manifold whose dimension, denoted by $\ind(A)$, is the difference between the dimensions of $\ft$-invariant elements in the kernel and the cokernel of the ASD operator associated to the connection $\widetilde A$.  Since $A$ is regular, the contribution from the cokernel to $\ind(A)$ is trivial. As usual this equivariant index extends to an arbitrary singular connection $\widetilde A$ where there might be some non-trivial contribution from the cokernel.

\begin{lemma}\label{index-reducible-non-orientable}
	Let $A=(\widetilde A,\widetilde E,\ft,\sigma)$ be a singular connection on $(W,S)$ asymptotic to the reducibles on 
	$(Y,K)$, $(Y,K')$.
	Then the index of the associated ASD operator is given as
	\begin{equation}\label{index-reducible-non-orientable-formula}
		\ind(A)=8\kappa(A)+ \chi(S) + \frac{1}{2}S\cdot S + \sigma(K) - \sigma(K') -1.
	\end{equation}
	where the topological energy is given by
	\[
	  \kappa(A)=\frac{1}{16\pi^2}\int_{\widetilde W^+} {\rm Tr}(F_{\widetilde A}\wedge F_{\widetilde A}).
	\]
\end{lemma}
\begin{proof}
	A similar result in the case that $S$ is an orientable surface is proved in \cite{DS} as Lemma 2.26. The proof there reduces the index computation to the case that $S$ 
	is a closed surface embedded into a smooth 4-manifold. The proof can be easily adapted to the non-orientable case using \cite[Lemma 2.11]{KM:unknot}. 
\end{proof}

There is a unique reducible in the moduli spaces $M(W,S;\alpha,\alpha')$. A singular connection $A=(\widetilde A,\widetilde E,\ft,\sigma)$ has a non-trivial stabilizer if there is an $S^1$-reduction of $\widetilde E$ as $\widetilde L\oplus \underline \R$ such that the connection $\widetilde A$ respects this reduction and the involution $\ft$ acts as
\[
  ((v,t),x)\xrightarrow{\ft} ((\ft'(v),t),\tau(x))
\]
where $\ft'$ is an involution on the $S^1$-bundle $\widetilde L$ which acts as multiplication by $-1$ on $\widetilde L|_{\widetilde S}$. In particular, the quotient of $\widetilde L$ by $-\ft'$ induces an $S^1$ bundle on $W$ which is necessarily isomorphic to the trivial bundle. Therefore, $A$ is isomorphic to $(\widetilde A,\widetilde E_0,\ft_0,\sigma_0=\text{id})$ where $(\widetilde E_0,\ft_0)$ is given in \eqref{trivial-singular-bundle}. Moreover, if the connection $\widetilde A$ on $\widetilde E_0$ is ASD, then $A$ is isomorphic to $\Theta:=(\widetilde \Theta,\widetilde E_0,\ft_0,\sigma_0)$ with $\widetilde \Theta$ being the trivial connection on $\widetilde E_0$. Lemma \ref{index-reducible-non-orientable} asserts that the index of this ASD connection is equal to 
\[
  \ind(\Theta)=\chi(S) + \frac{1}{2}S\cdot S + \sigma(K) - \sigma(K') -1.
\]
In the case that $\ind(\Theta)=-1$, we can proceed as in the orientable case to define the morphism $\widetilde C_*(S)$ as in \cite{DS}. More generally, if $\ind(\Theta)\geq -1$, there is a height $\frac{\ind(\Theta)+1}{2}$ morphism $\widetilde C_*(S)$ induced by $S$.

\begin{remark}
	Moduli spaces of singular instantons for non-orientable surfaces are studied in more detail in \cite{KM:unknot}. 
	Following \cite[Subsection 4.2]{KM:unknot}, a $2$-dimensional submanifold $\omega$ in $W$ whose boundary $\partial \omega$ is contained in $S$ determines singular bundle data on $(W,S)$.
	(The surface $S$ and the interior of $\omega$ are allowed to have finitely many transverse intersection points.) In particular, $\partial \omega$ induces an element of $H_1(S;\Z/2)$. 
	In the above discussion, we have specialized to a particular choice of singular bundle data on $(W,S)$, where the homology class determined by $\partial \omega$ is dual to $w_1(TS)$.
	The importance of this particular singular bundle data is the presence of the unique reducible described above. To let this point stand out, we redefined singular connections for a non-orientable surface cobordism $S$ of knots in 
	terms of branched double covers.
	 $\diamd$ 
\end{remark}

\subsection{Negative definite cobordisms and non-orientable surfaces}

In this subsection we study how non-orientable surface cobordisms provide new information about the singular Fr\o yshov invariant $h_\sS$. Initially we start with the case that $\sS=\Z$. 
\begin{theorem}\label{crosscap-number-ineq}
	Suppose $S:K\to K'$ is a connected surface cobordism smoothly embedded into $W=I\times S^3$, not necessarily orientable.	Then:
	\begin{equation}\label{non-orientable-ineq}
		h_\Z(K')-h_\Z(K)\geq \frac{1}{2}\chi(S) + \frac{1}{4}S\cdot S + \frac{1}{2}(\sigma(K) - \sigma(K')).
	\end{equation}
\end{theorem}
\begin{proof}
	If $S$ is orientable, this follows from Theorem \ref{thm:hineq}. For non-orientable surfaces it can be verified by following a similar argument. In the case that the right hand side of \eqref{non-orientable-ineq}, denoted by $i$, is non-negative, the proof follows from the existence of the height $i$ morphism 
	$\widetilde C_*(S):\widetilde C_*(K)\to \widetilde C_*(K')$. If $i$ is negative, we can use the suspension trick in the proof of Theorem \ref{thm:hineq} which uses the torus 
	knot $T_{3,4}$. 
\end{proof}
As an immediate corollary of Theorem \ref{crosscap-number-ineq} we have:
\begin{cor}
	Suppose $S$ is an embedded surface in the $4$-ball with boundary $K$. Then:
	\begin{equation}\label{non-orientable-ineq-1}
		\left\vert h_\Z(K)+\frac{1}{2}\sigma(K)-\frac{1}{4}S\cdot S \right\vert \leq -\frac{\chi(S)}{2}.
	\end{equation}
	In particular, $h_\Z(K)$ gives a lower bound for the smooth $4$d crosscap number of $K$:
	\begin{equation}\label{non-orientable-ineq-2}
		\left\vert h_\Z(K)\right\vert \leq \gamma_4(K).
	\end{equation}	
\end{cor}
\begin{proof}
	By removing a ball from $S\subset B^4$, we obtain a cobordism from the unknot to $K$ embedded in $I\times S^3$. Applying the inequality in \eqref{non-orientable-ineq}
	to this cobordism and its reverse gives rise to \eqref{non-orientable-ineq-1} in the case that $S$ is non-orientable. The orientable case is a consequence of 
	Theorem \ref{thm:hineq}. As in \cite{OSS:upsilon}, inequality \eqref{non-orientable-ineq-2} can be obtained by combining \eqref{non-orientable-ineq-1} with:
	\[
	  \left\vert \frac{1}{2}\sigma(K)-\frac{1}{4}S\cdot S\right\vert \leq -\frac{\chi(S)}{2}
	\]
	which is in turn based on the Gordon--Litherland formula \cite{GL:signature}. (A closely related computation goes into the proof of Lemma 
	\ref{index-reducible-non-orientable}.)
\end{proof}

\begin{prop}\label{prop:gammazquasialt}
	For any quasi-alternating knot $K$, the enriched $\cS$-complex $\widetilde C(K;\Z)$ is locally equivalent to $\widetilde C(U_1;\Z)$, the trivial $\cS$-complex 
	defined over $\Z$.
	In particular, $h_\Z(K)=0$ and $\Gamma_K^\Z=\Gamma_{U_1}^\Z$ for any quasi-alternating knot $K$.
\end{prop}

\begin{proof}
	If $K$ is a quasi-alternating knot, then there exist surface cobordisms $S:U_1\to K$ and $S':K\to U_1$ (not necessarily orientable) whose double branched covers over $[0,1]\times S^3$ are negative definite (see \cite{os-branched}), which is equivalent to $\text{ind}(\Theta)=-1$ in Lemma \ref{index-reducible-non-orientable}. These cobordisms yield morphisms of $\cS$-complexes $\widetilde C_\ast(S):\widetilde C_\ast(U_1) \to \widetilde C_\ast(K)$ and $\widetilde C_\ast(S'):\widetilde C_\ast(K) \to \widetilde C_\ast(U_1)$ over $\Z$, implying the result.
\end{proof}

The following proposition provides a relation which can be used to compute $h_\Z$ for torus knots recursively. This inductive approach to compute $h_\Z(T_{p,q})$ is formally the same as the ones given in \cite{JV:unorientable,ballinger} for $\upsilon(T_{p,q})$ and $t(T_{p,q})$. 

\begin{prop}
	The invariant $h_\Z$ of torus knots $T_{p,q}$ and $T_{p,q+p}$ satisfies the identity
		\[
		  h_\Z(T_{p,p+q}) +\frac{1}{2}\sigma(T_{p,p+q}) = h_\Z(T_{p,q}) + \frac{1}{2}\sigma(T_{p,q}) - \left\lfloor p^2/4 \right\rfloor.
		\]
\end{prop}
\begin{proof}
	The proof uses the twist cobordism of Proposition \ref{eq:blowuptwistineq} and {\it pinch move} cobordisms between torus knots \cite{batson,JV:unorientable}. Given positive coprime integers $(p,q)$, there is another pair of positive coprime $(r,s)$ and a non-orientable cobordism $S_{p,q}^{r,s}:T_{p,q} \to T_{r,s}$ 
	such that $\chi(S_{p,q}^{r,s})=-1$, $r$ has the same parity as $p$, $s$ has the same parity as $q$, 
	$r+s \leq p+q$ and this inequality is strict unless $(p,q)=(1,1)$. Although the values of $(r,s)$ can be 
	determined by $(p,q)$, we only point out that the pair associated to $(p,p+q)$ is equal to $(r,r+s)$. Moreover, there is the following relation 
	among the self-intersection numbers \cite{ballinger}:
	\begin{equation}\label{selfintersectionrelation}
	  S_{p,p+q}^{r,r+s}\cdot S_{p,p+q}^{r,r+s}=S_{p,q}^{r,s}\cdot S_{p,q}^{r,s}+p^2-r^2.
	\end{equation}
	By induction on $p+q$, we show 
	\begin{align*}
	  h_\Z(T_{p,p+q}) +\frac{1}{2}\sigma(T_{p,p+q}) &= h_\Z(T_{p,q}) + \frac{1}{2}\sigma(T_{p,q}) - \left\lfloor p^2/4\right\rfloor\\
	 &= h_\Z(T_{r,r+s}) +\frac{1}{2}\sigma(T_{r,r+s})+\frac{1}{2}-\frac{1}{4}S_{p,p+q}^{r,r+s}\cdot S_{p,p+q}^{r,r+s}.
	\end{align*}
	
	Firstly note that Proposition \ref{eq:blowuptwistineq} implies that 
	\begin{equation}\label{ineq-1-twist-torus}
	  h_\Z(T_{p,p+q}) +\frac{1}{2}\sigma(T_{p,p+q})\geq h_\Z(T_{p,q}) + \frac{1}{2}\sigma(T_{p,q}) - \left\lfloor p^2/4 \right\rfloor.
	\end{equation}
	Theorem \ref{crosscap-number-ineq} applied to the cobordism $S_{p,p+q}^{r,r+s}$ and \eqref{selfintersectionrelation} imply
	\begin{equation}\label{ineq-2-twist-torus}
	  h_\Z(T_{p,p+q}) +\frac{1}{2}\sigma(T_{p,p+q})\leq h_\Z(T_{r,r+s}) +\frac{1}{2}\sigma(T_{r,r+s})+\frac{1}{2} -\frac{1}{4}S_{p,q}^{r,s}\cdot S_{p,q}^{r,s} - \frac{p^2-r^2}{4}.
	\end{equation}
	Combining the above inequalities and the induction assumption for $(r,s)$ implies that
	\[
	  h_\Z(T_{p,q}) + \frac{1}{2}\sigma(T_{p,q})\leq h_\Z(T_{r,s}) + \frac{1}{2}\sigma(T_{r,s})+\frac{1}{2}-\frac{1}{4}S_{p,q}^{r,s}\cdot S_{p,q}^{r,s}. 
	\]
	Another application of the induction assumption asserts that the above inequality is sharp. Therefore, the inequalities in \eqref{ineq-1-twist-torus} and \eqref{ineq-2-twist-torus}
	can be improved to identities.
\end{proof}

Having established some computations for $h_\Z(K)$, we point out the following analogue of Theorem \ref{thm:existrep2}, which follows from \cite[Theorem, 1.17]{DS}.

\begin{theorem}\label{thm:existrep3}
	Suppose $(W,S):(Y,K)\to (Y',K')$ is a homology concordance and that:
\begin{equation}
	h_\Z(Y,K) + 4h(Y) \neq 0. 
\end{equation}
Then there exists a traceless representation $\pi_1(W\setminus S)\to SU(2)$ which extends non-abelian traceless representations of $\pi_1(Y\setminus K)$ and $\pi_1(Y'\setminus K')$.
\end{theorem}

\noindent This provides different examples than does Theorem \ref{thm:existrep2}, simply because $h_\Z$ and $-\sigma/2$ are linearly independent homomorphisms from the homology concordance group to the integers, as follows from the above computations. For example, for $K=T_{3,4}\#-T_{2,7}$, we have $h_\Z(K)=1$ while $-\sigma(K)/2=0$. In this case, for a homology concordance $(W,S)$ from $K$ to itself, Theorem \ref{thm:existrep3} implies the existence of a non-abelian traceless representation of $\pi_1(W\setminus S)$ into $SU(2)$, while Theorem \ref{thm:existrep2} gives no information.

We may extend the results of this section for slightly more general coefficient rings. For example, the results of this section (definition of non-orientable cobordism maps and the implications for $h_\sS$) extend word for word to the case that $\sS$ is an abelian group. The same claims hold when $\sS$ is an integral domain algebra over $\Z[U^{\pm 1}, T^{\pm 1}]/(T^4-1)$. In this case, we may regard $\sS$ as an algebra over $\Z[U^{\pm 1}, T^{\pm 1}]$ in two different ways by requiring $T\in \Z[U^{\pm 1}, T^{\pm 1}]$ acts as $1$ or $T\in \Z[U^{\pm 1}, T^{\pm 1}]/(T^4-1)$, and Lemma \ref{lift-morphisms} implies that we have similar results for $h_\sS$ in both cases. In the case that $T$ acts as $1$, the proofs of this section adapt to give the same results for the coefficient ring $\sS$.


\addcontentsline{toc}{section}{References}

\bibliography{references}
\bibliographystyle{alpha.bst}
\Addresses
\end{document}